\documentclass[reqno]{amsart}
\usepackage[top=1.1in, bottom=1in, left=1in, right=1in]{geometry}
\usepackage{amsfonts}
\usepackage{amssymb}
\usepackage{amsthm}
\usepackage{amsmath}
\usepackage{mathrsfs}
\usepackage{color}
\usepackage{enumerate}
\usepackage[numbers,sort&compress]{natbib}
\usepackage{pdfsync}
\usepackage{esint}
\usepackage{graphicx}
\usepackage{float}
\usepackage{caption}
\usepackage{subfigure}

\allowdisplaybreaks

\usepackage{tikz} 
\usetikzlibrary{calc}
\usetikzlibrary{intersections}
\usetikzlibrary{patterns}

\usepackage[colorlinks,
            linkcolor=black,
            anchorcolor=blue,
            citecolor=blue,
            ]{hyperref}

\makeatother

\numberwithin{equation}{section}

\newtheorem{theorem}{Theorem}[section]
\newtheorem{definition}[theorem]{Definition}

\newtheorem{lemma}[theorem]{Lemma}
\newtheorem{remark}[theorem]{Remark}

\newtheorem{proposition}[theorem]{Proposition}
\newtheorem{corollary}[theorem]{Corollary}
\numberwithin{equation}{section}

\newcommand*{\Id}{\ensuremath{\mathrm{Id}}}
\newcommand*{\supp}{\ensuremath{\mathrm{supp\,}}}

\newcommand{\aint}{{\fint}}

\newcommand{\T}{{\mathbb{T}}}

\newcommand{\rr}{R}
\newcommand{\rrq}{R_q}
\newcommand{\rl}{R_\ell}

\newcommand{\p}{\partial}
\renewcommand{\P}{\mathbb{P}}

\renewcommand{\div}{{\mathrm{div}}}

\renewcommand{\rq}{{\rho_q}}
\newcommand{\rqq}{{\rho_{q+1}}}
\newcommand{\m}{{m_q}}

\renewcommand{\d}{{\rm d}}

\newcommand{\g}{g_{(k)}}
\newcommand{\norm}[1]{\lVert#1\rVert}

\newcommand{\la}{\lambda_q}
\newcommand{\laq}{\lambda_{q+1}}

\newcommand{\dqq}{\delta_{q+1}}
\newcommand{\va}{\varepsilon}
\newcommand{\wo}{w_{q+1}^{(o)}}
\newcommand{\thq}{z_{q+1}}
\newcommand{\mq}{m_{q+1}}
\renewcommand{\u}{u_q}

\newcommand{\rqe}{\rho_{q+1}^{-1}}

\newcommand{\R}{{\mathbb R}}
\newcommand{\lbb}{\overline{\lambda}}

\def\a{{\alpha}}
\def\vf{{\varphi}}
\def\lbb{\lambda}
\def\wt{\widetilde}
\def\9{{\infty}}
\def\ve{{\varepsilon}}
\def\na{{\nabla}}
\def\bbr{{\mathbb{R}}}

\def\({\left(}
\def\){\right)}				

\begin{document}
	
\title[] {Non-uniqueness for the hypo-viscous compressible Navier-Stokes equations}

\author{Yachun Li}
\address{School of Mathematical Sciences, CMA-Shanghai, MOE-LSC, and SHL-MAC,  Shanghai Jiao Tong University, China.}
\email[Yachun Li]{ycli@sjtu.edu.cn}
\thanks{}

\author{Peng Qu}
\address{School of Mathematical Sciences $\&$ Shanghai Key Laboratory for Contemporary Applied Mathematics, Fudan University, China.}
\email[Peng Qu]{pqu@fudan.edu.cn}
\thanks{}

\author{Zirong Zeng}
\address{School of Mathematical Sciences, Shanghai Jiao Tong University, China.}
\email[Zirong Zeng]{beckzzr@sjtu.edu.cn}
\thanks{}

\author{Deng Zhang}
\address{School of Mathematical Sciences, CMA-Shanghai, Shanghai Jiao Tong University, China.}
\email[Deng Zhang]{dzhang@sjtu.edu.cn}
\thanks{}



\begin{abstract}
We study the Cauchy problem for the isentropic hypo-viscous compressible Navier-Stokes equations (CNS)
under general pressure laws in all dimensions $d\geq 2$.
For all hypo-viscosities $(-\Delta)^\alpha$ with $\alpha\in (0,1)$,
we prove that there exist infinitely many weak solutions with the same initial data.
This provides the first non-uniqueness result of weak solutions
to viscous compressible fluid.
Our proof features new constructions of building blocks 
for both the density and momentum, 
which respect the compressible structure. 
It also applies to the compressible Euler equations and
the hypo-viscous incompressible Navier-Stokes equations (INS).
In particular, 
in view of the Lady\v{z}enskaja-Prodi-Serrin criteria,
the obtained non-uniqueness of $L^2_tC_x$ weak solutions to the hypo-viscous INS is sharp, 
and reveals that $\alpha =1$ is the sharp viscosity threshold
for the well-posedness in $L^2_tC_x$.
Furthermore, we prove that the H\"older continuous weak solutions
to the compressible Euler equations
may be obtained as a strong vanishing viscosity limit of a sequence of
weak solutions to the hypo-viscous CNS.
\end{abstract}

\maketitle

{
\tableofcontents
}

\section{Introduction and main results}    \label{Sec-Intro}

\subsection{Background} \label{Subsec-intro}

We consider the
$d$-dimensional isentropic hypo-viscous compressible Navier-Stokes equations (CNS for short) on the torus $\T^d:=[-\pi,\pi]^d$,
\begin{equation}\label{equa-NS}
\left\{\aligned
& \p_t \rho +\div (\rho u) = 0,\\
&\p_t(\rho u)+\mu(-\Delta)^{\a}u-(\mu+\nu)\nabla \div u +\div(\rho u\otimes u) + \nabla P(\rho)=0.
\endaligned
\right.
\end{equation}
where $d\geq 2$, $\rho:[0,T]\times \T^d\to (0,\infty)$
and $u: [0,T]\times \T^d \to \bbr^d$ are the mass density and velocity of fluid, respectively.
The coefficients $\mu$ and $\nu$ are constants,
$\mu$ is the shear viscosity coefficient,
$\nu+\frac{2}{d}\mu$ is the bulk viscosity coefficient satisfying the physical assumptions
\begin{align}
\mu>0,\quad \nu+\frac{2}{d}\mu\geq 0.
\end{align}
The hypo-viscosity $(-\Delta)^{\alpha}$, $\alpha\in(0,1)$, is defined
by the Fourier transform
$$
\mathcal{F}((-\Delta)^{\alpha}u)(\xi)=|\xi|^{2\alpha}\mathcal{F}(u)(\xi),\ \ \xi\in \mathbb{Z}^d,
$$
which arises in various models including, e.g.,
pure jump L\'evy process \cite{sato99}, radiation hydrodynamics \cite{rose89} and flows in porous media \cite{dqrv11,dqrv12}.

The pressure
$P(\rho)$ is a function of density.
It is usually assumed that $P(\rho)$ satisfies
the monotonicity assumptions  $P(\rho)>0,\ P'(\rho)>0$ and $P''(\rho)\geq 0$  for $\rho>0$ for  well-posedness problems.
For instance,
one typical example is the polytropic case $P(\rho)=A \rho^\gamma$,
where $A>0$ and $\gamma>1$.
In the present work,
the monotonicity condition is not necessary,
and we only assume that
\begin{align}\label{p-bound}
    P(\rho)\ \text{is}\ C^2-\text{regular with respect to}\ \rho.
\end{align}

We note that,
\eqref{equa-NS} is the classical compressible Navier-Stokes equations when $\alpha=1$,
and \eqref{equa-NS} reduces to the compressible Euler equations
when the viscosities vanish $\mu=\nu=0$:
\begin{equation}\label{equa-Euler}
	\left\{\aligned
	& \p_t \rho +\div (\rho u) = 0,\\
	&\p_t(\rho u) +\div(\rho u\otimes u) + \nabla P(\rho) =0.
	\endaligned
	\right.
\end{equation}
It would also be convenient to formulate \eqref{equa-NS} in terms of
the density and momentum $m:=\rho u$ as follows
\begin{align}\label{equa-mns}
\begin{cases}
	\p_t \rho +\div m = 0,\\
	\p_t m+\mu(-\Delta)^{\a} (\rho^{-1}m) -(\mu+\nu)\nabla \div (\rho^{-1}m)
      +\div(\rho^{-1}m\otimes m) + \nabla P(\rho) =0.
\end{cases}
\end{align}

The compressible Navier--Stokes equations \eqref{equa-NS}
with $\alpha = 1$
is one of the famed models of fluid dynamics
and has been extensively studied in literature.
In the 1D case,  when the initial data are away from vacuum,
the global well-posedness of classical solutions with large initial data
has been proved for initial-boundary value problem by Kazhikhov-Shelukhin  \cite{KS77} in 1977,
and for Cauchy problem by Kazhikhov \cite{K82} in 1982.
See also Kanel$'$ \cite{kanel} for the case of small initial data.
In contrast, the multi-dimensional case becomes much more complicated.
For the case without vacuum,
Serrin \cite{S59} and Nash \cite{N62} proved the uniqueness and local existence of 3D classical solutions in 1959 and 1962,
respectively.
See also Itaya \cite{IT71, IT76} and Tani \cite{tn77}.
For the global well-posedness, Matsumura-Nishida \cite{MN80}
proved the uniqueness and existence of classical solutions with small initial data.
The case with certain kinds of smallness on initial data was proved in \cite{CMZ10,HLX12,LX19},
and the finite time blow-up with large data was studied in  \cite{X98,XY13,MRRS22}.
Furthermore, Lions \cite{L93,L98} introduced the concept of renormalized solutions to establish
the global existence of finite energy weak solutions for $\gamma-$pressure law with $\gamma>9/5$,
concerning large initial data that may vanish.
It was then extended by
Feireisl-Novotn\'y-Petzeltov\'a \cite{FNP01} (see also \cite{FE04})  to $\gamma>3/2$,
and was further extended to $\gamma>1$ by Jiang-Zhang \cite{JZ03}  for 3D axisymmetric flow.
These existence results were also extended by Vasseur-Yu \cite{VY16}
and Li-Xin \cite{LX15}
to the case of density-dependent viscosity for $1<\gamma<3$.
Recently, Bresch-Jabin \cite{BJ18} proved the global existence for more general pressure laws
without any monotonicity assumption and even more general viscous stress tensors than
those previously covered by Lions-Feireisl's theory.

However,
besides the weak-strong uniqueness results (see, e.g., \cite{FJN12}),
the uniqueness of weak solutions is still a long standing open problem.
The aim of the present work is to
make some progress towards the non-uniqueness problem
for the compressible fluid models with viscosity.

Since the ground breaking works \cite{dls09,dls10} by De Lellis and Sz\'{e}kelyhidi
on the existence of infinitely many solutions to incompressible Euler equations,
there have been significant progresses towards the non-uniqueness problem for various fluid models
in the last decade.
One milestone is the resolution of the flexible part of Onsager's conjecture for incompressible Euler equations,
by Isett \cite{I18} and Buckmaster-De Lellis-Sz\'{e}kelyhidi-Vicol \cite{bdsv19}.
See also Bressan-Murray \cite{BM20} for other constructive methods.

For the Euler equations of inviscid compressible fluid,
although the uniqueness and stability results have been well established in the 1D case
with small BV initial data and mild assumptions (see, e.g., \cite{BCP00,LY99}),
it was a long time open problem for the multi-dimensional case  until it was observed
by De Lellis-Sz\'{e}kelyhidi \cite{dls10}, that
non-unique bounded entropy solutions can be constructed for multi-dimensional compressible Euler equations.
Afterwards, several achievements have been obtained
for the flexibility of compressible Euler equations.
It has been proved in \cite{CDlK15,CKMS21,MK18,KKMM20} that,
for the one-dimensional Riemann initial data,
the multi-dimensional $L^\infty$ entropy weak solutions are not unique,
provided the one-dimensional solution contains  at least one shock.
Non-uniqueness of entropy weak solutions to the Chaplygin gas with
one-dimensional Riemann initial data was constructed in \cite{BKM21},
provided the one-dimensional solution contains a contact discontinuity or a delta shock.
In \cite{LXX16},
Luo-Xie-Xin proved that neither the damping nor the rotation effect can prevent this non-uniqueness mechanism.
Furthermore,
the dense existence of the initial data
that generate non-unique solutions was provided
by Chen-Vasseur-Yu \cite{CVY21} and Chiodaroli-Feireisl \cite{CF22}.
Very recently,
non-unique smooth density and H\"older-continuous momentum solutions to the
3D compressible Euler equations have been constructed by Giri-Kwon \cite{GK22}. We also refer to Feireisl-Li \cite{FL20} for the non-uniqueness of weak solutions to the inviscid compressible fluid under the mutual interactions with magnetic field.

Regarding the incompressible fluid models with viscosity,
in the breakthrough work  \cite{bv19b} Buckmaster-Vicol proved the non-uniqueness of weak solutions
to the 3D incompressible Navier-Stokes equations (INS for short).
One of the key ingredients in \cite{bv19b} is the spatial intermittency,
which in particular permits to control the viscosity.
The intermittent convex integration schemes have been applied successfully
to various viscous fluid models
including,
e.g.,
INS \cite{bcv21,cl20.2,luo19,lt20,lq20,lqzz22},
MHD equations \cite{bbv20,dai18,lzz21,lzz21.2,MY22,NY22}
and Euler equations \cite{BMNV21,NV22}.
We would like to mention that,
the non-uniqueness results proved by Luo-Titi \cite{lt20}
and Buckmaster-Colombo-Vicol \cite{bcv21}
show that,
the Lions exponent $\alpha=5/4$
is the sharp threshold viscosity for the $C_tL^2_x$ well-posedness
of the 3D hyper-viscous INS.
Moreover,
in \cite{cl20.2}
Cheskidov-Luo proved the non-uniqueness
in $L^p_tL^\infty_x$ for every $1\leq p<2$,
which is sharp
in view of the Lady\v{z}enskaja-Prodi-Serrin criteria.
Recently,
the authors \cite{lqzz22} proved the sharp non-uniqueness
at two endpoint spaces for the 3D hyper-viscous INS,
where the viscosity is beyond the Lions exponent $5/4$.

We would also like to refer to
another programme by Jia-\v{S}ver\'ak \cite{js14,js15} towards
the non-uniqueness problem of Leray-Hopf solutions
to INS,
under a certain assumption of linearized Navier-Stokes operators.
Recently,
Albritton-Bru\'e-Colombo \cite{ABC21} proved the non-uniqueness of
Leray solutions to the 3D INS with a force.
See also \cite{ABC22} for the bounded domain case.
Moreover,
the non-unique Leray solutions of the 2D hypo-viscous forced INS
has been proved in \cite{AC22}.

In contrast to the above extensive studies of
compressible Euler equations and INS,
the non-uniqueness of weak solutions to compressible Navier-Stokes equations
remains a challenging problem.

One major obstruction in the compressible case lies in the relative rigidity of the pressure,
which depends on the density that propagates along the transport equation,
and cannot be simply removed by the Leray projection
as usually done in the incompressible case.
This indeed leads to a rather involving computations of the Reynolds stress,
and requires a careful construction of appropriate perturbations
for both the density and momentum.
More detailed comparisons
between the incompressible and compressible systems with viscosity
can be found in \S \ref{Subsec-Distin} below,
when applying the convex integration scheme.

Furthermore, for the compressible systems,
the result in \cite{LXX16} reveals that the damping and rotation effects,
which can guarantee the well-posedness of small classical solutions,
cannot rule out the non-uniqueness mechanism.
But, for the viscosity effect, there is still no such result.

\subsection{Main results} \label{Subsec-Main}

The aim of this paper is to make some progress
towards the understanding of the non-uniqueness problem
for the compressible Navier-Stokes equations.
We provide the positive answer in the regime of hypo-viscosity.
This shows that the hypo-viscosity cannot rule out
the non-uniqueness mechanism for compressible systems.

Our proof features new constructions of building blocks 
for both the density and momentum, 
which respects the compressible structure. 

Moreover, the strategy of the proof also provides the non-unique $L^2_tC_x$
weak solutions to the hypo-viscous INS (i.e., \eqref{equa-NSE-Incomp} below with $\alpha\in(0,1)$).
As a matter of fact, 
$L^2_tC_x$ is one {\it endpoint space} of the Lady\v{z}enskaja-Prodi-Serrin condition, 
and so one cannot construct non-unique $L^2_tC_x$ weak solutions to
the INS (i.e., \eqref{equa-NSE-Incomp} with $\alpha=1$).
Thus,
the obtained non-uniqueness in $L^2_tC_x$
is sharp for the  hypo-viscous INS,
which reveals that $\alpha=1$
is the sharp viscosity threshold for the well-posedness in the space $L^2_tC_x$.
We note that, 
for another endpoint space $C_tL^2_x$ of the Lady\v{z}enskaja-Prodi-Serrin criteria,
the viscosity threshold is the Lions exponent $\alpha =5/4$,
due to the works \cite{lt20,bcv21,lions69}.

Before formulating the main results,
let us first present the notion of weak solutions
in the distributional sense to the equations \eqref{equa-NS}.

\begin{definition} \label{Def-Weak-Sol} (Weak solutions)
Let $0<T<\infty$. Given any initial data $\rho_0 \in L^{\9}(\T^d)$, $\rho_0> 0$, and $u_0 \in L^2(\mathbb{T}^d)$,
we say that $(\rho, m)\in L^\9([0,T]\times\mathbb{T}^d)\times L^2((0,T)\times \mathbb{T}^d)$
is a weak solution for the hypo-viscous compressible Navier-Stokes equations \eqref{equa-mns} if
\begin{itemize}
\item $\rho\geq 0$ a.e. and
\begin{align*}
	\int\limits_{\mathbb{T}^d} \rho_0(x) \vf(0,x) \d x
	&	= - \int_0^T \int\limits_{\mathbb{T}^d}
	\rho \partial_t \vf + (m\cdot \na) \vf  \d x \d t
\end{align*}
for any test function
$\vf\in C_0^\infty([0,T)\times\mathbb{T}^d)$.
\item $m=0$ whenever $\rho=0$, and
\begin{align*}
	\int\limits_{\mathbb{T}^d} m_0 \vf(0,x) \d x
	&	= - \int_0^T \int\limits_{\mathbb{T}^d}
	m\cdot \partial_t \vf - \frac{1}{\rho}m\cdot (\mu (-\Delta)^\alpha \vf
    - (\mu+\nu)\nabla \div\vf) + (\frac{1}{\rho}m\otimes m):\na\vf + P\div\vf \d x \d t
\end{align*}
for any test function
$\vf\in C_0^\infty([0,T)\times\mathbb{T}^d)$, where $m_0:=\rho_0u_0$ .
\end{itemize}
\end{definition}

The non-uniqueness of weak solutions to the hypo-viscous CNS \eqref{equa-mns} is formulated in Theorem \ref{Thm-Nonuniq-CNS} below.

\begin{theorem} [Non-uniqueness for hypo-viscous CNS] \label{Thm-Nonuniq-CNS}
There exist $\rho_0 \in L^{\9}(\T^d)$, $\rho_0> 0$,
and $m_0 \in L^2(\mathbb{T}^d)$, such that for any exponents $(p,s)$ satisfying
\begin{align} \label{ps-supercri}
	 \a+s-\frac{2\a}{p}<0,  \ \  (p,s)\in [1, 2] \times [0,1),
\end{align}
	there exist infinitely many weak solutions $(\rho, m)\in C_t C^1_x \times L^p_tC^s_x$
	to the hypo-viscous CNS \eqref{equa-mns}
	with the same initial data $(\rho_0, m_0)$.
\end{theorem}

Theorem \ref{Thm-Nonuniq-CNS} is actually
a consequence of a more general result in Theorem~\ref{Thm-Nonuniq-hypoNSE} below.

\begin{theorem} \label{Thm-Nonuniq-hypoNSE}
Let $\alpha \in (0,1)$.
Let $(\wt \rho, \wt m)$ be any given smooth solution to the transport equation
\begin{align}\label{equa-trans}
\p_t \wt \rho+\div \wt m=0
\end{align}
on $\T^d$, $d\geq 2$,
such that $0<c_*\leq\wt\rho(t,x)\leq C_*$ for all $(t,x)\in [0,T]\times \T^d$,
where $c_*$ and $C_*$ are universal constants.
Then, there exists $\beta'\in(0,1)$,
such that for any $\va_*>0$ and any exponents $(p,s)$
satisfying \eqref{ps-supercri},
there exist  $\rho$ and $m$ to \eqref{equa-mns}
such that the following holds:
\begin{enumerate}
    \item[(i)] Weak solutions: $(\rho, m)$ is the weak solution to \eqref{equa-mns}
    in the sense of Definition \ref{Def-Weak-Sol} but with $m_0\in H^{-1}_x$.
	\item[(ii)] Regularity:
            $$\rho\in C_t C^1_x, \ \ m\in H^{\beta'}_tC_x\cap L^p_tC^s_x\cap C_tH^{-1}_x.$$
    \item[(iii)] Mass preservation: $$\int_{\mathbb{T}^d}\rho(t,x)\d x= \int_{\mathbb{T}^d}\wt \rho(t,x)\d x,\ \forall\ t\in [0,T].$$
	\item[(iv)] Small deviation of norms:
	$$\|\rho-\wt \rho\|_{C_{t}C_x^1 }\leq \va_*,\
	\|m-\wt {m}\|_{L^1_tC_x \cap L^p_t C_x^s\cap C_tH^{-1}_x} \leq \va_*.$$
	\item[(v)] Small deviation of temporal support: if $\wt T:=\inf_{t\in [0,T]}\{t\, | \nabla \wt \rho(t,\cdot)\neq0,\, \wt m(t,\cdot)\neq 0  \}>0$,
then for any $0<\ve_*<\wt T$,
	$$\supp_t (\nabla \rho, m)  \subseteq N_{\va_*}([\wt T,T]).$$
\end{enumerate}
\end{theorem}

Note that, for any $A\subseteq [0,T]$ and $\ve_*>0$,
$N_{\va_*}(A)$ denotes the $\ve_*$-neighborhood of $A$ in $[0,T]$, namely,
	\begin{align*}
		N_{\va_*}(A):=\{t\in [0,T]:\ \exists\,s\in A,\ s.t.\ |t-s|\leq \va_*\}.
	\end{align*}

Our proof of Theorem \ref{Thm-Nonuniq-hypoNSE} also
applies to the compressible Euler equations \eqref{equa-Euler}
and the hypo-viscous INS \eqref{equa-NSE-Incomp} below.
Hence, we also have the following non-uniqueness of weak solutions.

\begin{corollary} [Non-uniqueness for compressible Euler equations] \label{Cor-Nonuniq-Euler}
There exist $\rho_0\in L^\infty(\mathbb{T}^d)$ and $m_0\in L^2(\mathbb{T}^d)$,
such that
for any exponents $(p,s)$ satisfying \eqref{ps-supercri} with any $\alpha\in (0,1)$,
there exist infinitely many weak solutions
$(\rho, m)\in C_t C^1_x \times L^p_tC^s_x$ to
the compressible Euler equations \eqref{equa-Euler}
with the initial data $(\rho_0, m_0)$.
\end{corollary}

\begin{corollary} [Non-uniqueness for hypo-viscous INS] \label{Cor-Nonuniq-NSE}
Consider the incompressible Navier-Stokes equations
\begin{equation}\label{equa-NSE-Incomp}
	\left\{\aligned
	&\p_t u+ \mu(-\Delta)^{\a}u +\div(u\otimes u) + \na P =0,  \\
	&  \div u = 0,
	\endaligned
	\right.
\end{equation}
where $\mu>0$, $\alpha \in (0,1)$.
Then, there exists $u_0\in L^2(\mathbb{T}^d)$
such that for any exponents $(p,s)$ satisfying \eqref{ps-supercri},
there exist infinitely many weak solutions
$u \in L^p_tC^s_x$
with the initial datum $u_0$.
\end{corollary}

Because the proof of Corollary \ref{Cor-Nonuniq-NSE} basically follows
the lines of the proofs of Theorems \ref{Thm-Nonuniq-CNS} and \ref{Thm-Nonuniq-hypoNSE},
we will mainly shows the differences,
such as the construction of building blocks and the choice of Reynolds stress,
in the remarks in the convex integration scheme.
See Remarks \ref{incom-3.7}, \ref{rem-3.10}, \ref{rq1-incom},
\ref{rem-supp-incom} and \ref{Rem-proof-INS} below.

Furthermore,
we consider the strong vanishing viscosity limit,
which relates the hypo-viscous CNS and compressible Euler equations.

\begin{theorem} [Strong vanishing viscosity limit for hypo-viscous CNS]   \label{Thm-hypoNSE-Euler-limit}
	Let $\alpha \in (0,1)$
	and $\rho, m\in C^{\wt \beta}_{t,x}$, $\wt \beta >0$,
    be any weak solution
	to the compressible Euler equations \eqref{equa-Euler},
    such that $c_1\leq \rho \leq c_2$ for some constants $c_1, c_2>0$.
	Then, there exist $\beta' \in (0, \wt \beta)$ and a sequence of weak solutions
	$(\rho^{({n})},m^{({n})})\in C_{t,x}\times H^{\beta'}_{t}C_x $
	to the hypo-viscous CNS \eqref{equa-NS} with viscous coefficients $\kappa_n\mu$ and $\kappa_n\nu$,
	such that
	as $\kappa_n\rightarrow 0$,
	\begin{align}\label{convergence}
	\rho^{( {n})}\rightarrow \rho \quad\text{strongly in}\ C_{t,x},
    \quad \text{and}\quad	m^{( {n})}\rightarrow m \quad\text{strongly in}\  H^{\beta'}_{t}C_x.
	\end{align}
\end{theorem}

\paragraph{\bf Comments on main results}
Let us present some comments on the main results below.
\medskip

$(i)$ {\bf Non-uniqueness of weak solutions for hypo-viscous CNS.}
To the best of our knowledge,
Theorem \ref{Thm-Nonuniq-CNS} provides the first result
for the non-uniqueness of weak solutions
to the compressible fluid with viscosity.

It also shows that,
the hypo-vsicosity cannot rule out the non-uniqueness mechanism,
which, actually, coincides
with the Lady\v{z}enskaja-Prodi-Serrin criteria
(see also Comment $(ii)$ below).

Another interesting outcome is concerned with the anomalous dissipation
for Euler equations.
Differently from the incompressible Euler equations,
for which the flexible part of the Onsager conjecture
in the 3D case has been proved recently in \cite{I18,bdsv19}
by using the convex integration method
(see also \cite{BMNV21}  and \cite{NV22} for anomalous dissipation in the spaces $C_tH^{\frac 12 -}$ and $C_tB^{s}_{3,\infty}$,
respectively),
the anomalous dissipation for the compressible Euler equations
is well known for low regular entropy solutions with shock even in the 1D case.
Meanwhile, by the result in \cite{BGSTW19}, the energy is conserved for $C^{\frac{1}{3}+}$ weak solutions,
while by the result in \cite{GK22}, the energy can be dissipated for $C^{\frac{1}{7}-}$ entropy weak solutions.
Thus, the compressible Onsager conjecture, namely, the threshold for the regularity of anomalous dissipation,
is still open.

Corollary \ref{Cor-Nonuniq-Euler} provides new examples
of anomalous dissipation in the class $L^1_tC^{1-}_x$
for the compressible Euler equations,
by taking $p=1$ and $\alpha$ close to $1$.
We note that a similar result for the incompressible Euler equations
was provided in  \cite{cl20.2} with a different convex integration scheme.

It is worth noting that
our proof also applies to the non-uniqueness problem of stochastic
hypo-viscous CNS driven by additive noise.
This will be done in a forthcoming work.
\medskip

$(ii)$ {\bf Sharp viscosity threshold for $L^2_tC_x$ well-posedness.}
It is usually believed that,
high dissipation helps to
establish the well-posedness of PDEs,
yet ill-posedness may exhibit in the low dissipative case.
Hence, heuristically,
there should be a threshold of the dissipation for the well-posedness.

For the 3D hyper-viscous INS,
on one hand,
it is well-posed in $C_tL^2_x$
if the viscosity exponent $\alpha\geq 5/4$,
due to Lions \cite{lions69}.
On the other hand,
the non-uniqueness exhibits
in the low viscous regime where $1\leq \alpha<5/4$,
due to Luo-Titi \cite{lt20}
and Buckmaster-Colombo-Vicol \cite{bcv21}.
Thus,
$\alpha=5/4$ (known as the Lions exponent)
is the threshold viscosity for the $C_tL^2_x$ well-posedness
for the 3D hyper-viscous INS.

Another well-known criteria of well-posedness
is the Lady\v{z}enskaja-Prodi-Serrin condition,
which relates closely to the scaling invariance of equations.
It is known that the INS
is well-posed in the space $L^2_tL^\infty_x$,
which is actually the critical space in view of the  Lady\v{z}enskaja-Prodi-Serrin criteria.
Yet, there exist infinitely many weak solutions
in the supercritical spaces $L^p_tL^\infty_x$ for any $1\leq p<2$,
due to Cheskidov-Luo \cite{cl20.2}.
See also \cite{lqzz22} for the non-uniqueness for hyperdissipative NSE
near two endpoints of Lady\v{z}enskaja-Prodi-Serrin condition.

When $(p,s)=(2,0)$,
Corollary \ref{Cor-Nonuniq-NSE} provides the non-uniqueness of weak solutions
in the space $L^2_tC_x$
for the hypo-viscous INS,
which is sharp due to the above well-posedness in $L^2_tL^\infty_x$,
in view of the Lady\v{z}enskaja-Prodi-Serrin criteria.

Thus,
Corollary \ref{Cor-Nonuniq-NSE} reveals that
$\alpha=1$ is the sharp threshold viscosity
for the $L^2_tC_x$ well-posedness for the INS.
It would serve as a counterpart
of the recent results in \cite{lt20,bcv21}
regarding the sharp threshold of the Lions exponent $\alpha=5/4$
for the $C_tL^2_x$ well-posedness of the 3D INS.
It is worth noting that, both the spaces $L^2_tC_x$ and $C_tL^2_x$
are the endpoint spaces of the  Lady\v{z}enskaja-Prodi-Serrin criteria.

This also complements the very recent work \cite{AC22},
where the non-uniqueness of Leray-Hopf solutions is proved
for the 2-D hypo-viscous INS with a force.

We would expect the non-uniqueness in the case $(p,s)=(2,0)$
in Theorem \ref{Thm-Nonuniq-CNS} is also sharp for the hypo-viscous CNS,
say, at least if the density is of small variation.
To our best knowledge,
the Lady\v{z}enskaja-Prodi-Serrin criteria for the uniqueness of general weak solutions
to CNS remains open,
yet the Serrin blow-up criterion
for the 3D CNS is known (see, e.g., \cite{HLX11}).
\medskip

$(iii)$ {\bf Vanishing viscosity limit for hypo-viscous CNS.}
In the breakthrough work \cite{bv19b},
Buckmaster-Vicol proved the vanishing viscosity limits for the 3D INS,
based on the intermittent convex integration scheme.
The case of hyper-viscous INS was proved by the authors \cite{lqzz22}.
See also \cite{lzz21,lzz21.2} for the case of viscous and resistive MHD equations.

Regarding the compressible case,
the 1D case with artificial viscosity was solved by Bianchini-Bressan \cite{BB05}.
But it is still open for the multi-dimensional case with general initial data.
By virtue of Theorem \ref{Thm-hypoNSE-Euler-limit},
we see that in the vanishing viscosity limit,
the set of accumulation points
of weak solutions to the hypo-viscous CNS \eqref{equa-mns}
contains all the H\"older continuous weak solutions to the compressible Euler equations \eqref{equa-Euler}.

In particular,
the entropy solutions in the H\"older spaces recently constructed in \cite{GK22}
may be obtained as the strong vanishing
viscosity limits of weak solutions
to the hypo-viscous CNS \eqref{equa-NS} when $d=3$.

\subsection{Distinctions between compressible and incompressible cases}   \label{Subsec-Distin}
Let us mention the distinctions between the compressible and incompressible fluids with viscosity,
from the aspects of the pressure and density,
the Reynolds stress,
as well as the building blocks.
These three aspects are actually not isolated,
but related closely to each other.
\medskip

$(i)$ {\bf The pressure and density.}
One obvious difference is that,
the pressure law in the compressible case depends on the density,
which propagates along the transport equation.
In the convex integration scheme,
this reflects the relative rigidity for the choice of the pressure in the compressible case,
while it is rather flexible in the incompressible case.

In fact, a widely used technique in the convex integration scheme
for the incompressible systems is,
that one can take the advantage of the pressure term to absorb large error terms
as long as they are in the gradient form.
However,  for the compressible systems,
the variation of the pressure is tied to that of the density,
which, via the transport equation, also changes the momentum and
may give rise to rather complicated calculations.

In this paper, we do face this kind of difficulties.
For instance,
due to the temporal intermittency,
the temporal corrector is needed to balance the high temporal oscillation error (see \eqref{wo} below).
The application of the Leray projection,
as in the incompressible systems (cf. \cite{cl20.2}, and also  \cite{cl21.2,cl21,lqzz22,lzz21.2,lzz21}),
would leave the main part of the error as a gradient term to be absorbed in the pressure,
which now has to be identified in the compressible case
and would lead to complicated calculations.

Instead, we cancel the main part of the high temporal oscillation error directly,
but then this temporal corrector is not divergence-free,
it thus leads to the variation of the density in the iteration scheme,
which in turn requires the construction of suitable density perturbation.
This phenomenon is also quite different from the existing schemes
for compressible Euler equations in \cite{dls10,GK22},
where the density function is a priori given and does not change through the whole iteration procedure.

It is also worth noting that,
due to the intermittency used to control the large errors caused by the viscosity,
in most of the convex integration schemes for the INS,
the pressure is not bounded from below
and hence is not allowed for the compressible systems.
\medskip

$(ii)$  {\bf The Reynolds stress.}
Another major technical obstruction when treating the compressible fluid with viscosity
can be seen from the Reynolds stress in the convex integration scheme.

As a matter of fact,
one of the basic ideas
in the convex integration method
is to decrease the Reynolds error
by the effect of the quadratic nonlinearity of appropriate perturbations,
namely,
\[
	\int_{\mathbb{T}^d} w_{q+1} \otimes w_{q+1} d x \approx - R_q,
\]
where $R_q$ is the Reynolds stress at the previous level $q$,
and $w_{q+1}$ is the perturbation at level $q+1$.
It requires that the main part of $R_q$ should be at least semi-negative definite.
This property was carefully kept in each step in the differential inclusion schemes (see, e.g., \cite{dls09}),
however, is hard to be preserved in the intermittent schemes as in \cite{bv19b,cl21}.
Instead, in the Nash style schemes,
the Reynolds error cancelled by the low-frequency of nonlinearity
is taken of the form $ \varrho \Id - R_q$,
rather than $R_q$ itself,
namely,
\begin{align} \label{w2-rho-Rq}
\int_{\mathbb{T}^d} w_{q+1} \otimes w_{q+1} d x \approx \varrho \Id - R_q.
\end{align}
The extra function $\varrho$ is usually chosen as $| R_q(t,x)|$ for the systems with viscosity,
see, e.g., \cite{lt20,bcv21,cl21}.
It causes no trouble in the incompressible case
because,
as mentioned above, the resulting term $\div \varrho \Id = \nabla \varrho $
can be absorbed into the pressure,
which can be removed simply by using the Leray projection.

However, this strategy fails in the compressible case
because of the relative rigidity of the pressure term.
The analysis of Reynolds stress is rather involving in
the compressible case, see, e.g., the expression \eqref{ru} below.
In order to preserve the above convex integration scheme,
namely, to preserve the semi-positivity of $\varrho \Id - R_q$
and the identity $\div \varrho \Id = 0$ simultaneously,
we are led to choose $\varrho $ in terms of $\| R_q(t,\cdot) \|_{C_{x}}$.
This in turn requires the construction of appropriate  building blocks.
\medskip

$(iii)$ {\bf The building blocks.}
Unlike the previous intermittent schemes,
the building blocks in our proof feature {\it spatial homogeneity} but {\it temporal intermittency}.

On one hand,
we need to control the $L^1_tC_x$-norm of
Reynolds error,
for which the spatial intermittency would cause even larger errors.
This suggests that the suitable building blocks should be spatially homogeneous.

On the other hand,
in order to capture the viscosity,
an extra intermittency should be exploited from the building blocks.
Inspired by the spatial-temporal intermittent scheme raised by Cheskidov and Luo in \cite{cl21}
(see also \cite{cl20.2,cl21.2} and \cite{lzz21,lzz21.2,lqzz22} for further developments),
we make use of the temporal intermittency.

The temporal intermittency is, actually, exploited in the full strength here,
in order to capture the hypo-viscosity $(-\Delta)^\alpha$ for all $\alpha\in (0,1)$.
This in turn requires a careful choice of three parameters $(\lbb, \tau, \sigma)$,
where $\lbb$ denotes the frequency of building blocks,
$\tau$ and $\sigma$ parameterize the temporal concentration
and oscillation effects, respectively.
See \eqref{larsrp} for the precise choice of parameters.
It turns out that,
the corresponding building blocks can provide almost
2D intermittency for $\alpha$ close to one,
i.e., close to the case of CNS.

Another feature of our building blocks is the
{disjointness} of the temporal supports
of different building blocks,
which in particular permits to treat the 2D case.
Actually,
it is known that, the periodic Mikado flows along different directions
interact with each other in the 2D case.
This kind of interactions would give rise to large errors,
and is one of the main reasons that the flexible part of Onsager's conjecture is still open
for the 2D incompressible Euler equations.
Moreover,
because
the Mikado flows used as the present spatial building blocks
are not concentrated,
the proof in the current compressible situation
is different from that of \cite{cl20.2} in the incompressible setting.
Rather than the spatial shifts,
we use the {\it temporal shifts} to decouple different interactions,
which seems new in the existing literature of temporal-intermittent method
and might be of use for further study of fluid models.
\medskip

\paragraph{\bf Organization}
The remaining of the paper is organized as follows.
In \S \ref{Sec-Main-Iter} we present the main iteration estimates,
which are the key in the proof of Theorem \ref{Thm-Nonuniq-hypoNSE}.
The mollification procedure is also performed in order to
avoid the loss of derivative.
Then, \S \ref{Sec-Perturb} is devoted to the construction of the
key building blocks, and the perturbations of the density and momentum.
A set of algebraic identities and analytic estimates are also obtained.
In \S \ref{Sec-Rey-Endpt1} we mainly treat the Reynolds stress
and verify the corresponding iterative estimates.
Finally, the main results are proved in  \S \ref{Sec-Proof-Main}.
\medskip

\paragraph{\bf Notations}  We denote for $d\geq 2$, $p\in [1,\infty]$, $s\in \R$, $N\in \mathbb{N}$ and $\eta\in(0,1)$,
\begin{align*}
	L^p_t:=L^p(0,T),\quad L^p_x:=L^p(\T^d),\quad C^N_x:=C^N(\T^d),\quad C^{N,\eta}_x:=C^{N,\eta}(\T^d)\quad  W^{s,p}_x:=W^{s,p}(\T^d),
\end{align*}
where $W^{s,p}_x$ is the usual Sobolev space,
$C^{N,\eta}_x$ is the H\"older space equipped with the norm
$$
\norm{u}_{C_{x}^{N,\eta}}:=\sum_{0\leq|\zeta|\leq N}
\norm{\na^{\zeta} u}_{C_{x}}+ \max_{|\zeta|=N}\sup_{x\neq y\in \mathbb{T}^d}
\frac{|\na^\zeta u(x)- \na^\zeta u(y) | }{ |x-y|^{\eta}},
$$
and $\zeta=(\zeta_1,\zeta_2,\dots\zeta_d)$ is the multi-index
and $\na^\zeta:= \partial_{x_1}^{\zeta_1} \partial_{x_2}^{\zeta_2}\cdots\partial_{x_d}^{\zeta_d}$.
When $N=0$, we denote $C^{\eta}_x:=C^{0,\eta}(\T^d)$ for brevity. We also use the shorthand notation $L^\gamma_tL^p_x$ to denote $L^\gamma(0,T;L^p(\T^d))$, where $p, \gamma\in [1,\infty]$. In particular, we write $L^p_{t,x}:= L^p_tL^p_x$ for short.
Moreover, let
\begin{align*}
	\norm{u}_{W^{N,p}_{t,x}}:=\sum_{0\leq m+|\zeta|\leq N} \norm{\p_t^m \na^{\zeta} u}_{L^p_{t,x}}, \ \
	\norm{u}_{C_{t,x}^N}:=\sum_{0\leq m+|\zeta|\leq N}
	\norm{\p_t^m \na^{\zeta} u}_{C_{t,x}},
\end{align*}
Given any Banach space $X$, $C([0,T];X)$ denotes the space of continuous functions from $[0,T]$ to $X$,
equipped with the norm $\|u\|_{C_tX}:=\sup_{t\in [0,T]}\|u(t)\|_X$.

We would also write $a\lesssim b$ to imply that $a\leq C b$ for some constant $C>0$.

\section{Main iteration and mollification procedure} \label{Sec-Main-Iter}

This section contains the main iteration of the density, momentum
and Reynolds stress,
which is the heart of the proof of main results.
\medskip

\paragraph{\bf Main iteration}
Let us consider the approximate solutions to the following
hypo-viscous compressible Navier-Stokes-Reynolds system
\begin{equation}\label{equa-nsr}
	\left\{\aligned
	& \p_t \rho_q +\div m_q = 0,\\
	&\p_t m_q+\mu(-\Delta)^{\a}(\rho_q^{-1}m_q)-(\mu+\nu)\nabla \div (\rho_q^{-1}m_q) +\div(\rho_q^{-1}\m\otimes \m) + \nabla P(\rho_q)=\div \rrq.
	\endaligned
	\right.
\end{equation}
where $q\geq 0$, $R_q$ is the Reynolds stress which is a symmetric $3\times 3$ matrix.

Two important quantities measuring the size of relaxed solutions
$(\rq, \m, \rrq)$, $q\in \mathbb{N}$,
are the frequency parameter $\lbb_q$ and the amplitude parameter $\delta_{q}$,
chosen in the following way:
\begin{equation}\label{la}
	\la=a^{(b^q)}, \ \
	\delta_{q}=\lambda_{1}^{3\beta}\lambda_{q}^{-2\beta},
\end{equation}
where $a\in \mathbb{N}$ is a sufficiently large integer,
$\beta>0$ is the regularity parameter,
$b\in 2\mathbb{N}$ is a large integer of multiple $2$ such that
\begin{align} \label{b-beta-ve}
	b>\frac{1000}{\varepsilon}, \ \
	0<\beta<\frac{1}{100b^2}
\end{align}
with $\varepsilon\in \mathbb{Q}_+$ sufficiently small such that
\begin{equation}\label{e3.1}
	\varepsilon\leq\frac{1}{20}\min\{1-\alpha,\a, \frac{2\a}{p}-\a-s\}\quad \text{and}\quad b\ve\in\mathbb{N}.
\end{equation}

The heart of proof is to construct suitable relaxed solutions
$(\rq, \m, \rrq)$, $q\in \mathbb{N}$,
such that the Reynolds stress gets smaller
while the density and momentum converge in certain appropriate spaces,
as $q$ goes to infinity.

The quantitative description of this objective will be stated in the main iteration in Theorem \ref{Thm-Iterat} below,
which in part includes the following crucial inductive estimates
at level $q\in \mathbb{N}$:
\begin{align}
	& C_1+\la^{-\beta}\leq \rq\leq  C_2-\la^{-\beta},\label{rhobd}\\
	& \|\p_t^M\rq\|_{C_tC^N_{x}} \lesssim  \lambda_{q}^{\frac{N\ve}{4}}, \label{rhoc1} \\
    & \|\m\|_{C^N_{t,x}} \lesssim  \lambda_{q}^{2N+2}, \label{mc1} \\
	& \|\rrq\|_{C^1_{t,x}} \lesssim   \lambda_{q}^{9},\label{rc1}	\\
    & \|\rrq\|_{L^{1}_{t}C_x} \lesssim \dqq, \label{rl1}
\end{align}
where $1\leq N\leq 4$, $0\leq M\leq 1$, $C_1$ and $C_2$ are universal constants
such that $0<C_1<{c_*}/{2},\ C_2>2C_*$, and the implicit constants are independent of $q$. Concerning the temporal support of relaxed solutions, we denote
\begin{align}\label{def-tq}
T_q:=\inf_{t\in [0,T]}\{t\, | \nabla\rho_{q}(t,\cdot)\neq0,\,  m_{q}(t,\cdot)\neq0,\,  R_{q}(t,\cdot)\neq0  \},
\end{align}
for all $q\in \mathbb{N}$.

\begin{theorem} [Main iteration]\label{Thm-Iterat}
Let $\alpha \in (0,1)$ and $(p,s)$ satisfy \eqref{ps-supercri}.
Then, there exist $\beta\in (0,1)$ and $a_0=a_0(\beta, M^*)$ large enough,
 such that for any integer $a\geq a_0$,
the following holds:

Suppose that
$(\rq,\m, \rrq)$ is a relaxation solution to \eqref{equa-nsr} satisfying \eqref{rhobd}-\eqref{rl1}
at level $q\in \mathbb{N}$.
Then, there exists a new relaxation solution $(\rqq,\mq, {R}_{q+1} )$
to \eqref{equa-nsr} at level $q+1$,
which satisfies \eqref{rhobd}-\eqref{rl1} with $q+1$ replacing $q$
and, in addition,
the following estimates:
\begin{align}
& \int_{\mathbb{T}^d}\rho_{q+1}(t,x)\d x= \int_{\mathbb{T}^d} \rho_q(t,x)\d x,\ \forall\ t\in [0,T].\label{rhopre}\\
&\norm{ \rqq - \rq }_{C_tC^1_x} \lesssim  \delta_{q+2}^{\frac{1}{2}} ,\label{rho-l9-conv} \\
&\|\mq-\m\|_{L^{2}_{t}C_x} \lesssim \delta_{q+1}^{\frac{1}{2}}, \label{m-L2t-conv}\\
&\|m_{q+1}-m_{q}\|_{L^1_tC_x \cap L^p_t C^s_x\cap C_tH^{-1}_x} \lesssim  \delta_{q+2}^{\frac 12},   \label{m-L1tLace-conv}
\end{align}
where the implicit constants are independent of $q$. In addition, if $T_0>0$, then
\begin{align} \label{suppru}
&T_{q+1}\geq T_q-\delta_{q+2}^{\frac 12}>0.
\end{align}
\end{theorem}

In order to prove Theorem \ref{Thm-Iterat},
let us first perform the mollification procedure to the relaxed Navier-Stokes-Reynolds system \eqref{equa-nsr},
which permits to avoid the loss of derivative in the convex integration scheme.
\medskip

\paragraph{\bf Mollification procedure}
Let $\phi_{\epsilon}$ and $\varphi_{\epsilon}$ be standard mollifiers on $\T^d$ and $\T$, respectively,
and $\supp \vf_\ve\subseteq (-\epsilon, \epsilon)$, $\epsilon>0$.
Then, the mollifications of $(\rq, m_{q}, R_{q})$ in space and time are defined by
\begin{align}\label{mol}
\rho_{\ell}:=\left(\rho_{q} *_{x} \phi_{\ell}\right) *_{t} \varphi_{\ell},  	
   \ m_{\ell}:=\left(\m *_{x} \phi_{\ell}\right) *_{t} \varphi_{\ell},
   \  R_\ell:=\left(R_q *_{x} \phi_{\ell}\right) *_{t} \varphi_{\ell},
   \ P_{\ell} := \left( P(\rho_q)*_{x} \phi_{\ell}\right) *_{t} \varphi_{\ell},
\end{align}
where the scale of mollification is given by
\begin{align} \label{l-lbbq}
	\ell:=\la^{-30}.
\end{align}
Then, we derive from \eqref{equa-nsr} that
$(\rho_{\ell}, m_{\ell},R_\ell)$ satisfies
\begin{equation}\label{equa-me}
	\left\{\aligned
	& \p_t \rho_\ell +\div m_\ell = 0,\\
	&\p_t m_\ell+\mu(-\Delta)^{\a} ({\rho_\ell}^{-1}m_\ell) -(\mu+\nu)\nabla \div ({\rho_\ell}^{-1}m_\ell)
    +\div({\rho_\ell}^{-1}m_\ell\otimes m_\ell)
+ \nabla P(\rho_{\ell}) =\div( R_\ell+R_{com}).
	\endaligned
	\right.
\end{equation}
where the commutator stress
	\begin{align}\label{def-rcom}
	 R_{com} 	&:=\mathcal{R}\nabla(P(\rho_{\ell})- P_{\ell})
      +\mu \mathcal{R} (-\Delta)^{\a} \(\rho_\ell^{-1}m_\ell  - (\rho_q^{-1}m_q)*_{x} \phi_{\ell} *_{t} \varphi_{\ell}\)\notag \\
	  &\quad - (\mu+\nu) \mathcal{R} \nabla \div \(\rho_\ell^{-1}m_\ell- (\rho_q^{-1}m_q)*_{x} \phi_{\ell} *_{t} \varphi_{\ell}\) \notag\\
	  &\quad  + \rho_\ell^{-1} m_{\ell}\otimes m_{\ell}- (\rho_q^{-1}\m \otimes \m ) *_{x} \phi_{\ell} *_{t} \varphi_{\ell}
	\end{align}
with $\mathcal{R}$ being the inverse-divergence operator given by \eqref{calR-def} below.

Thus, using the standard mollification estimates and inductive estimates \eqref{rhobd}-\eqref{rl1}
we derive that for every integers $1\leq N\leq 4$ and $0\leq M\leq 1$,
\begin{align}
	&C_1+\frac12\la^{-\beta}\leq \rho_{\ell} \leq
	C_2-\frac12\la^{-\beta},\label{est-moll-rhobd}\\
    & \left\|\p_t^M\rho_\ell\right\|_{C_tC_{x}^N}
	\lesssim  \left\| \p_t^M\rho_q \right\|_{C_tC_{x}^N} \lesssim \lambda_{q}^{\frac{N\ve}{4}},    \label{est-moll-rhoq} \\
	& \left\|\rho_\ell- \rho_q\right\|_{C_tC_{x}^{N-1}}
	\lesssim  \ell \left\|\rho_q\right\|_{C_t^1C_{x}^{N-1}} + \ell \left\|\rho_q\right\|_{C_tC_{x}^{N}} \lesssim \ell\lambda_{q}^{\frac{N\ve}{4}},    \label{est-moll-rl-rq} \\
	& \left\|m_\ell\right\|_{C_{t,x}^N}
	\lesssim \ell^{-N+1}\left\|m_q\right\|_{C_{t,x}^1} \lesssim  \ell^{-N+1}\la^4,    \label{est-moll-mc1} \\
	& \left\|m_\ell-m_q\right\|_{C_{t,x}^{N-1}}
	\lesssim \ell \left\|m_q\right\|_{C_{t,x}^{N}} \lesssim  \ell \la^{2N+2},   \label{est-moll-ml-mq} \\
	&  \|  {R}_\ell\|_{C_{t,x}^N}  \lesssim \ell^{-N+1} \|  {R}_q \|_{C_{t,x}^1}
	\lesssim  \ell^{-N+1} \lambda_{q}^{9}\lesssim \ell^{-N},  \label{est-moll-rcn}\\
	&  \|{R}_\ell \|_{L^{1}_{t}C_x}\lesssim  \|{R}_q \|_{L^{1}_{t}C_x}\lesssim   \delta_{q+1}. \label{est-moll-rl1}
\end{align}

\section{Constructions of momentum and density perturbations}  \label{Sec-Perturb}

The aim of this section is to construct appropriate momentum and density perturbations,
such that the corresponding inductive estimates in Theorem \ref{Thm-Iterat} propagate through
in the convex integration scheme.

The fundamental building blocks in this section will be indexed by the following three parameters
\begin{equation}\label{larsrp}
	 \lambda := \lambda_{q+1},\ \
      \tau:=\lambda_{q+1}^{2\a-10\varepsilon}, \ \  \sigma:=\lambda_{q+1}^{15\varepsilon},
\end{equation}
where $\alpha\in (0,1)$, and
$\varepsilon$ is a sufficiently small constant satisfying \eqref{e3.1}.
In particular,
the parameters $\tau$ and $\sigma$
parameterize, respectively,
the temporal concentration and oscillation effects
of the building blocks.
It is also worth noting that,
the building blocks are able to provide almost 2D intermittency for $\alpha$ close to one.

\subsection{Spatial building blocks.}
Let us first recall the Geometric Lemma in \cite[Lemma 4.1]{bcv21}
which will be used to construct the basic spatial building blocks,
namely, the Mikado flows.

\begin{lemma} ({\bf Geometric Lemma}, \cite[Lemma 4.1]{bcv21})
	\label{geometric lem 2}
	There exist a finite set $\Lambda \subseteq \mathbb{S}^{d-1} \cap \mathbb{Q}^d$ and smooth positive functions
    $\gamma_{(k)}: B_{\varepsilon_u}(\Id) \to \mathbb{R}$,
	where $B_{\varepsilon_u}(\Id)$ is the ball of radius $\varepsilon_u(>0)$ centered at the identity
	in the space of $d \times d$ symmetric matrices,
	such that for  $S \in B_{\varepsilon_u}(\Id)$ we have the following identity:
	\begin{equation}
		\label{sym}
		S = \sum_{k \in \Lambda} \gamma_{(k)}^2(S) k \otimes k.
	\end{equation}
\end{lemma}

Let $\{k,k_1,...,k_{d-1}\}$ be an orthonormal base associated with $k$.
Note that, there exists $N_{\Lambda} \in \mathbb{N}$ such that
\begin{equation} \label{NLambda}
	\{ N_{\Lambda} k,N_{\Lambda}k_1,...,N_{\Lambda}k_{d-1} \} \subseteq N_{\Lambda} \mathbb{S}^{d-1} \cap \mathbb{Z}^d.
\end{equation}
Then, let $\Phi : \mathbb{R}^{d-1} \to \mathbb{R}$ be any smooth cut-off function supported
on a ball of sufficiently small radius around the origin,
and then normalize $\Phi$ such that $\phi := - \Delta\Phi$ satisfies
\begin{equation}\label{e4.91}
	\frac{1}{(2 \pi)^{d-1}}\int_{\mathbb{R}^{d-1}} \phi^2(x)\d x = 1.
\end{equation}
By an abuse of notation,
we periodize $\phi$ and $\Phi$ so that they are treated as periodic functions defined on  $\mathbb{T}^d$.

Inspired by \cite{bv19r,cl20.2,DS2017},
we choose the {\it Mikado flows}
as the basic spatial building blocks, defined by
\begin{equation*}
	W_{(k)}(x) :=  \phi ( \lambda  N_{\Lambda}k_1\cdot x,\dots,\lambda  N_{\Lambda}k_{d-1}\cdot x)k,\ \ x\in \T^d,\ \  k \in \Lambda.
\end{equation*}
For brevity of notations, we set
\begin{equation}\label{snp}
	\begin{array}{ll}
		&\phi_{(k)}(x) := \phi ( \lambda  N_{\Lambda}k_1\cdot x,\dots,\lambda  N_{\Lambda}k_{d-1}\cdot x), \\
		&\Phi_{(k)}(x) :=  \lambda^{-2} \Phi (  \lambda  N_{\Lambda}k_1\cdot x,\dots,\lambda  N_{\Lambda}k_{d-1}\cdot x).
		\end{array}
\end{equation}
Thus,
the Mikado flows can be reformulated as
\begin{equation}\label{snwd}
	W_{(k)} = \phi_{(k)} k,\quad  k\in \Lambda.
\end{equation}

Furthermore,
in order to define the incompressible corrector of the perturbations,
we also use the skew-symmetric potential $W_{(k)}^c$,
defined by
\begin{equation}
	\begin{aligned}\label{potential vector}
		W_{(k)}^c := \nabla \Phi_{(k)}\otimes k-k\otimes \nabla \Phi_{(k)}.
	\end{aligned}
\end{equation}
It holds that
\begin{equation}\label{wcwc}
	 \div W_{(k)}^c=W_{(k)}.
\end{equation}

We present the analytic estimates of Mikado flows in the following lemma,
see also \cite{bv19r} for the case where $d=3$.

\begin{lemma} [Estimates of Mikado flows] \label{buildingblockestlemma}
	For $N \geq 0$ and $k\in \Lambda$, we have
	\begin{align}
	&\left\| \phi_{(k)}\right\|_{C^N_{x}}+ \lbb^2 \left\| \Phi_{(k)}\right\|_{C^N_{x}} \lesssim \lambda^N,\label{est-phi-cn}\\
	&\left\| W_{(k)}\right\|_{C^N_{x}}+\lambda \left\| W_{(k)}^c\right\|_{C^N_{x}}\lesssim \lambda^N , \label{ew}
	\end{align}
	where the implicit constants are independent of $\lambda$.
\end{lemma}

\begin{remark}
The Mikado flows used here are constructed in d-dimensional space where $d\geq 2$.
In the 2D case, the supports of Mikado flows are thin planes in each periodic domains,
and so the interaction between different flows can not be avoided
even if the support of cuf-off function $\phi$ is very small.
This is different from the 3D case in \cite{bv19r},
where different Mikado flows have no interactions by choosing suitable spatial shifts.

Let us also mention that,
the Mikado flows in \eqref{snwd} provide no intermittent effect,
and thus can not control the viscosity $(-\Delta)^\alpha$ in \eqref{equa-nsr}.
The keypoint in the control of viscosity would be
to explore the extra temporal intermittency,
which is the content of the next subsection.
\end{remark}

\subsection{Temporal building blocks}

In order to control the viscosity term $(-\Delta)^\alpha$ in \eqref{equa-nsr},
we proceed to construct another type of building blocks
with the temporal intermittency.
The construction below is inspired by the recent progress on temporal
intermittency in \cite{cl20.2,cl21,cl21.2},
which permits to obtain sharp non-uniqueness results.
However, unlike these works,
because of spatial interactions of Mikado flows,
suitable shifts should also be taken into account in the temporal building blocks,
so that the supports of different temporal building blocks are disjoint.

More precisely, let $\{g_k\}_{k\in \Lambda}\subset C_c^\infty([0,T])$
be cut-off functions such that $g_k$ and $g_{k'}$ have disjoint temporal supports if $k\neq k'$,
and
\begin{align*}
	\aint_{0}^T g_k^2(t) \d t=1
\end{align*}
for all $k\in \Lambda$. Since there are finitely many wavevectors in $\Lambda$,
the existence of such $\{g_k\}_{k\in \Lambda}$ can be guaranteed by choosing,
e.g., $g_k=g(t-\alpha_k)$,
where $g\in  C_c^\infty([0,T])$
with very small support and
$\{\alpha_k\}_{k\in \Lambda}$ are the temporal shifts
such that the supports of $\{g_k\}$ are disjoint.

Then, for each $k\in\Lambda$, we rescale $g_k$ by
\begin{align}\label{gk1}
	g_{k,\tau}(t)=\tau^{\frac 12} g_k(\tau t),
\end{align}
where the concentration parameter $\tau$ is given by \eqref{larsrp}.
Note that, $g_{k,\tau}$ indeed concentrates on a small interval of length $\sim \tau^{-1}$.
By an abuse of notation, we periodic $g_{k,\tau}$ such that it is treated as a periodize function defined on $[0,T]$.

Furthermore, we set
\begin{align} \label{hk}
	h_{k,\tau}(t):= \int_{0}^t \left(g_{k,\tau}^2(s)  - 1\right)\ ds,\ \ t\in [0,T],
\end{align}
and
\begin{align}\label{gk}
	\g(t):=g_{k,\tau}(\sigma t),\ \
	h_{(k)}(t):= h_{k,\tau}(\sigma t).
\end{align}
Then, we obtain
\begin{align}   \label{pt-h-gt}
	\p_t(\sigma^{-1} h_{(k)}) = \g^2-1=\g^2-\aint_{0}^T \g^2(t) \d t,
\end{align}
where $\sigma$ is given by \eqref{larsrp}.

The disjointness of the temporal supports
enables us to eliminate the interactions between different building blocks,
which is particularly useful in the 2D case
where the spatial interactions of homogeneous Mikado flows
cannot be avoided.
Another nice feature is that,
the large parameter $\sigma$ permits to balance high temporal oscillation errors
in the construction of the new Reynolds stress later,
see \eqref{I2-esti-endpt1} below.

The crucial temporal intermittent estimates of $\g$ and $h_{(k)}$ are summarized in Lemma \ref{Lem-gk-esti} below.
The proof is similar to the unshifted case in \cite{cl21}.

\begin{lemma} [Estimates of temporal intermittency]   \label{Lem-gk-esti}
	For  $p \in [1,\infty]$, $M  \in \mathbb{N}$,
	we have
	\begin{align}
		\left\|\partial_{t}^{M}\g \right\|_{L^{p}_t} \lesssim \sigma^{M}\tau^{M+\frac12-\frac{1}{p}},\label{gk estimate}
	\end{align}
	where the implicit constants are independent of $\sigma$ and $\tau$.
	Moreover, we have the uniform bound
	\begin{align}\label{hk-esti}
	\|h_{(k)}\|_{C_t}\leq 1.
	\end{align}
\end{lemma}

\subsection{Momentum perturbation}  \label{Subsec-Velo-perturb}

We are now in the stage to construct the momentum perturbation,
which mainly consists of the amplitudes,
the principal perturbation, the incompressible corrector and the temporal corrector.
A set of algebraic identities and analytic estimates
will also be given,
which are important to run the convex integration iterative scheme.

To begin with, let us first construct the amplitudes,
which play an important role in the cancellation
between the zero frequency part of the nonlinearity and the old Reynolds stress.
\medskip

\paragraph{\bf Amplitudes}
Let $\chi: [0, \infty) \to \mathbb{R}$ be a smooth cut-off function satisfying
\begin{equation}\label{e4.0}
	\chi (z) =
	\left\{\aligned
	& 1,\quad 0 \leq z\leq 1, \\
	& z,\quad z \geq 2,
	\endaligned
	\right.
\end{equation}
and
\begin{equation}\label{e4.1}
	\frac 12 z \leq \chi(z) \leq 2z \quad \text{for}\quad z \in (1,2).
\end{equation}

Let
\begin{equation}\label{def-varrho}
	\varrho(t) := 2 \varepsilon_u^{-1} \delta_{q+ 1}
	\chi\left( \frac{ \| R_\ell(t,\cdot) \|_{C_x}}{\delta_{q+1} } \right),
\end{equation}
where $\varepsilon_u>0$ is the small constant as in the Geometric Lemma \ref{geometric lem 2}.

Unlike in the contexts of INS
(see, e.g., \cite{bv19b,lt20,cl20.2,lqzz22}),
the function $\varrho$ is spatially independent.
This feature in particular enables to
treat the large amplitude term $\varrho f{\rm Id}$
arising from the cancellation between the nonlinearity of perturbations
and the Reynolds stress
(see identity \eqref{oscillation cancellation calculation} below).

Using \eqref{rc1}, \eqref{rl1}, \eqref{est-moll-rl1}, \eqref{def-varrho} and
the standard H\"{o}lder estimate ,
we derive that for $1\leq N\leq 4$,
\begin{align}
&	\left|  \frac{\rl(t,x)}{\varrho(t)} \right|  \leq \ve_u,\quad \varrho(t)\geq \ve_u^{-1}\delta_{q+ 1},
   \ \ \forall t\in [0,T],\ x\in \mathbb{T}^d,  \label{rl-varho} \\
&   \norm{ \varrho }_{L^p_{t}} \leq 4\ve_u^{-1}(T^{1/p}\delta_{q+1}+\|R_\ell\|_{L^p_tC_x} ),\label{varrhobd}\\
& \norm{ \varrho }_{C_{t}} \lesssim  \ell^{-1} , \quad   \norm{ \varrho }_{C_{t}^N}  \lesssim \ell^{-2N}, \label{rhoB-Ctx.1}\\
&\norm{ \varrho^{1/2}}_{C_{t}} \lesssim \ell^{-1}, \quad  \norm{  \varrho^{1/2} }_{C_{t}^N} \lesssim \ell^{-2N},  \label{rhoB-Ctx.2} \\
&\norm{ \varrho^{-1}}_{C_{t}} \lesssim \ve_u\delta_{q+1}^{-1}, \quad\norm{ \varrho^{-1} }_{C_{t}^N } \lesssim \ell^{-2N},  \label{rhoB-Ctx.3}
\end{align}
where the implicit constants are independent of $q$.

In order to ensure that the temporal support of perturbation
is compatible with that of the mollified Reynolds stress $R_\ell$,
we choose the smooth temporal cut-off function $f: [0,T]\rightarrow [0,1]$
such that
\begin{itemize}
	\item $0\leq f\leq 1$ and $f \equiv 1$ on $\supp_t R_\ell$;
	\item $\supp_t f\subseteq N_{\ell}(\supp_t R_\ell )$;
	\item $\|f \|_{C_t^N}\lesssim \ell^{-N}$,\ \  $1\leq N\leq 6$.
\end{itemize}

We now define the amplitudes of the momentum perturbations by
\begin{equation}\label{akb}
	a_{(k)}(t,x):= f (t) \varrho^{\frac{1}{2} } (t) \rho_\ell^{\frac12}(t,x)\gamma_{(k)}
       \left(\Id-\frac{\rl(t,x)}{\varrho(t)}\right) , \quad k \in \Lambda,
\end{equation}
where $\gamma_{(k)}$ and $\Lambda$ are given in the Geometric Lemma~\ref{geometric lem 2},
and $\rho_\ell$ is the mollified density given by \eqref{mol}.

In view of the Geometric Lemma~\ref{geometric lem 2} and the construction of $a_{(k)}$,
the following identity holds:
\begin{align}\label{velcancel}
	\sum\limits_{ k \in  \Lambda} \rho_\ell^{-1}a_{(k)}^2 \fint_{\T^3}
	W_{(k)} \otimes W_{(k)} dx
	= & \varrho f^2 {\rm Id} -\rl.
\end{align}

Moreover, the analytic estimates of amplitudes are collected in Lemma \ref{mae-endpt1} below,
which can be proved in a similar fashion as that of \cite[Lemma~4.1]{lzz21},
and so the proof is omitted here.

\begin{lemma} [Estimates of amplitudes] \label{mae-endpt1}
For $1\leq N\leq 6$, $k\in \Lambda$, we have
	\begin{align}
		\label{e3.15}
		&\norm{a_{(k)}}_{L^2_{t}C_x} \lesssim \delta_{q+1}^{\frac{1}{2}} ,\\
		\label{mag amp estimates}
		& \norm{ a_{(k)} }_{C_{t,x}} \lesssim \ell^{-1},\ \ \norm{ a_{(k)} }_{C_{t,x}^N} \lesssim \ell^{-5N},
	\end{align}
where the implicit constants are independent of $q$.
\end{lemma}
\medskip

\paragraph{\bf Principal perturbation}
The principal part of the momentum perturbation is defined by
\begin{align}  \label{pv}
		w_{q+1}^{(p)} &:= \sum_{k \in \Lambda } a_{(k)}\g W_{(k)},
\end{align}
where  $a_{(k)}$ is the amplitude defined in \eqref{akb},
$\g$ is the temporal building block given by \eqref{gk},
and  $ W_{(k)}$ is the Mikado flow constructed in \eqref{snwd},
$k\in \Lambda$.

In view of \eqref{velcancel}
and the disjointness of the supports of $\{g_{(k)}\}$,
the following algebraic identity holds
\begin{align} \label{oscillation cancellation calculation}
	\rho_\ell^{-1} w_{q+ 1}^{(p)} \otimes w_{q+ 1}^{(p)} + \rl
	=  \varrho f^2 {\rm Id}+ \sum_{k \in \Lambda }\rho_\ell^{-1}a_{(k)}^2 \g^2 \P_{\neq 0}(W_{(k)}\otimes W_{(k)})
	 + \sum_{k \in \Lambda }\rho_\ell^{-1} a_{(k)}^2 (\g^2-1)k\otimes k,
\end{align}
where $\P_{\neq 0}$ denotes the spatial projection onto nonzero Fourier modes.

Identity \eqref{oscillation cancellation calculation} shows that the old Reynolds stress
can be reduced by the zero frequency part of nonlinearity,
resulting in the remaining terms being two high spatial and temporal errors and
an extra term
$\varrho f^2 {\rm Id}$.
Note that,
in view of \eqref{rl-varho},
$|\varrho|$ is no less than $|\rl|$,
up to some constant,
the term $\varrho f^2{\rm Id}$ is actually large and
would destroy the inductive estimate \eqref{rl1}
for the Reynolds stress at level $q+1$.
This does not cause trouble in the incompressible case,
as its divergence can be absorbed into the pressure term
(or, by using the Leray projection).
But this strategy fails in the compressible case,
as the pressure depends on the density.
The key idea here is to make use of the time-independence
of  $\varrho$
to eliminate this trouble term by taking the divergence operator.
\medskip

\paragraph{\bf Incompressible corrector}
We also introduce the incompressible corrector, defined by
	\begin{align}  \label{wqc-dqc}
		w_{q+1}^{(c)}
		&:=   \sum_{k\in \Lambda }\g W^c_{(k)} \nabla a_{(k)}  ,
	\end{align}
where $W^c_{(k)}$ is the potential given by  \eqref{potential vector}, $k\in \Lambda$.
Note that, by \eqref{wcwc},
	\begin{align} \label{div free velocity}
		&  w_{q+1}^{(p)} + w_{q+1}^{(c)}
		=\div\left(  \sum_{k \in \Lambda} a_{(k)} \g W^c_{(k)} \right).
	\end{align}
Due to the fact that $a_{(k)} \g W^c_{(k)}$ is skew-symmetric for all $k\in \Lambda$, one has
\begin{align} \label{div-wpc-dpc-0}
	\div (w_{q+1}^{(p)} + w_{q +1}^{(c)})= 0,
\end{align}
which justifies the name of incompressibility.

\begin{remark}
We note that,
this kind of incompressible correctors is widely used in the convex integration schemes
for incompressible systems,
see for instance \cite{bv19r,bcv21}.
But for compressible systems,
it seems not necessary to keep the perturbation incompressible.
However, in this paper,
we still use this corrector to keep part of the perturbations incompressible,
which would permit to introduce a less variational density in our scheme.
In fact, our total perturbation for the momentum is not incompressible,
see the following paragraph.
\end{remark}

\medskip

\paragraph{\bf Temporal corrector}
The last part of momentum building blocks is the temporal corrector,
defined by
\begin{align}  \label{wo}
		& w_{q+1}^{(o)}:= -\sigma^{-1} \sum_{k\in\Lambda}h_{(k)}  k\otimes k\nabla(\rho_\ell^{-1}a_{(k)}^2 ).
\end{align}
By virtue of \eqref{hk}, \eqref{wo} and the Leibnitz rule,
we get  the identity
\begin{align} \label{utemcom}
& \partial_{t} w_{q+1}^{(o)}+
\sum_{k \in \Lambda }\div\(\rho_\ell^{-1}a_{(k)}^2 (\g^2-1)k\otimes k\)
  = -\sigma^{-1} \sum_{k\in\Lambda} h_{(k)}   k\otimes k \p_t\nabla(\rho_\ell^{-1}a_{(k)}^2 ).
\end{align}

\begin{remark}
In view of \eqref{utemcom},
the temporal corrector permits to balance the high temporal frequency error
in \eqref{oscillation cancellation calculation},
resulting in the low-frequency term $\partial_t \na (\rho_\ell^{-1} a^2_{(k)})$.
\end{remark}

\begin{remark} \label{incom-3.7}
For the case of imcompressible  hypo-viscous NSE \eqref{equa-NSE-Incomp},
the amplitudes of velocity perturbations are defined by
\begin{equation}\label{akb-1}
	a_{(k)}(t,x):= f (t) \varrho^{\frac{1}{2} } (t)\gamma_{(k)}
	\left(\Id-\frac{\mathring\rl(t,x)}{\varrho(t)}\right) , \quad k \in \Lambda,
\end{equation}
where $\mathring R := R -\frac{1}{d}tr(R)\Id$.

Moreover, since the relaxed solutions $\{u_q\}_{q\in \mathbb{N}}$ shall be divergence free,
the temporal corrector for the incompressible case is defined by
\begin{align}  \label{wo-1}
	& w_{q+1}^{(o)}:= -\sigma^{-1} \sum_{k\in\Lambda}\P_{H}\P_{\neq0}\(h_{(k)}  k\otimes k\nabla(a_{(k)}^2 ) \),
\end{align}
where $\P_{H}$ is the Helmholtz-Leray projector, i.e.,
$\P_{H}=\Id-\nabla\Delta^{-1}\div$ and $\P_{\neq 0}$ denotes the spatial projection onto nonzero Fourier modes.
Then, similarly to \eqref{utemcom},
the following identity holds
\begin{align} \label{utemcom-2}
	\partial_{t} w_{q+1}^{(o)}+
	\sum_{k \in \Lambda }\div\(a_{(k)}^2 (\g^2-1)k\otimes k\)
	& = \left(\nabla\Delta^{-1}\div\right) \sigma^{-1}  \sum_{k \in \Lambda} \P_{\neq 0} \partial_{t}\(h_{(k)} k\otimes k \nabla (a_{(k)}^2)\)\notag\\
	&\quad -\sigma^{-1} \sum_{k\in\Lambda}\P_{\neq 0} \(h_{(k)}   k\otimes k \p_t\nabla(a_{(k)}^2 )\).
\end{align}
\end{remark}
\medskip

\paragraph{\bf Momentum perturbation}
Now, we define the momentum perturbation at level $q+1$  by
	\begin{align}
		w_{q+1} &:= w_{q+1}^{(p)} + w_{q+1}^{(c)}+\wo.
		\label{velocity perturbation}
	\end{align}
Note that, by the constructions above,
$w_{q+1}$ is mean-free but not divergence-free.

The new momentum field $m_{q+1}$ at level $q+1$ is then defined by
	\begin{align}
		& m_{q+1}:= m_\ell + w_{q+1},
		\label{q+1 velocity}
	\end{align}
where $m_\ell$ is the previous mollified momentum defined in \eqref{mol}.
\medskip

\paragraph{\bf Density perturbation}
Because the momentum perturbation $w_{q+1}$ is not divergence-free,
we should further construct the density perturbation
such that the transport equation in \eqref{equa-nsr} is valid at level $q+1$.

For this purpose, we define the density perturbation by
\begin{align}  \label{zq1-def}
\thq(t,x):= -\int_0^t \div w_{q+1}(s,x) \d s
=\sigma^{-1} \sum_{k\in \Lambda}  \int_0^t  \(h_{(k)} (k\cdot\nabla)^2(\rho_{\ell}^{-1} a_{(k)}^2)\)(s,x)   \d s.
\end{align}
Then, the new density at level $q+1$ is defined by
\begin{align}\label{def-rqq}
\rqq:=\rho_{\ell}+\thq.
\end{align}
Thus, we have
\begin{align}
\p_t \rqq+\div \mq= \p_t \thq+\div w_{q+1}=0,
\end{align}
which verifies the transport equation in \eqref{equa-nsr} at level $q+1$.
\medskip

\paragraph{\bf Estimates of momentum and density perturbations}
We summarize the key estimates of the momentum and density perturbations in Proposition \ref{Prop-totalest} below,
which will be frequently used in \S \ref{Subsec-induc-vel-mag} and \S \ref{Sec-Rey-Endpt1}.

\begin{proposition}  [Estimates of momentum and density perturbations] \label{Prop-totalest}
	For any $p \in [1,\infty]$ and integers $0\leq N\leq 4$,
the following estimates hold:
	\begin{align}
		&\norm{\na^N w_{q+1}^{(p)} }_{L^p_tC_x }  \lesssim \ell^{-1} \lbb^N \tau^{\frac12-\frac{1}{ p}},\label{uprinlp-endpt1}\\
		&\norm{\na^N w_{q+1}^{(c)} }_{L^p_tC_x   } \lesssim \ell^{-5}\lbb^{N-1}\tau^{\frac12-\frac{1}{p}}, \label{ucorlp-endpt1}\\
		&\norm{\na^N \wo }_{L^p_tC_x  }\lesssim \ell^{-5N-7}\sigma^{-1},\label{dcorlp-endpt1}\\
		&\norm{ w_{q+1}^{(p)} }_{C_{t,x}^N }  + \norm{ w_{q+1}^{(c)} }_{C_{t,x}^N }+\norm{ \wo }_{C_{t,x}^N }
		\lesssim \lambda^{2N +2}.    \label{principal c1 est}
	\end{align} 	
	Moreover, for $0\leq N\leq 4$, $0\leq M\leq 1$,
	\begin{align}\label{est-thq}
	\|\p_t^M\thq\|_{C_tC_x^N}\lesssim \ell^{-5N-11}\sigma^{-1}.
	\end{align}
The above implicit constants are independent of $q$.
\end{proposition}

\begin{proof}
	First,
	using \eqref{gk estimate}, \eqref{pv} and Lemmas \ref{buildingblockestlemma} and \ref{mae-endpt1}
	we have that for any $p \in (1,\infty)$,
	\begin{align*}
		\norm{\nabla^N w_{q+1}^{(p)} }_{L^p_tC_x }
		\lesssim&  \sum_{k \in \Lambda}
		\sum\limits_{N_1+N_2 = N}
		\|a_{(k)}\|_{C^{N_1}_{t,x}}\|\g\|_{L_t^p}
		\norm{ W_{(k)} }_{C^{N_2}_{x} } \notag \\
		\lesssim& \sum\limits_{N_1+N_2 = N} \ell^{-5N_1-1}\tau^{\frac12-\frac{1}{p}}\lbb^{N_2}\notag\\
		\lesssim&	 \ell^{-1}\lbb^N\tau^{\frac12-\frac{1}{p}},
	\end{align*}
where the last step is due to $\ell^{-5} \ll \lbb$.
Thus, \eqref{uprinlp-endpt1} is  verified.
	
	Moreover, by \eqref{b-beta-ve},
	\eqref{gk estimate}, \eqref{wqc-dqc} and
Lemmas \ref{buildingblockestlemma} and \ref{mae-endpt1},
	\begin{align*}
		\norm{\na^N w_{q+1}^{(c)} }_{L^p_tC_x}	\lesssim&\,
		\sum\limits_{k\in \Lambda } \sum_{N_1+N_2=N}\norm{ a_{(k)} }_{C_{t,x}^{N_1+1}}\|\g\|_{L^p_t}
        \norm{W^c_{(k)}}_{C^{N_2}_x } \nonumber   \\
		\lesssim & \,  \sum_{N_1+N_2=N}  \ell^{-5N_1-5}\tau^{\frac12-\frac{1}{p}}\lbb^{N_2-1}  \nonumber   \\
		\lesssim &\, \ell^{-5}\lambda^{N-1} \tau^{\frac12-\frac{1}{p}},
	\end{align*}
	which implies \eqref{ucorlp-endpt1}.
	
	We then estimate the temporal corrector $\wo$.
By \eqref{est-moll-rhobd} and \eqref{est-moll-rhoq},
\begin{align}   \label{narho-1-lbbq}
     \|\na^N \rho_\ell \|_{C_{t,x}} \lesssim \lbb_q^{\frac{N\ve}{4}}.
\end{align}
Then, using \eqref{hk-esti}, \eqref{wo} and Lemma \ref{mae-endpt1} we get
	\begin{align*}
		\norm{ \na^N \wo }_{L^p_tC_x }
		&\lesssim \sigma^{-1}\sum_{k \in \Lambda}\|h_{(k)}\|_{C_{t}} \|\nabla^{N+1} (\rho_\ell^{-1}a^2_{(k)})\|_{C_{t,x}} \notag\\
		&\lesssim \sigma^{-1}\sum_{k \in \Lambda} \|h_{(k)}\|_{C_{t}} \sum_{N_1+N_2=N+1}  \| \rho_\ell^{-1}\|_{C^{N_1}_{t,x}} \|a^2_{(k)}\|_{C^{N_2}_{t,x}}\notag\\
		&\lesssim  \ell^{-5N-7} \sigma^{-1}.
	\end{align*}
	
	Regarding the $C^N_{t,x}$-estimate of velocity perturbations,
	using Lemmas \ref{buildingblockestlemma}, \ref{Lem-gk-esti} and \ref{mae-endpt1} we get
		\begin{align} \label{wprincipal c1 est}
		\norm{ w_{q+1}^{(p)} }_{C_{t,x}^N }
		\lesssim&   \sum_{k \in \Lambda}
		\|a_{(k)}\|_{C_{t,x}^N }
		\sum_{0\leq N_1+N_2 \leq N} \norm{ \g}_{C_{t}^{N_1}}\norm{ W_{(k)} }_{C_{x}^{N_2}}  \notag \\
		\lesssim& \sum_{0\leq N_1+N_2 \leq N}
		\ell^{-5N-1} \sigma^{N_1} \tau^{N_1+\frac 12} \lbb^{N_{2}} \notag \\
		\lesssim& \lambda^{2N+2},
	\end{align}
	where we also used \eqref{b-beta-ve}, \eqref{e3.1}, \eqref{l-lbbq} and \eqref{larsrp} in the last step.
	
	Similarly, we have
	\begin{align} \label{uc c1 est}
	 \norm{ w_{q+1}^{(c)} }_{C_{t,x}^N }
		& \lesssim   \sum_{k\in \Lambda }
		\left\|\g  W^c_{(k)}\nabla a_{(k)} \right\|_{C_{t,x}^N } \notag \\
		& \lesssim   \sum_{k \in \Lambda }
		\|a_{(k)}\|_{C_{t,x}^{N+1}}
		\sum_{0\leq N_1+N_2 \leq N} \norm{ \g}_{C_{t}^{N_1}}\norm{ W^c_{(k)} }_{C_{x}^{N_2}}   \nonumber \\
		& \lesssim   \sum_{0\leq N_1+N_2 \leq N}
		\ell^{-5N-5} \sigma^{N_1} \tau^{N_1+\frac 12}
	\lbb^{N_2-1}  \notag \\
		& \lesssim  \lambda^{2N+1},
	\end{align}
and
	\begin{align} \label{wo c1 est.2}
	\norm{ \wo }_{C_{t,x}^N }
	& \lesssim \sigma^{-1} \sum_{k \in \Lambda} \|h_{(k)} \na (\rho_\ell^{-1}a_{(k)}^2) \|_{C^N_{t,x}} \notag \\
	& \lesssim \sigma^{-1} \sum_{k \in \Lambda} \|h_{(k)}\|_{C_{t}^{N}} \|\nabla (\rho_\ell^{-1}a^2_{(k)})\|_{C_{t,x}^{N}} \notag \\
	& \lesssim \sigma^{N-1} \tau^{N} \ell^{-5N-7}  \lesssim \lbb^{2N+1},
	\end{align}
where the last step was due to \eqref{b-beta-ve}, \eqref{e3.1}, \eqref{larsrp}
and the fact that
\begin{align*}
   \|h_{(k)}\|_{C_{t}^N} \lesssim \sigma^N \tau^N.
\end{align*}
	
	Thus, combining \eqref{wprincipal c1 est}-\eqref{wo c1 est.2} altogether and using \eqref{larsrp}
	we conclude that
	\begin{align} \label{utotalc1}
		& \norm{ w_{q+1}^{(p)} }_{C_{t,x}^N }  + \norm{ w_{q+1}^{(c)} }_{C_{t,x}^N }+\norm{ \wo }_{C_{t,x}^N }
		\lesssim \lbb^{2N +2}.
	\end{align}
	
	It remains to prove \eqref{est-thq}.
This can be done by using \eqref{b-beta-ve}, \eqref{larsrp} and Lemmas~ \ref{Lem-gk-esti} and \ref{mae-endpt1}:
		\begin{align} \label{ver-thq-est}
	\norm{\p_t^M\thq }_{C_tC_{x}^N }
		&\leq \sigma^{-1}\sum_{k\in \Lambda} \left\|\p_t^M\int_0^t  \(h_{(k)} (k\cdot\nabla)^2(\rho_{\ell}^{-1} a_{(k)}^2)\)ds\right\|_{C_tC_{x}^N }\notag\\
		& \lesssim \sigma^{-1} \sum_{k \in \Lambda} \|h_{(k)}\|_{C_t} \| \rho_\ell^{-1}a^2_{(k)}\|_{C_{t}C_x^{N+2}} \notag \\
		& \lesssim \sigma^{-1} \sum_{k \in \Lambda} \|h_{(k)}\|_{C_t}
		\sum_{1\leq N'\leq N+2}\| \rho_\ell^{-1} \|_{C_tC_{x}^{N+2-N'}}\| a^2_{(k)}\|_{C_tC_{x}^{N'}} \notag \\
		&\quad + \sigma^{-1} \| \rho_\ell^{-1} \|_{C_tC_{x}^{N+2}}\| a^2_{(k)}\|_{C_{t,x}}\notag\\
		& \lesssim \sigma^{-1}\sum_{0<N'\leq N+2} \ell^{-N-2+N'}\ell^{-5N'-1} + \sigma^{-1} \ell^{-N-4}\lesssim \ell^{-5N-11} \sigma^{-1}.
	\end{align}
Thus, \eqref{est-thq} is verified.
Therefore, the proof of Proposition~\ref{Prop-totalest} is complete.
\end{proof}

\subsection{Verification of inductive estimates for density and momentum perturbations}  \label{Subsec-induc-vel-mag}

We are now in stage to verify the inductive estimates \eqref{rhobd}-\eqref{mc1}
and \eqref{rhopre}-\eqref{m-L1tLace-conv}
for the density and momentum perturbations.

Let us first treat the density function $\rho_{q+1}$.
By \eqref{est-moll-rhobd}, \eqref{larsrp}, \eqref{def-rqq} and \eqref{est-thq},
\begin{align}  \label{rhoq1-bdd.1}
	\rqq \leq \rho_{\ell}+\|\thq\|_{C_{t,x}} \leq  C_2-\frac12\la^{-\beta}+\ell^{-11}\sigma^{-1}\leq C_2-\frac12\la^{-\beta}+\laq^{-14\ve}\leq C_2-\laq^{-\beta},
\end{align}
where we also used $\lbb_{q+1}^{-14\ve} + \lbb_{q+1}^{-\beta} \ll \frac12\lbb_q^{-\beta}$
in the last step,
due to \eqref{b-beta-ve}.
Similarly, we also have
\begin{align}  \label{rhoq1-bdd.2}
	\rqq \geq \rho_{\ell}-\|\thq\|_{C_{t,x}} \geq C_1+\frac12\la^{-\beta}-\ell^{-11}\sigma^{-1}\geq C_1+\frac12\la^{-\beta}-\laq^{-14\ve}\geq C_1+\laq^{-\beta}.
\end{align}
Thus,  \eqref{rhobd} is verified at level $q+1$.

By virtue of \eqref{zq1-def} and \eqref{def-rqq},
we also see that
\begin{align}  \label{rhoq1-mass}
\int_{\mathbb{T}^d}\rho_{q+1}(t,x)\d x= \int_{\mathbb{T}^d} \rho_\ell(t,x)\d x+ \int_{\mathbb{T}^d} z_{q+1}(t,x)\d x= \int_{\mathbb{T}^d} \rho_q(t,x)\d x,
\end{align}
for all $t\in [0,T]$,
which yields \eqref{rhopre} at level $q+1$.

Moreover, for $0\leq N\leq 4$, $0\leq M\leq 1$,
by \eqref{est-moll-rhoq} and \eqref{est-thq},
\begin{align}\label{rhoq1-cn}
\|\p_t^M \rqq\|_{C_tC^N_{x}}
 \lesssim& \|\p_t^M \rho_\ell\|_{C_tC^N_{x}}+\|\p_t^M \thq\|_{C_tC^N_{x}} \notag \\
 \lesssim& \lambda^{\frac{N\ve}{4}}_{q} +\ell^{-5N-11}\sigma^{-1}\lesssim \lambda_{q+1}^{\frac{N\ve}{4}}.
\end{align}
Thus, \eqref{rhoc1} at level $q+1$ is verified.

We also derive from  \eqref{est-moll-rl-rq} and \eqref{est-thq} that
\begin{align}
\|\rqq-\rho_q\|_{C_tC^1_x}& \lesssim \|\thq\|_{C_tC^1_x}+\|\rho_\ell-\rho_q\|_{C_tC^1_x}\notag\\
 &\lesssim \ell^{-16}\sigma^{-1}+\ell \la^\ve\leq \delta_{q+2}^{\frac12},
\end{align}
where the last step was due to  \eqref{b-beta-ve}, \eqref{e3.1} and \eqref{l-lbbq}.
This gives \eqref{rho-l9-conv}.

Now, we turn to the momentum field.
First,  by virtue of \eqref{mc1}, \eqref{q+1 velocity} and \eqref{principal c1 est},
we derive
\begin{align}\label{ver-mq1-cn}
		\norm{\mq}_{C^N_{t,x}}
         \leq&  \norm{  m_\ell}_{C^N_{t,x}}+ \norm{w_{q+1}}_{C^N_{t,x}} \notag \\
        \lesssim& \lambda_q^{2N+2}+\lambda_{q+1}^{2N+2} \lesssim \lambda_{q+1}^{2N+2},
\end{align}
which yields \eqref{mc1} with $q+1$ replacing $q$.

In order to prove the $L^2_tC_x$-decay estimate \eqref{m-L2t-conv},
we recall from \cite[Lemma 2.4]{cl21} (see also \cite[Lemma 3.7]{bv19b}) the key $L^p$-decorrelation.

\begin{lemma} [$L^p$-decorrelation, \cite{cl21} Lemma 2.4, \cite{bv19b} Lemma 3.7]   \label{Decorrelation1}
	Let $\sigma\in \mathbb{N}$ and $f,g:\mathbb{T}^N\rightarrow \R$ be smooth functions. Then for every $p\in[1,\infty]$,
	\begin{equation}\label{lpdecor}
		\big|\|fg(\sigma\cdot)\|_{L^p(\T^N)}-\|f\|_{L^p(\T^N)}\|g\|_{L^p(\T^N)} \big|
        \lesssim \sigma^{-\frac{1}{p}}\|f\|_{C^1(\T^N)}\|g\|_{L^p(\T^N)}.
	\end{equation}
\end{lemma}
Applying the $L^p$-decorrelation Lemma~\ref{Decorrelation1} with   $N=1$,
$f= \|a_{(k)}(t,\cdot)\|_{L^\infty}$,
$g = \g$, $\sigma = \lambda^{15\ve}$
and using \eqref{la}, \eqref{b-beta-ve},
\eqref{mag amp estimates}, \eqref{pv}
and Lemmas \ref{buildingblockestlemma}, \ref{Lem-gk-esti} and \ref{mae-endpt1}
we get
\begin{align}
	\label{Lp decorr vel}
	\norm{w^{(p)}_{q+1}}_{L^2_{t}C_x}
    & \lesssim \sum\limits_{k\in \Lambda} \big\| \|a_{(k)}\|_{C_x} g_{(k)} \big\|_{L^2_t} \|W_{(k)}\|_{C_x} \notag \\
	&\lesssim \sum\limits_{k\in \Lambda}
       \Big(\|a_{(k)}\|_{L^2_{t}C_x}\norm{ \g }_{L^2_{t}} \norm{\phi_{(k)}}_{C_{x}}  +\sigma^{-\frac12}\|a_{(k)}\|_{C^1_{t,x}}\norm{ \g }_{L^2_{t}}\norm{\phi_{(k)}}_{C_{x}} \Big) \notag\\
	&\lesssim  \delta_{q+1}^{\frac{1}{2}}+\ell^{-5}\lambda^{-7\ve}_{q+1}   \lesssim \delta_{q+1}^{\frac{1}{2}}.
\end{align}
Then,
taking into account \eqref{b-beta-ve} and Proposition~\ref{Prop-totalest}
we deduce
\begin{align}  \label{e3.41.1}
	\norm{w_{q+1}}_{L^2_{t}C_x} &\lesssim\norm{w_{q+1}^{(p)} }_{L^2_{t}C_x} + \norm{ w_{q+1}^{(c)} }_{L^2_{t}C_x} +\norm{ \wo }_{L^2_{t}C_x}\notag \\
	&\lesssim \delta_{q+1}^{\frac{1}{2}} +\ell^{-5}\lbb^{-1}+\ell^{-7}\sigma^{-1}\lesssim \delta_{q+1}^{\frac{1}{2}}.
\end{align}
Moreover, an application of Proposition \ref{Prop-totalest} again gives
\begin{align}  \label{wql1.1}
	\norm{w_{q+1}}_{L^1_tC_x} &\lesssim\norm{w_{q+1}^{(p)} }_{L^1_tC_x} + \norm{ w_{q+1}^{(c)} }_{L^1_tC_x}  +\norm{ \wo }_{L^1_tC_x}\notag \\
	&\lesssim \ell^{-1}\tau^{-\frac12} +\ell^{-5}\lbb^{-1}\tau^{-\frac12} +\ell^{-7}\sigma^{-1}\lesssim \delta_{q+2}^{\frac{1}{2}}.
\end{align}
Thus, taking into account \eqref{est-moll-mc1}, \eqref{est-moll-ml-mq} and
\eqref{q+1 velocity}
we lead to
\begin{align}
 \norm{m_{q+1} - \m}_{L^2_{t}C_x} 	
     \leq& \norm{ m_\ell-\m}_{L^2_{t}C_x} + \norm{m_{q+1} -m_\ell}_{L^2_{t}C_x} \nonumber   \\
	\lesssim&  \norm{  m_\ell-\m }_{C_{t,x}}+ \norm{w_{q+1}}_{L^2_{t}C_x}  \nonumber  \\
	\lesssim& \ell\la^4+\delta_{q+1}^{\frac12}
    \lesssim \delta_{q+1}^{\frac{1}{2}}, \label{ver-mml2}
\end{align}
and
\begin{align} \label{ver-mml1}
\norm{m_{q+1} - \m}_{L^1_tC_x} 	
\lesssim&   \norm{ m_\ell-\m}_{L^1_{t}C_x} + \norm{m_{q+1} -m_\ell}_{L^1_{t}C_x} \nonumber   \\	
\lesssim&  \|m_\ell-\m\|_{C_{t,x}}+ \norm{w_{q+1}}_{L^1_{t}C_x} \notag \\
\lesssim&  \ell\la^4+\delta_{q+2}^{\frac12}
\lesssim \delta_{q+2}^{\frac{1}{2}}.
\end{align}

Regarding the $L^p_tC^s_x$-estimate of the momentum perturbation,
using interpolation and Proposition~\ref{Prop-totalest} we get
\begin{align}\label{est-wp-lgce}
\norm{w_{q+1}^{(p)}}_{L^p_tC^s_x}
& \lesssim \norm{w_{q+1}^{(p)}}_{L^p_tC_x}^{1-s}\norm{w_{q+1}^{(p)}}_{L^p_tC^1_x}^s \notag\\
&\lesssim (\ell^{-1} \tau^{\frac12-\frac1p}  )^{1-s} (\ell^{-1}\lambda \tau^{\frac12-\frac1p} )^s\notag\\
&\lesssim \ell^{-1}\lambda^s \tau^{\frac12-\frac1p} \lesssim \lambda^{\a-\frac{2\a}{p} +s+\ve(\frac{10}{p}-4)},
\end{align}
and, similarly,
\begin{align}
	& \norm{w_{q+1}^{(c)}}_{L^p_tC^s_x} \lesssim \ell^{-1}\lambda^{s-1} \tau^{\frac12-\frac1p} \lesssim \lambda^{\a-1-\frac{2\a}{p} +s+\ve(\frac{10}{p}-4)},\label{est-wc-lgce}\\
	& \norm{w_{q+1}^{(o)}}_{L^p_tC^s_x} \lesssim \ell^{-7-5s}\sigma^{-1}\lesssim \ell^{-12}\lambda^{-15\ve}.\label{est-wo-lgce}
\end{align}
Taking into account \eqref{est-moll-ml-mq} and \eqref{q+1 velocity}
we get
\begin{align}\label{est-mq1-mq-lgce}
\norm{m_{q+1} - \m}_{L^p_tC^s_x}
& \lesssim  \norm{  m_\ell-\m }_{L^p_tC^s_x}+ \norm{w_{q+1}}_{L^p_tC^s_x} \notag\\
&\lesssim \norm{  m_\ell-\m }_{L^p_tC_x}^{1-s} \norm{  m_\ell-\m }_{L^p_tC^1_x}^s+ \norm{w_{q+1}}_{L^p_tC^s_x} \notag\\
&\lesssim (\ell \la^4)^{1-s}(\ell \la^6)^s+ \lambda^{\a-\frac{2\a}{p} +s+\ve(\frac{10}{p}-4)}+\ell^{-12}\lambda^{-15\ve} \notag\\
&\lesssim \ell \la^{4+2s}+ \lambda^{\a-\frac{2\a}{p} +s+\ve(\frac{10}{p}-4)}+\ell^{-12}\lambda^{-15\ve},
\end{align}
which along with \eqref{b-beta-ve} and \eqref{e3.1} yields that
\begin{align}\label{mq1-mq-lgce-end}
	\norm{m_{q+1} - \m}_{L^p_tC^s_x}
     \lesssim \delta_{q+2}^{\frac12}.
\end{align}

Finally, concerning the $C_tH^{-1}_x$-estimate of the momentum,
by Lemmas~\ref{buildingblockestlemma}, \ref{mae-endpt1}, \eqref{mc1} \eqref{div free velocity} and the standard mollification estimates,
\begin{align}
\norm{m_{q+1} - \m}_{C_tH^{-1}_x}
& \lesssim  \norm{  m_\ell-\m }_{C_tH^{-1}_x}+ \norm{w_{q+1}}_{C_tH^{-1}_x} \notag\\
&\lesssim \norm{  m_\ell-\m }_{C_{t,x}} + \norm{w_{q+1}^{(p)}+w_{q+1}^{(c)}}_{C_tH^{-1}_x}+\norm{w_{q+1}^{(o)}}_{C_tL^2_x} \notag\\
&\lesssim \ell \|m_q\|_{C_{t,x}^1}+\sum_{k \in \Lambda}\| \g a_{(k)} W^c_{(k)}\|_{C_tL^2_x}+ \norm{w_{q+1}^{(o)}}_{C_tL^2_x} \notag\\
&\lesssim \ell \la^{4}+ \sum_{k \in \Lambda}\|a_{(k)}\|_{C_{t,x}} \| \g\|_{C_t} \|W^c_{(k)}\|_{L^2_x}+ \ell^{-7}\sigma^{-1}\notag\\
&\lesssim \la^{-26}+ \ell^{-1}\lambda^{\a-1} +\la^{-14\ve}\lesssim \delta_{q+2}^{\frac12},
\end{align}
where in the last inequality we also used \eqref{b-beta-ve} and \eqref{e3.1}.

Therefore, the iterative estimates \eqref{rhobd}-\eqref{mc1}
and \eqref{rhopre}-\eqref{m-L1tLace-conv} are verified
at level $q+1$.

\begin{remark}\label{rem-3.10}
In the incompressible case,
the density perturbation $z_{q+1}$ is no longer needed,
and the choice of $a_{(k)}$ and $w_{q+1}^{(o)}$ given by \eqref{akb-1} and \eqref{wo-1},
respectively,
does not change the intermittent estimates \eqref{uprinlp-endpt1}-\eqref{principal c1 est}.
Thus, the iterative estimates \eqref{mc1}, \eqref{m-L2t-conv} and \eqref{m-L1tLace-conv} can be verified
in a similar manner as above.
\end{remark}

\section{Reynolds stress}   \label{Sec-Rey-Endpt1}

The main purpose of this section is to
treat the delicate Reynolds stress and
verify the corresponding inductive estimates \eqref{rc1} and \eqref{rl1} at level $q+1$.

Unlike the incompressible case,
because of the  presence of the density,
the analysis of Reynolds stress for CNS is rather involving.
The corresponding pressure term should also be treated.
All these would require specific analysis of the density of fluid.

One key ingredient in the construction is  the inverse-divergence operator $\mathcal{R}$ introduced in \cite{cl20.2}
\begin{align} \label{calR-def}
	&(\mathcal{R} v)_{i j}=\mathcal{R}_{i j k} v_k,
\end{align}	
where
\begin{align*}
	&\mathcal{R}_{i j k}=\frac{2-d}{d-1} \Delta^{-2} \partial_i \partial_j \partial_k+\frac{-1}{d-1} \delta_{i j}\Delta^{-1} \partial_k +\Delta^{-1} \delta_{j k}\partial_i +\Delta^{-1}\delta_{i k} \partial_j ,
\end{align*}
and $v\in C^\9(\T^d)$.

The  inverse-divergence operator $\mathcal{R}$ maps smooth functions to symmetric and trace-free matrices
and satisfies the algebraic identity
\begin{align}\label{r-iden}
	\div \mathcal{R}(v) = v
\end{align}
for mean free smooth function $v$. See \cite{cl20.2} for more details.
We note that $\mathcal{R} \nabla$ and $\mathcal{R} \div$ are Calder\'{o}n-Zygmund operators,
which are bounded operators on periodic H\"older spaces (see, e.g., \cite{cz54}).

\subsection{Decomposition of Reynolds stress}
By virtue of \eqref{q+1 velocity} and \eqref{equa-me}, we derive that
\begin{align}  \label{ru}
	 \div {R}_{q+1}
		&  =\partial_t (w_{q+1}^{(p)}+w_{q+1}^{(c)}) +\mu(-\Delta)^{\alpha}\( \rho_\ell^{-1}w_{q+1} +( \rho_{q+1}^{-1}- \rho_\ell^{-1})m_\ell+ ( \rho_{q+1}^{-1}- \rho_\ell^{-1})w_{q+1}\) \notag\\
		& \quad-(\mu+\nu)\nabla\div\( \rho_\ell^{-1}w_{q+1} +(\rho_{q+1}^{-1}-\rho_\ell^{-1})m_\ell+ (\rho_{q+1}^{-1}-\rho_\ell^{-1})w_{q+1}\) \notag\\
		&\quad \underbrace{+\div\( \rho_{q+1}^{-1}(m_\ell\otimes w_{q+1}+w_{q+1}\otimes m_\ell)+( \rho_{q+1}^{-1}-\rho_\ell^{-1}) m_\ell\otimes m_\ell \)\hspace{2.2cm} }_{ \div R_{lin}} \notag\\
		& \quad+ \underbrace{\div (( \rho_{q+1}^{-1}-\rho_\ell^{-1})w_{q+1}\otimes w_{q+1}+\rho_\ell^{-1}w_{q+1}^{(p)} \otimes w_{q+1}^{(p)}+  R_\ell)+\partial_t \wo}_{ \div R_{osc}} \notag\\
		& \quad +\underbrace{ \div\Big(\rho_{\ell}^{-1}(w_{q+1}^{(c)}+\wo)\otimes w_{q+1}+ \rho_{\ell}^{-1}w_{q+1}^{(p)} \otimes (w_{q+1}^{(c)}+\wo) \Big)}_{ \div R_{cor}}\notag\\
		&\quad +\underbrace{\nabla P(\rho_{q+1}) -\nabla P(\rho_\ell)}_{ \div R_{pre}}
        +\div R_{com},
\end{align}
where $R_{com}$ is the commutator  stress given by \eqref{def-rcom}.

Applying the inverse-divergence operator $\mathcal{R}$,
we may choose a new Reynolds stress at level $q+1$
composed of five parts
\begin{align}\label{rucom}
R_{q+1} := R_{lin} +  R_{osc}+R_{cor}+ R_{pre}+R_{com},
\end{align}
where the linear error
\begin{align}   \label{rup}
R_{lin} & :=\mathcal{R} \partial_t (w_{q+1}^{(p)}+w_{q+1}^{(c)}) +\mu\mathcal{R}(-\Delta)^{\alpha}\(\rho_\ell^{-1}w_{q+1} +( \rho_{q+1}^{-1}- \rho_\ell^{-1})m_\ell+ ( \rho_{q+1}^{-1}- \rho_\ell^{-1})w_{q+1}\) \notag\\
	& \quad-(\mu+\nu)\mathcal{R}\nabla\div\( \rho_\ell^{-1}w_{q+1} +(\rho_{q+1}^{-1}-\rho_\ell^{-1})m_\ell+ (\rho_{q+1}^{-1}-\rho_\ell^{-1})w_{q+1}\) \notag\\
	&\quad+\mathcal{R}\div\( \rho_{q+1}^{-1}(m_\ell\otimes w_{q+1}+w_{q+1}\otimes m_\ell)+( \rho_{q+1}^{-1}-\rho_\ell^{-1})m_\ell\otimes m_\ell \),
\end{align}
the oscillation error
\begin{align}\label{rou}
	R_{osc}
   :=&\ \mathcal{R}\div( ( \rho_{q+1}^{-1}-\rho_\ell^{-1})w_{q+1}\otimes w_{q+1})
   + \sum_{k \in \Lambda } \mathcal{R} \P_{\neq 0}\left(\g^2 \P_{\neq 0}(W_{(k)}\otimes W_{(k)})
   \nabla (\rho_{\ell}^{-1} a_{(k)}^2)\right) \notag\\
	&-\sigma^{-1} \sum_{k\in\Lambda} h_{(k)} \p_t\mathcal{R}((k\cdot\nabla)(\rho_\ell^{-1}a_{(k)}^2 )k),
\end{align}
 the corrector error
\begin{align}  \label{rup2}
	 {R}_{cor} &
	:= \mathcal{R}  \div\Big(\rho_{\ell}^{-1}(w_{q+1}^{(c)}+\wo)\otimes w_{q+1}+ \rho_{\ell}^{-1}w_{q+1}^{(p)} \otimes (w_{q+1}^{(c)}+\wo) \Big),
\end{align}
the pressure error
\begin{align}  \label{rpre}
	R_{pre} &
	:= \mathcal{R} \nabla ( P(\rho_{q+1}) -  P(\rho_\ell)),
\end{align}
and $R_{com}$ is the commutator error given by \eqref{def-rcom}.

In view of \eqref{ru},
the algebraic identities \eqref{oscillation cancellation calculation}, \eqref{utemcom}
and the fact that $\div W_{(k)} \otimes W_{(k)} =0$,
it holds that
\begin{align} \label{calRuPHdiv-Ru}
    R_{q+1}  = \mathcal{R}\div R_{q+1}.
\end{align}

\subsection{Verification of $C_{t,x}^1$-estimate of Reynolds stress}
Below we verify the $C_{t, x}^1$-estimate \eqref{rc1} of $\rr_{q+1}$.
By the decomposition identity \eqref{calRuPHdiv-Ru},
Sobolev's embedding $W^{1,d+1}_x\hookrightarrow C_x$,
equations \eqref{equa-nsr} at level $q+1$ and
the fact that $\mathcal{R}\sim |\na|^{-1}$ in $W^{s,p}$ spaces,
\begin{align} \label{Rq+1-Ctx}
	\norm{ R_{q+1}}_{C_tC^1_x}
	& \lesssim \norm{\mathcal{R}  (\div R_{q+1})}_{C_tW^{2,d+1}_x}\notag\\
	&\lesssim \norm{\partial_t m_{q+1}}_{C_tW^{1,d+1}_x}
          +\norm{\mu (-\Delta)^{\alpha}(\rho_{q+1}^{-1} m_{q+1})}_{C_tW^{1,d+1}_x}
          +\norm{(\mu+\nu)\nabla\div(\rho_{q+1}^{-1} m_{q+1})}_{C_tW^{1,d+1}_x} \notag \\
	 &\quad+\norm{\div(\rho_{q+1}^{-1}m_{q+1}\otimes m_{q+1} )}_{C_tW^{1,d+1}_x} +\norm{\nabla P(\rho_{q+1}) }_{C_tW^{1,d+1}_x}.
\end{align}
Note that, by the interpolation,
\begin{align} \label{Delta-rhom-esti}
	   \norm{\mu (-\Delta)^{\alpha}(\rho_{q+1}^{-1} m_{q+1})}_{C_tW^{1,d+1}_x}
\lesssim&   \norm{\rho_{q+1}^{-1}m_{q+1}}_{C_tC_x}^{\frac{3-2\a}{4}} \norm{\rho_{q+1}^{-1}m_{q+1}}_{C_tW^{4,\9}_x}^{\frac{2\a+1}{4}}  \notag \\
\lesssim&  \norm{m_{q+1}}_{C_{t,x}}^{\frac{3-2\a}{4}} (\norm{\rho_{q+1}^{-1}}_{C_tC^4_{x}}\norm{ m_{q+1}}_{C^4_{t,x}})^{\frac{2\a+1}{4}}.
\end{align}
Moreover,
\begin{align}
   \norm{(\mu+\nu)\nabla\div(\rho_{q+1}^{-1} m_{q+1})}_{C_tW^{1,d+1}_x}
   \lesssim  \norm{\rho_{q+1}^{-1}m_{q+1}}_{C_tC^3_x}
   \lesssim \norm{\rho_{q+1}^{-1} }_{C_tC^3_x}\norm{ m_{q+1}}_{C_tC^3_x},
\end{align}
and
\begin{align} \label{div-rhom}
	\norm{\div(\rho_{q+1}^{-1}m_{q+1}\otimes m_{q+1} )}_{C_tW^{1,d+1}_x}
    \lesssim& \norm{\rho_{q+1}^{-1}m_{q+1}\otimes m_{q+1}}_{C_tC^{2}_x} \notag \\
    \lesssim& \norm{\rho_{q+1}^{-1} }_{C_tC^2_x}\sum_{N_1+N_2=2}\norm{m_{q+1}}_{C_tC^{N_1}_x}\norm{ m_{q+1}}_{C_tC^{N_2}_x}.
\end{align}
Thus, plugging \eqref{Delta-rhom-esti}-\eqref{div-rhom} into \eqref{Rq+1-Ctx}
we obtain
\begin{align}   \label{ver-rc1.1}
   \norm{ R_{q+1}}_{C_tC^1_x}
   \lesssim& \norm{\partial_t m_{q+1}}_{C_tW^{1,d+1}_x}
            + \norm{m_{q+1}}_{C_{t,x}}^{\frac{3-2\a}{4}} (\norm{\rho_{q+1}^{-1}}_{C_tC^4_{x}}\norm{ m_{q+1}}_{C^4_{t,x}})^{\frac{2\a+1}{4}} \notag \\
           & + \norm{\rho_{q+1}^{-1} }_{C_tC^3_x}\norm{ m_{q+1}}_{C_tC^3_x}
            + \norm{\rho_{q+1}^{-1} }_{C_tC^2_x}\sum_{N_1+N_2=2}\norm{m_{q+1}}_{C_tC^{N_1}_x}\norm{ m_{q+1}}_{C_tC^{N_2}_x} \notag \\
           & +\norm{\nabla P(\rho_{q+1}) }_{C_tC^1_x}.
\end{align}

Using \eqref{rhoq1-bdd.1}, \eqref{rhoq1-bdd.2} and \eqref{rhoq1-cn}, for $0\leq N\leq 4$ and $0\leq M\leq 1$, we have
\begin{align}\label{rhoq1-1-bd}
\|\p_t^{M}\rho^{-1}_{q+1}\|_{C_tC^N_{x}}\lesssim \lambda_{q+1}^{\frac{N\ve}{4}}.
\end{align}
Moreover,
\begin{align}\label{p-1-bd}
	\|\nabla P(\rho_{q+1})\|_{C_tC^1_x}
    \lesssim  \|P'(\rho_{q+1})\|_{C_{t,x}} \|\rho_{q+1}  \|_{C_tC_x^2}
     + \|P''(\rho_{q+1})\|_{C_{t,x}} \|\rho_{q+1}  \|_{C_tC_x^1}^2.
\end{align}
Since by \eqref{p-bound}, $P'$, $P''$ are continuous,
and by \eqref{rhoq1-bdd.1} and \eqref{rhoq1-bdd.2},
$\rho_{q+1}$ is uniformly bounded away from zero and infinity,
we have
\begin{align} \label{P'-P''-bdd}
   \|P'(\rho_{q+1})\|_{C_{t,x}}
   +  \|P''(\rho_{q+1})\|_{C_{t,x}}
   \lesssim 1,
\end{align}
where the implicit constant is independent of $q$.
Then, in view of \eqref{rhoq1-cn} we get
\begin{align} \label{naP-Ctx-esti}
   \|\na P(\rho_{q+1})\|_{C_tC_x^1} \lesssim \lbb_{q+1}^\ve.
\end{align}
Taking into account \eqref{ver-mq1-cn} and \eqref{rhoq1-cn} we get
\begin{align}  \label{Rq1-CtC1x}
\norm{R_{q+1}}_{C_tC^1_x}
&\lesssim  \lambda_{q+1}^{6}+ \lambda_{q+1}^{\frac{3-2\a}{2}}(\lambda_{q+1}^\ve \cdot \lambda_{q+1}^{10})^{\frac{2\a+1}{4}} +\lambda_{q+1}^{\ve} \cdot\lambda_{q+1}^{8}+ \lambda_{q+1}^{\ve}
\lesssim \lambda_{q+1}^{9}.
\end{align}

Similarly, we deduce
\begin{align}\label{ver-rc1.2}
	\norm{\partial_t R_{q+1}}_{C_{t,x}}
	&\lesssim \norm{\partial_t^2 m_{q+1}}_{C_tL^{\9}_x}+\norm{\mu\partial_t (-\Delta)^{\alpha}(\rho_{q+1}^{-1} m_{q+1})}_{C_tL^{\9}_x} +\norm{(\mu+\nu)\partial_t\nabla\div(\rho_{q+1}^{-1} m_{q+1})}_{C_tL^{\9}_x}\notag\\
	&\quad+\norm{\partial_t\div(\rho_{q+1}^{-1}m_{q+1}\otimes m_{q+1} )}_{C_tL^{\9}_x} +\norm{\partial_t\nabla P(\rho_{q+1}) }_{C_tL^{\9}_x}\notag\\
   &\lesssim \norm{ m_{q+1}}_{{C^2_{t,x}}}+ \norm{ (-\Delta)^{\alpha}(\p_t \rho_{q+1}^{-1} m_{q+1})}_{C_{t,x}}+ \norm{ (-\Delta)^{\alpha}(\rho_{q+1}^{-1}\p_t  m_{q+1})}_{C_{t,x}}\notag\\
  &\quad	+\sum_{M_1+M_2=1}\norm{\p_t^{M_1}\rho_{q+1}^{-1} }_{C_tC^2_{x}}\norm{ \p_t^{M_2}m_{q+1}}_{C_tC^2_{x}}+\norm{\rho_{q+1}^{-1} }_{C_tC^2_{x}}\sum_{N_1+N_2=2}\norm{m_{q+1}}_{ C^{N_1}_{t,x}}\norm{ m_{q+1}}_{ C^{N_2}_{t,x}} \notag\\
	&\quad+\|P'(\rho_{q+1})\|_{C_{t,x}}\|\p_t\nabla\rho_{q+1} \|_{C_{t,x}}+\|P''(\rho_{q+1})\|_{C_{t,x}}\|\nabla \rho_{q+1} \|_{C_{t,x}}\| \p_t\rho_{q+1} \|_{C_{t,x}}.
\end{align}
Concerning the shear viscosity error, using interpolation,
\eqref{ver-mq1-cn} and \eqref{rhoq1-1-bd} we get
\begin{align}\label{ver-sv-cn1}
\norm{ (-\Delta)^{\alpha}(\p_t \rho_{q+1}^{-1} m_{q+1})}_{C_{t,x}}& \lesssim \norm{\p_t \rho_{q+1}^{-1}m_{q+1}}_{C_{t,x}}^{1-\a} \norm{\p_t\rho_{q+1}^{-1}m_{q+1}}_{C_{x}C^2_x}^{\a} \notag\\
&\lesssim ( \norm{\p_t \rho_{q+1}^{-1}}_{C_{t,x}}\norm{m_{q+1}}_{C_{t,x}})^{1-\a} ( \norm{\p_t\rho_{q+1}^{-1}}_{C_{t}C^2_x}\norm{ m_{q+1}}_{C_{t}C^2_x})^{\a} \notag\\
&\lesssim (\lambda_{q+1}^{2+\ve} )^{1-\a} ( \lambda_{q+1}^{6+\ve} )^{\a} \lesssim \lambda_{q+1}^7.
\end{align}
In a similar manner, we estimate
\begin{align}\label{ver-sv-cn2}
	\norm{ (-\Delta)^{\alpha}(\rho_{q+1}^{-1} \p_t m_{q+1})}_{C_{t,x}}
	&\lesssim ( \norm{\rho_{q+1}^{-1}}_{C_{t,x}}\norm{\p_t m_{q+1}}_{C_{t,x}})^{1-\a} ( \norm{\rho_{q+1}^{-1}}_{C_{t}C^2_x}\norm{ \p_tm_{q+1}}_{C_{t}C^2_x})^{\a} \notag\\
	&\lesssim (\lambda_{q+1}^{4+\ve} )^{1-\a} ( \lambda_{q+1}^{8+\ve} )^{\a} \lesssim \lambda_{q+1}^9.
\end{align}
Thus, taking into account \eqref{rhoq1-cn}, \eqref{ver-mq1-cn} and \eqref{rhoq1-1-bd}
we conclude that
\begin{align} \label{ptRq1-Ctx}
	\norm{\p_t R_{q+1}}_{C_{t,x}}
	\lesssim& \lambda_{q+1}^6+ \lambda_{q+1}^{7}+ \lambda_{q+1}^{9} +\sum_{M_1+M_2=1}\laq^{\frac{3\ve}{4}} \laq^{6+2M_1}   \notag \\
      &+  \laq^{\frac{\ve}{2}} \sum_{N_1+N_2=2}\laq^{2N_1+2} \laq^{2N_2+2} + \laq^\ve \notag \\
    \lesssim&  \lambda_{q+1}^{9}  .
\end{align}
Therefore, it follows from \eqref{Rq1-CtC1x} and \eqref{ptRq1-Ctx}
that the $C_{t,x}^1$-estimate \eqref{rc1} is verified at level $q+1$.

\subsection{Verification of $L^1_tC_x$-decay of Reynolds stress}

We aim to prove the crucial $L^1_tC_x$-decay estimate \eqref{rl1} at level $q+1$.
Because Calder\'{o}n-Zygmund operators are bounded in H\"{o}lder
spaces $C^\ve_x$, where $0<\ve<1$ can be chosen as in \eqref{e3.1},
we would estimate the new Reynolds stress in the space $L^1_tC^\ve_x$
rather than $L^1_tC_x$.

Below we deal with the five parts
given by \eqref{rucom} of the Reynolds stress separately.
\medskip

\paragraph{\bf (i) Linear error $R_{lin}$.}
Recall from \eqref{rup} that
\begin{align}
R_{lin} & =\mathcal{R} \partial_t (w_{q+1}^{(p)}+w_{q+1}^{(c)}) +\mu\mathcal{R}(-\Delta)^{\alpha}\(\rho_\ell^{-1}w_{q+1} +( \rho_{q+1}^{-1}- \rho_\ell^{-1})m_\ell+ ( \rho_{q+1}^{-1}- \rho_\ell^{-1})w_{q+1}\) \notag\\
	& \quad-(\mu+\nu)\mathcal{R}\nabla\div\( \rho_\ell^{-1}w_{q+1} +(\rho_{q+1}^{-1}-\rho_\ell^{-1})m_\ell+ (\rho_{q+1}^{-1}-\rho_\ell^{-1})w_{q+1}\) \notag\\
	&\quad+\mathcal{R}\div\( \rho_{q+1}^{-1}(m_\ell\otimes w_{q+1}+w_{q+1}\otimes m_\ell)+( \rho_{q+1}^{-1}-\rho_\ell^{-1})m_\ell\otimes m_\ell \)    \notag \\
  &=: R_{lin,1} + R_{lin,2} + R_{lin,3} + R_{lin,4}.
\end{align}

\paragraph{\bf Estimate of $R_{lin,1}$}
First, in view of Lemmas \ref{buildingblockestlemma}, \ref{Lem-gk-esti} and \ref{mae-endpt1}, \eqref{e3.1},  \eqref{larsrp},
\eqref{div free velocity} and interpolation, we have
\begin{align}
	\|R_{lin,1}\|_{L_t^1C^\ve_x}
    \lesssim& \sum_{k \in \Lambda}\| \mathcal{R} \div\partial_t(\g a_{(k)} W^c_{(k)}) \|_{L_t^1C^\ve_x} \nonumber \\
	\lesssim& \sum_{k \in \Lambda} (\| \g\|_{L^1_t}\|  a_{(k)} \|_{C_{t,x}^2} +\| \p_t\g\|_{L_t^1}\| a_{(k)} \|_{C_{t,x}^1})\| W^c_{(k)} \|_{C^\ve_x} \nonumber \\
		\lesssim& \sum_{k \in \Lambda} (\| \g\|_{L^1_t}\|  a_{(k)} \|_{C_{t,x}^2}+\| \p_t\g\|_{L_t^1}\| a_{(k)} \|_{C_{t,x}^1})\|W^c_{(k)} \|_{C_{x}}^{1-\ve}\|  W^c_{(k)} \|_{C^1_x}^\ve \nonumber \\
    \lesssim& (\ell^{-10}\tau^{-\frac12}
        + \ell^{-5}\sigma\tau^{\frac12})\lambda^{-1+\ve}   \notag \\
	\lesssim& \ell^{-10}(\lambda^{-\a-1+6\ve }+\lambda^{\a-1+11\ve} )
     \lesssim \ell^{-10}\lambda^{-9\ve},    \label{time derivative}
\end{align}
where we also used $\alpha-1 \leq -20 \ve$, due to \eqref{e3.1}.

\paragraph{\bf Estimate of $R_{lin,2}$}
Regarding the shear viscosity error,
we consider the cases where $\a\in (0,1/2]$ and $\a\in (1/2,1)$ separately.
\medskip

{\bf The case where $\a\in (0,1/2]$.}
By  Proposition~\ref{Prop-totalest},
\begin{align}\label{est-lin-vis1}
 \|R_{lin,2}\|_{L_t^1C^\ve_x}
  \lesssim& \norm{  \rho_\ell^{-1}w_{q+1} +( \rho_{q+1}^{-1}- \rho_\ell^{-1}) m_\ell+ ( \rho_{q+1}^{-1}- \rho_\ell^{-1})w_{q+1} }_{L_t^1C^\ve_x}\notag\\
 \lesssim & \norm{  \rho_\ell^{-1}}_{C_tC^\ve_x}   \norm{  w_{q+1}}_{L_t^1C_x}
   +\norm{  \rho_\ell^{-1}}_{C_{t,x}} \norm{  w_{q+1}}_{L_t^1C^\ve_x}   \notag \\
   & +\norm{(\rho_{q+1}^{-1}-\rho_\ell^{-1})}_{C_tC^\ve_x}(\|m_\ell\|_{L_t^1C^\ve_x}+\|w_{q+1}\|_{L_t^1C^\ve_x} ).
\end{align}
Note that, by \eqref{def-rqq}, \eqref{est-thq}, \eqref{narho-1-lbbq} and \eqref{rhoq1-1-bd},
\begin{align}\label{est-rhoq1-rhol}
	\norm{( \rho_{q+1}^{-1}- \rho_\ell^{-1})}_{C_tC^N_{x}}
	\lesssim&  \norm{ \rho_{q+1}^{-1} }_{C_tC^N_{x}} \norm{ \rho_\ell^{-1}}_{C_tC^N_{x}}\norm{ \thq}_{C_tC^N_{x}} \notag \\
	\lesssim& \lambda^{\frac{N\ve}{4}}  \la^{\frac{N\ve}{4}} \ell^{-5N-11}\sigma^{-1} \leq \ell^{-5N-12}\lambda^{-14\ve} ,
\end{align}

Using \eqref{rhoq1-1-bd}, \eqref{narho-1-lbbq} and the interpolation inequality
we also get
\begin{align}
& \norm{ \rho_\ell^{-1}}_{C_tC^\ve_x} \lesssim \norm{ \rho_\ell^{-1}}_{C_{t,x}}^{1-\ve} \norm{ \rho_\ell^{-1}}_{C_tC^1_x}^\ve \lesssim  \la^{\frac{\ve^2}{4}},\label{e4.16}\\
& \norm{ \rho_{q+1}^{-1}}_{C_tC^\ve_x} \lesssim \norm{\rqe}_{C_{t,x}}^{1-\ve} \norm{\rqe}_{C_tC^1_x}^\ve\lesssim \lambda^{\frac{\ve^2}{4}},\label{e4.17}
\end{align}
and, via \eqref{est-moll-mc1},
\begin{align}
   \norm{ m_\ell }_{C_tC^\ve_x} \lesssim \norm{ m_\ell }_{C_{t,x}}^{1-\ve} \norm{ m_\ell }_{C_tC^1_x}^\ve \lesssim  \la^{4}.\label{est-ml-ve}
\end{align}
Moreover, by interpolation and \eqref{est-thq},
\begin{align}
\norm{ \thq}_{C_tC^\ve_x}
 \lesssim& \norm{\thq}_{C_{t,x}}^{1-\ve} \norm{\thq}_{C_tC^1_x}^\ve \notag \\
 \lesssim& \ell^{-5\ve-11}\sigma^{-1}\lesssim \ell^{-12}\lambda^{-15\ve},
\end{align}
and thus,
similarly to \eqref{est-rhoq1-rhol},
\begin{align}\label{est-rhoq1-rhol-cv}
\norm{( \rho_{q+1}^{-1}- \rho_\ell^{-1})}_{C_tC^\ve_x}
\lesssim&  \norm{ \rho_{q+1}^{-1} }_{C_tC^\ve_x} \norm{ \rho_\ell^{-1}}_{C_tC^\ve_x}\norm{ \thq}_{C_tC^\ve_x} \notag \\
\lesssim&  \la^{\frac{\ve^2}{4}} \lambda^{\frac{\ve^2}{4}} \ell^{-12}\lambda^{-15\ve} \leq \ell^{-12}\lambda^{-14\ve}.
\end{align}

Concerning the $L^1_tC_x^\ve$ estimate of $w_{q+1}$,
we have
\begin{align}\label{w-l1cve}
\|w_{q+1}\|_{L^1_tC_x^\ve} & \lesssim \|w_{q+1}^{(p)}\|_{L^1_tC_x^\ve}+ \|w_{q+1}^{(c)}\|_{L^1_tC_x^\ve}+ \|w_{q+1}^{(o)}\|_{L^1_tC_x^\ve}.
\end{align}
Using interpolation, Proposition~\ref{Prop-totalest},
\eqref{e3.1} and \eqref{larsrp} we get
\begin{align}\label{wp-l1cve}
\|w_{q+1}^{(p)}\|_{L^1_tC_x^\ve}\lesssim \|w_{q+1}^{(p)}\|_{L^1_tC_x}^{1-\ve}\|w_{q+1}^{(p)}\|_{L^1_tC_x^1}^\ve \lesssim \ell^{-1} \lambda^{\ve}\tau^{-\frac12} \lesssim \ell^{-1} \lambda^{-14\ve}.
\end{align}
Similarly, we have
\begin{align}
	\|w_{q+1}^{(c)}\|_{L^1_tC_x^\ve}\lesssim \ell^{-5} \lambda^{\ve-1}\tau^{-\frac12} \lesssim \ell^{-5} \lambda^{-1-14\ve},\label{wc-l1cve}\\
	\|w_{q+1}^{(o)}\|_{L^1_tC_x^\ve}\lesssim \ell^{-7-5\ve} \sigma^{-1} \lesssim \ell^{-12} \lambda^{-15\ve}.\label{wo-l1cve}
\end{align}
Hence, substituting \eqref{wp-l1cve}-\eqref{wo-l1cve} into \eqref{w-l1cve}
we obtain
\begin{align}\label{w-l1cve-end}
	\|w_{q+1}\|_{L^1_tC_x^\ve} & \lesssim \ell^{-1} \lambda^{-14\ve}+\ell^{-5} \lambda^{-1-14\ve}+\ell^{-12} \lambda^{-15\ve}\lesssim \ell^{-1} \lambda^{-14\ve}.
\end{align}

Thus, combining Proposition \ref{Prop-totalest},
\eqref{narho-1-lbbq}, \eqref{est-ml-ve}, \eqref{est-rhoq1-rhol-cv}, \eqref{w-l1cve-end}
together,
we conclude that for $\alpha \in (0,1/2]$,
\begin{align}\label{e4.19}
  \|R_{lin,2}\|_{L_t^1C^\ve_x}
 & \lesssim \lambda_{q}^{\frac{\ve^2}{4}} \ell^{-1}\tau^{-\frac12}+ \ell^{-1} \lambda^{-14\ve} +\ell^{-12}\lambda^{-14\ve}(\lambda_{q}^4+\ell^{-1} \lambda^{-14\ve})\notag\\
&\lesssim \ell^{-13}\lambda^{-14\ve},
\end{align}
where the last step is due to $\alpha\geq 20\ve$ and \eqref{larsrp}.

{\bf The case where $\a\in (1/2,1)$.}
Regarding the case where $\a\in (\frac12,1)$,
we first deduce
\begin{align}\label{est-lin-vis2}
	\|R_{lin,2}\|_{L_t^1C^\ve_x}
 \lesssim& \norm{ |\nabla|^{2\a-1}( \rho_\ell^{-1}w_{q+1})}_{L_t^1C^\ve_x} +\norm{ |\nabla|^{2\a-1}( ( \rho_{q+1}^{-1}- \rho_\ell^{-1}) m_\ell)}_{L_t^1C^\ve_x} \notag \\
         &+ \norm{ |\nabla|^{2\a-1} (( \rho_{q+1}^{-1}- \rho_\ell^{-1})w_{q+1}) }_{L_t^1C^\ve_x}.
\end{align}
For the first term on the right-hand side of \eqref{est-lin-vis2},
we use the interpolation inequality,
\eqref{b-beta-ve}, \eqref{larsrp}, \eqref{uprinlp-endpt1},
\eqref{narho-1-lbbq} and the fact that $1-\alpha \geq  20\va$ to derive
\begin{align}\label{est-lin-vis2-i1-p}
	\norm{ |\nabla|^{2\a-1} (\rho_\ell^{-1}w_{q+1}^{(p)})}_{L_t^1C^\ve_x}
	& \lesssim  \norm{\rho_\ell^{-1}w_{q+1}^{(p)}}_{L_t^1C_x} ^{1-\frac{2\a-1+\ve}{3}} \norm{\rho_\ell^{-1}w_{q+1}^{(p)}}_{L_t^1C^{3}_x} ^{\frac{2\a-1+\ve}{3}}\notag\\
	& \lesssim (\norm{\rho_\ell^{-1}}_{C_{t,x}}\norm{w_{q+1}^{(p)}}_{L_t^1C_x})^{1-\frac{2\a-1+\ve}{3}} ( \norm{\rho_\ell^{-1}}_{C_tC^{3}_x}\norm{w_{q+1}^{(p)}}_{L_t^1C^{3}_x} )^{\frac{2\a-1+\ve}{3} }\notag\\
	& \lesssim (\ell^{-1}\tau^{-\frac12})^{1-\frac{2\a-1+\ve}{3}} (\lambda_{q}^{\frac{3\ve}{4}}\ell^{-1}\lambda^3\tau^{-\frac12})^{\frac{2\a-1+\ve}{3}} \notag\\
	&\lesssim \ell^{-2} \tau^{-\frac12}\lambda^{2\a-1+\ve} \lesssim \ell^{-2} \lambda^{ -14\ve}.
\end{align}
Similarly, by Proposition \ref{Prop-totalest},
\begin{align}
\norm{ |\nabla|^{2\a-1} (\rho_\ell^{-1}w_{q+1}^{(c)})}_{L_t^1C^\ve_x}
	& \lesssim (\ell^{-5} \lambda^{-1}\tau^{-\frac12})^{1-\frac{2\a-1+\ve}{3}} (\lambda_{q}^{\frac{3\ve}{4}} \ell^{-5}\lambda^2\tau^{-\frac12})^{\frac{2\a-1+\ve}{3}}  \lesssim \ell^{-6}\lambda^{-1},\label{est-lin-vis2-i1-c}\\
\norm{ |\nabla|^{2\a-1} (\rho_\ell^{-1}w_{q+1}^{(o)})}_{L_t^1C^\ve_x}
      	&\lesssim(\ell^{-7} \sigma^{-1})^{1-\frac{2\a-1+\ve}{3}} (\lambda_{q}^{\frac{3\ve}{4}} \ell^{-22}\sigma^{-1})^{\frac{2\a-1+\ve}{3}}  \lesssim \ell^{-22 }\sigma^{-1}.\label{est-lin-vis2-i1-o}
\end{align}
Hence, combining \eqref{est-lin-vis2-i1-p}-\eqref{est-lin-vis2-i1-o} and the fact that $\ell^{-22}\ll \lambda^{\ve}$ altogether we obtain
\begin{align}  \label{est-lin-vis2-i1-end}
\norm{ |\nabla|^{2\a-1} (\rho_\ell^{-1}w_{q+1})}_{L_t^1C^\ve_x}\lesssim \ell^{-2} \lambda^{ -14\ve}+\ell^{-6}\lambda^{-1}+\ell^{-22 }\sigma^{-1}\lesssim \ell^{-2} \lambda^{ -14\ve}.
\end{align}

Concerning the second term on the right-hand side of \eqref{est-lin-vis2},
using \eqref{est-moll-mc1}, \eqref{est-rhoq1-rhol},
\eqref{est-rhoq1-rhol-cv} and the interpolation inequality we get
\begin{align}\label{est-lin-vis2-i2}
\norm{ |\nabla|^{2\a-1} (( \rho_{q+1}^{-1}- \rho_\ell^{-1}) m_\ell)}_{L_t^1C^\ve_x}
&\lesssim  \norm{ ( \rho_{q+1}^{-1}- \rho_\ell^{-1}) m_\ell}_{C_{t,x}}^{1-\frac{2\a-1+\ve}{3}}
\norm{  ( \rho_{q+1}^{-1}- \rho_\ell^{-1}) m_\ell}_{L_t^1C^3_x}^{\frac{2\a-1+\ve}{3}}\notag\\
&\lesssim (\norm{ \rho_{q+1}^{-1}- \rho_\ell^{-1}}_{C_{t,x}}\norm{ m_\ell}_{C_{t,x}})^{1-\frac{2\a-1+\ve}{3}}
(\norm{ \rho_{q+1}^{-1}- \rho_\ell^{-1} }_{C_tC^3_x}\norm{ m_\ell}_{L_t^1C^3_x}  )^{\frac{2\a-1+\ve}{3}}\notag\\
&\lesssim (\ell^{-12}\lambda^{-14\ve}\cdot \la^4)^{1-\frac{2\a-1+\ve}{3}} ( \ell^{-27}\lambda^{-14\ve}\cdot\ell^{-3})^{\frac{2\a-1+\ve}{3}}\notag\\
&\lesssim \ell^{-25} \lambda^{-14\ve}.
\end{align}

Regarding the third term on the right-hand side of \eqref{est-lin-vis2},
by \eqref{e3.1}, \eqref{larsrp}, \eqref{uprinlp-endpt1},  \eqref{est-rhoq1-rhol} and the interpolation inequality,
\begin{align}\label{est-lin-vis2-i3-p}
&\quad 	\norm{ |\nabla|^{2\a-1} (( \rho_{q+1}^{-1}- \rho_\ell^{-1}) w_{q+1}^{(p)})}_{L_t^1C^\ve_x} \notag\\
	&\lesssim  \norm{ ( \rho_{q+1}^{-1}- \rho_\ell^{-1}) w_{q+1}^{(p)}}_{L^1_tC_{x}}^{1-\frac{2\a-1+\ve}{3}}
	\norm{  ( \rho_{q+1}^{-1}- \rho_\ell^{-1}) w_{q+1}^{(p)}}_{L_t^1C^3_x}^{\frac{2\a-1+\ve}{3}}\notag\\
	&\lesssim (\norm{ \rho_{q+1}^{-1}- \rho_\ell^{-1}}_{C_{t,x}}\norm{ w_{q+1}^{(p)}}_{L^1_tC_{x}})^{1-\frac{2\a-1+\ve}{3}}
	(\norm{ \rho_{q+1}^{-1}- \rho_\ell^{-1} }_{C_tC^3_x}\norm{w_{q+1}^{(p)}}_{L_t^1C^3_x}  )^{\frac{2\a-1+\ve}{3}}\notag\\
	&\lesssim (\ell^{-12}\lambda^{-14\ve}\cdot \ell^{-1}\tau^{-\frac12})^{1-\frac{2\a-1+\ve}{3}} ( \ell^{-27}\lambda^{-14\ve}\cdot\ell^{-1}\lambda^3\tau^{-\frac12})^{\frac{2\a-1+\ve}{3}}\notag\\
	&\lesssim \ell^{-22} \lambda^{\a-1-8\ve}\lesssim \ell^{-22} \lambda^{-28\ve}.
\end{align}
Arguing in a similar manner as above we have
\begin{align}
	\norm{ |\nabla|^{2\a-1} (( \rho_{q+1}^{-1}- \rho_\ell^{-1}) w_{q+1}^{(c)})}_{L_t^1C^\ve_x}
	&\lesssim (\ell^{-12}\lambda_{q+1}^{-14\ve}\cdot \ell^{-5}\lambda^{-1}\tau^{-\frac12})^{1-\frac{2\a-1+\ve}{3}} ( \ell^{-27}\lambda^{-14\ve}\cdot\ell^{-5}\lambda^2\tau^{-\frac12})^{\frac{2\a-1+\ve}{3}}\notag\\
	&\lesssim \ell^{-27} \lambda^{\a-2-8\ve}\lesssim \ell^{-27} \lambda^{-1},\label{est-lin-vis2-i3-c}\\
	\norm{ |\nabla|^{2\a-1}( ( \rho_{q+1}^{-1}- \rho_\ell^{-1}) w_{q+1}^{(o)})}_{L_t^1C^\ve_x}
	&\lesssim (\ell^{-12}\lambda_{q+1}^{-14\ve}\cdot \ell^{-7}\sigma^{-1})^{1-\frac{2\a-1+\ve}{3}} ( \ell^{-27}\lambda_{q+1}^{-14\ve}\cdot\ell^{-22}\sigma^{-1})^{\frac{2\a-1+\ve}{3}}\notag\\
	&\lesssim \ell^{-43} \lambda^{-29 \ve}.\label{est-lin-vis2-i3-o}
\end{align}
Combining \eqref{est-lin-vis2-i3-p}-\eqref{est-lin-vis2-i3-o} together we obtain
\begin{align}\label{est-lin-vis2-i3-end}
\norm{ |\nabla|^{2\a-1}( ( \rho_{q+1}^{-1}- \rho_\ell^{-1}) w_{q+1})}_{L_t^1C^\ve_x}
\lesssim \ell^{-22} \lambda^{-28\ve}+ \ell^{-27} \lambda^{-1}+\ell^{-43} \lambda^{-29\ve}
\lesssim \ell^{-43} \lambda^{-28\ve}.
\end{align}

Thus, plugging \eqref{est-lin-vis2-i1-end}, \eqref{est-lin-vis2-i2} and \eqref{est-lin-vis2-i3-end} into \eqref{est-lin-vis2}
we conclude that for $\alpha \in (1/2,1)$,
\begin{align}\label{est-lin-vis2-end}
	\|R_{lin,2}\|_{L_t^1C^\ve_x}
  \lesssim \ell^{-2} \lambda^{ -14\ve}
  +\ell^{-25} \lambda^{ -14\ve} +\ell^{-43} \lambda^{-28\ve}
  \lesssim \ell^{-2} \lambda^{ -13\ve}.
\end{align}

\paragraph{\bf Estimate of $R_{lin,3}$}
Concerning the bulk viscosity $R_{lin,3}$,
we first note that,
similarly to \eqref{e4.16},
by \eqref{est-moll-rhoq},
\begin{align*}
   \|\na \rho_\ell^{-1}\|_{C_tC^\ve_x}
   \lesssim& \|\rho_\ell^{-2} \na \rho_\ell \|_{C_tC^\ve_x}
   \lesssim \|\na \rho_\ell\|_{C_tC_x^\ve} \notag \\
   \lesssim&  \|\na \rho_\ell\|_{C_tC_x}^{1-\ve} \|\na \rho_\ell\|_{C_tC_x^1}^\ve
   \lesssim \lbb_q^{\frac \ve 4  + \frac{\ve^2}{4}}
   \ll \lbb_q^\ve,
\end{align*}
and by \eqref{est-thq}, \eqref{narho-1-lbbq} and \eqref{rhoq1-1-bd},
similarly to \eqref{est-rhoq1-rhol-cv},
\begin{align*}
   \|\na(\rho_{q+1}^{-1} - \rho_\ell^{-1}) \|_{C_tC_x^\ve}
   \lesssim& \|\nabla(\rho_{q+1}^{-1} \rho_\ell^{-1} z_{q+1})\|_{C_tC_x^{\ve}} \notag \\
   \lesssim& \|\rho_{q+1}^{-1}\|_{C_tC_x^{1,\ve}}   \|\rho_{\ell}^{-1}\|_{C_tC_x^{1,\ve}}
             \|z_{q+1}\|_{C_tC_x^1}^{1-\ve}  \|z_{q+1}\|_{C_tC_x^{2}}^\ve \notag \\
   \lesssim& \lbb_q^{\frac{\ve }{4}} \lbb^{\frac{\ve }{4}} (\ell^{-16} \sigma^{-1})^{1-\ve} (\ell^{-21} \sigma^{-1})^\ve \notag \\
   \lesssim& \ell^{-17} \lbb^{-14\ve}.
\end{align*}
Moreover,
\eqref{est-moll-mc1}, \eqref{dcorlp-endpt1} and interpolation yield
\begin{align*}
   \|\div m_\ell\|_{L^1_t C_x^\ve}
   \lesssim \|m_\ell\|_{C_tC_x^1}^{1-\ve} \|m_\ell\|_{C_tC_x^2}^\ve
   \lesssim (\lbb_q^4)^{1-\ve} (\ell^{-1} \lbb_q^4)^\ve
   \lesssim \lbb_q^4 \ell^{-\ve},
\end{align*}
and
\begin{align*}
   \|\div w^{(o)}_{q+1}\|_{L^1_tC^\ve_x}
   \lesssim&  \|\div w^{(o)}_{q+1}\|_{L^1_tC_x}^{1-\ve} \|\div w^{(o)}_{q+1}\|_{L^1_tC^1_x}^\ve \notag \\
   \lesssim& (\ell^{-12} \sigma^{-1})^{1-\ve} (\ell^{-17} \sigma^{-1})^\ve \notag \\
   \lesssim& \ell^{-12} \sigma^{-1} \ell^{-5\ve}
   \lesssim \ell^{-17} \lbb^{-15\ve}.
\end{align*}

Then,
by \eqref{est-moll-rhoq}, \eqref{est-moll-mc1},
\eqref{narho-1-lbbq},
\eqref{div-wpc-dpc-0},
\eqref{velocity perturbation},
\eqref{est-ml-ve},
\eqref{est-rhoq1-rhol-cv},
\eqref{w-l1cve-end}, \eqref{est-ml-ve} and the interpolation inequality,
\begin{align}\label{est-lin-vis3}
	\|R_{lin,3}\|_{L^1_t C_x^\ve}
    & \lesssim \norm{  \nabla\rho_\ell^{-1}\cdot w_{q+1} + \nabla( \rho_{q+1}^{-1}- \rho_\ell^{-1})\cdot m_\ell+  \nabla( \rho_{q+1}^{-1}- \rho_\ell^{-1})\cdot w_{q+1} }_{L_t^1C^\ve_x}\notag\\
	&\quad + \norm{  \rho_\ell^{-1} \div w_{q+1}^{(o)} +  ( \rho_{q+1}^{-1}- \rho_\ell^{-1}) \div  m_\ell+  ( \rho_{q+1}^{-1}- \rho_\ell^{-1}) \div  w_{q+1}^{(o)} }_{L_t^1C^\ve_x}\notag\\
	& \lesssim  \norm{  \nabla \rho_\ell^{-1}}_{C_tC^\ve_x}   \norm{  w_{q+1}}_{L_t^1C^\ve_x}
 +\norm{\nabla (\rho_{q+1}^{-1}-\rho_\ell^{-1})}_{C_tC^\ve_x}(\|m_\ell\|_{L_t^1C^\ve_x}+\|w_{q+1}\|_{L_t^1C^\ve_x} )  \notag\\
 &\quad+ \norm{  \rho_\ell^{-1}}_{C_tC^\ve_x}   \norm{ \div w_{q+1}^{(o)}}_{L_t^1C^\ve_x}
 +\norm{ (\rho_{q+1}^{-1}-\rho_\ell^{-1})}_{C_tC^\ve_x}(\|\div m_\ell\|_{L_t^1C^\ve_x}+\|\div w_{q+1}^{(o)}\|_{L_t^1C^\ve_x} )  \notag\\
 &\lesssim \la^\ve \ell^{-1}\lambda^{-14\ve}+ \ell^{-18} \lambda^{-14\ve}(\la^4+ \ell^{-1}\lambda^{-14\ve}  ) +\la^{\frac{\ve^2}{4}} \ell^{-17}\lambda^{-15\ve}+ \ell^{-12} \lambda^{-14\ve}(\ell^{-\ve} \lbb_q^4 + \ell^{-17}\lambda^{-15\ve}  ) \notag\\
 & \lesssim \ell^{-19} \lambda^{-14\ve}.
\end{align}

\paragraph{\bf Estimate of $R_{lin,4}$}
For the remaining term $R_{lin,4}$ in \eqref{rup},
by \eqref{e4.17}, \eqref{est-ml-ve}, \eqref{est-rhoq1-rhol-cv},
and \eqref{w-l1cve-end},
\begin{align} \label{linear estimate1}
	\|R_{lin,4}\|_{L^1_tC_x^\ve}
    &\lesssim\norm{\rho_{q+1}^{-1}}_{C_tC^\ve_x} \norm{m_\ell}_{C_tC^\ve_x}\norm{w_{q+1}}_{L_t^1C^\ve_x}+ \| \rho_{q+1}^{-1}- \rho_\ell^{-1}\|_{C_tC^\ve_x}\|m_\ell\|_{C_tC^\ve_x}^2 \nonumber \\
	 & \lesssim \lambda^{\frac{\ve^2}{4}} \la^4 \ell^{-1}\lambda^{-14\ve}+ \ell^{-12} \lambda^{-14\ve} \la^8
     \lesssim 	\ell^{-13} \lambda^{-13\ve}.
\end{align}

Now, we conclude from \eqref{time derivative}, \eqref{e4.19}, \eqref{est-lin-vis2-end},
\eqref{est-lin-vis3} and \eqref{linear estimate1} altogether that
\begin{align}   \label{linear estimate}
	\norm{R_{lin} }_{L_t^1C^\ve_x}
       \lesssim&  \ell^{-10}\lambda^{-9\ve}+ \ell^{-13}\lambda^{-14\ve}
       + \ell^{-2} \lambda^{-13\ve}+\ell^{-19} \lambda^{-14\ve} +\ell^{-13} \lambda^{-13\ve} \notag \\
       \lesssim& \ell^{-10}\lambda^{-9\ve}.
\end{align}

\paragraph{\bf (ii) Oscillation error.}
We now deal with the oscillation error.
For this purpose,
we decompose the oscillation error into three parts:
\begin{align*}
	R_{osc} = R_{osc.1} +  R_{osc.2}+  R_{osc.3},
\end{align*}
where the high-low spatial oscillation error
\begin{align*}
	R_{osc.1}
	&:=   \sum_{k \in \Lambda }\mathcal{R}  \P_{\neq 0}\left(\g^2 \P_{\neq 0}(W_{(k)}\otimes W_{(k)} )\nabla (\rho_\ell^{-1}a_{(k)}^2) \right),
\end{align*}
the temporal oscillation error
\begin{align*}
	R_{osc.2} &
	:= -\sigma^{-1} \sum_{k\in\Lambda} h_{(k)} \p_t\mathcal{R}((k\cdot\nabla)(\rho_\ell^{-1}a_{(k)}^2 )k),
\end{align*}
and the density error
\begin{align*}
	R_{osc.3}
	&:= \mathcal{R}\div( ( \rho_{q+1}^{-1}-\rho_\ell^{-1})w_{q+1}\otimes w_{q+1}).
\end{align*}

In order to handle the high-low spatial oscillation error,
we use the following stationary phase lemma. The proof of this lemma is similar to \cite[Proposition 5.2]{dls13}, here we omit the details.

\begin{lemma}[Stationary phase lemma] \label{commutator estimate1}
Let $\ve \in(0,1)$ and $m \geq 1$. Assume that $\theta \in C^{m, \ve}\left(\mathbb{T}^d\right)$, then we have
	$$
	\left\|\mathcal{R}\left(\theta(x) e^{i \lambda \xi \cdot x}\right)\right\|_{C^\ve_x}
    \lesssim \frac{\|\theta\|_{C_{x}}}{\lambda^{1-\ve}}
    +\frac{\|\theta\|_{C^{m, \ve}_x}}{\lambda^{m-\ve}},
	$$
	where $\lambda \xi \in \mathbb{Z}^3$ and the implicit constant is independent of $q$.
\end{lemma}

Since
\begin{align*}
	\P_{\neq 0}(W_{(k)}\otimes W_{(k)} )=\P_{\geq (\lambda/2)}(W_{(k)}\otimes W_{(k)})= \P_{\geq (\lambda/2)}(\phi_{(k)}^2)(k\otimes k),
\end{align*}
and $\{\phi^2_{(k)}\}_{k}$ are $(\mathbb{T}/\lambda)^d$-periodic functions, we can decompose
\begin{align}\label{fourier}
	\P_{\geq (\lambda/2)}(\phi_{(k)}^2)=\sum_{\xi\in \mathbb{Z}^d\setminus\{0\}} f_k(\xi) e^{ i\lambda\xi \cdot x},
\end{align}
where $f_k(\xi)$ are Fourier coefficients of $\P_{\geq (\lambda/2)}(\phi_{(k)}^2)$. We note that,
$f_k(\xi)$ decay arbitrarily fast, i.e.,
$$ |f_k(\xi) |\leq C |\xi|^{-m}, $$
for all $m\in \mathbb{N}$. Here, $C$ is a constant independent of $q$.

Thus, using the identity \eqref{fourier}, Lemmas \ref{buildingblockestlemma},  \ref{mae-endpt1} and \ref{commutator estimate1}
we deduce that
\begin{align}  \label{I1-esti-endpt1}
	\norm{R_{osc.1} }_{L^1_tC^\ve_x}
	&\lesssim  \sum_{ k \in \Lambda } \|\g\|_{L^2_t}^2\norm{\mathcal{R} \P_{\not =0}
		\left(\P_{\geq (\lambda/2)}(W_{(k)}\otimes W_{(k)} )\nabla (\rho_\ell^{-1}a_{(k)}^2)\right)}_{C_tC^\ve_x} \notag \nonumber  \\
	& \lesssim \sum_{ k \in \Lambda }\sum_{\xi\in \mathbb{Z}^d\setminus\{0\}} |\xi|^{-3}	\lambda^{-1+\ve}( \|\na (\rho_\ell^{-1}a^2_{(k)})\|_{C_{t,x}} + \|\na (\rho_\ell^{-1}a^2_{(k)})\|_{C_tC^{1,\ve}_{x}}  )
	\nonumber  \\
	&\lesssim \sum_{ k \in \Lambda }	\lambda^{-1+\ve} \|\rho_\ell^{-1}\|_{C_tC_{x}^3} \|a^2_{(k)}\|_{C^3_{t,x}} \notag\\
	& \lesssim \lambda^{-1+\ve} \la^{\frac{3\ve}{4}}\ell^{-16} \lesssim \ell^{-17}\lambda^{-1+\ve}.
\end{align}

Regarding the temporal oscillation error,
we apply \eqref{hk-esti}, \eqref{mag amp estimates} and \eqref{narho-1-lbbq} to estimate
\begin{align}  \label{I2-esti-endpt1}
	\norm{R_{osc.2} }_{L^1_tC_x^\ve}
 &\lesssim  \sigma^{-1} \sum_{k\in\Lambda}
\left\|h_{(k)}\|_{C_t}\|\p_t\nabla(\rho_\ell^{-1} a_{(k)}^{2})\right\|_{L^1_tC_x^\ve} \nonumber  \\
&\lesssim \sigma^{-1} \sum_{k\in\Lambda} \|h_{(k)}\|_{C_t}\|\p_t\nabla(\rho_\ell^{-1} a_{(k)}^{2})\|_{C_{t,x}}^{1-\ve} \|\p_t\nabla(\rho_\ell^{-1} a_{(k)}^{2})\|_{C^1_{t,x}}^\ve\nonumber \\
&\lesssim  \sigma^{-1} \sum_{k\in\Lambda} \|h_{(k)}\|_{C_t}(\|\rho_\ell^{-1}\|_{C^2_{t,x}} \|a_{(k)}^{2}\|_{C^2_{t,x}})^{1-\ve}(\|\rho_\ell^{-1}\|_{C^3_{t,x}} \|a_{(k)}^{2}\|_{C^3_{t,x}})^\ve\notag\\
&\lesssim  \sigma^{-1}( \ell^{-13})^{1-\ve} ( \ell^{-19} )^\ve
\lesssim \sigma^{-1} \ell^{-14} \lesssim \ell^{-14}\lambda^{-15\ve}.
\end{align}

Finally, for density error $R_{osc.3}$,
we apply Proposition \ref{Prop-totalest} to get
\begin{align}\label{w-l2cve}
	\|w_{q+1}\|_{L^2_tC_x^\ve} & \lesssim \|w_{q+1}^{(p)}\|_{L^2_tC_x^\ve}+ \|w_{q+1}^{(c)}\|_{L^2_tC_x^\ve}+ \|w_{q+1}^{(o)}\|_{L^2_tC_x^\ve} \notag\\
	&\lesssim \ell^{-1} \lambda^{\ve}+\ell^{-5} \lambda^{\ve-1}+\ell^{-7-5\ve} \sigma^{-1}\notag\\
	&\lesssim \ell^{-1} \lambda^{\ve},
\end{align}
which along with \eqref{est-rhoq1-rhol-cv} yields that
\begin{align}  \label{I3-esti-endpt1}
	\norm{R_{osc.3} }_{L^1_tC^\ve_x}
    &\lesssim \| \rho_{q+1}^{-1}-\rho_\ell^{-1}\|_{C_tC^\ve_x}\|w_{q+1}\|_{L^2C^\ve_x}^2 \lesssim \ell^{-14} \lambda^{-12\ve}.
\end{align}

Therefore, combining \eqref{I1-esti-endpt1}, \eqref{I2-esti-endpt1} and \eqref{I3-esti-endpt1}
altogether
we conclude
\begin{align}
	\label{oscillation estimate}
	\norm{R_{osc}}_{L_t^1C^\ve_x}
    &\lesssim  \ell^{-17}\lambda^{-1+\ve}+\ell^{-14}\lambda^{-15\ve}+\ell^{-14} \lambda^{-12\ve} \notag \\
    &\lesssim \ell^{-14} \lambda^{-12\ve}.
\end{align}

\paragraph{\bf (iii) Corrector error.}
Regarding the corrector error $R_{cor}$ in \eqref{rup2},
using \eqref{ucorlp-endpt1}, \eqref{dcorlp-endpt1}, \eqref{narho-1-lbbq},
\eqref{w-l2cve} and interpolation we derive
\begin{align}\label{corrector estimate}
	\norm{R_{cor} }_{L^1_{t}C^{\ve}_x}
	\lesssim& \|\rho_{\ell}^{-1}\|_{C_tC_x^\ve}
	(\norm{ w_{q+1}^{(p)} \otimes (w_{q+1}^{(c)}+ w_{q+1}^{(o)}) \|_{L^1_{t}C^{\ve}_x}
        + \|(w_{q+1}^{(c)}+w_{q+1}^{(o)}) \otimes w_{q+1} }_{L^1_{t}C^{\ve}_x} ) \notag \\
	\lesssim&  \|\rho_{\ell}^{-1}\|_{C_tC_x^\ve}
	\norm{w_{q+1}^{(c)}+w_{q+1}^{(o)}}_{L^2_{t}C^{\ve}_x} (\norm{w^{(p)}_{q+1} }_{L^2_{t}C^{\ve}_x} + \norm{w_{q+1} }_{L^2_{t}C^{\ve}_x})\notag  \\
	\lesssim& \la^{\frac{\ve^2}{4}} (\ell^{-5} \lambda^{\ve-1}+\ell^{-7-5\ve} \sigma^{-1}) \ell^{-1} \lambda^{\ve}\notag\\
	\lesssim& \ell^{-7}  \lambda^{2\ve-1}
	+ \ell^{-9} \lambda^{-14\ve}
	\lesssim \ell^{-9} \lambda^{-14\ve}.
\end{align}

\paragraph{\bf (iv) Pressure error.}
In order to estimate the pressure error in \eqref{rpre},
we use the boundedness of Calder\'{o}n-Zygmund operators,
the mean-valued Theorem and interpolation to get
\begin{align*}
	\norm{R_{pre} }_{L^1_{t}C^{\ve}_x}
	&\lesssim \norm{ P(\rho_{q+1}) -P(\rho_\ell) }_{C_{t}C^{\ve}_x} \notag \\
	&\lesssim \left\|\int_{\rho_\ell}^{\rho_{q+1}}P'(y) \d y \right\|_{C_{t,x}}^{1-\ve}\norm{ P(\rho_{q+1}) -P(\rho_\ell) }_{C_{t}C^{1}_x}^\ve \notag \\
	&\lesssim \norm{ P'(\xi) (\rho_{q+1}-\rho_{\ell}) }_{C_{t,x}}^{1-\ve} (\norm {P(\rho_{q+1})}_{C_{t}C^{1}_x}+\norm{  P(\rho_\ell) }_{C_{t}C^{1}_x}  )^\ve\notag  \\
	&\lesssim \|P'(\xi)\|_{C_{t,x}}^{1-\ve} \norm{ \rho_{q+1}- \rho_{\ell} }_{C_{t,x}}^{1-\ve}
        \bigg( \norm {P(\rho_{q+1})}_{C_{t,x} }+\norm{  P(\rho_\ell) }_{C_{t,x} }   \notag \\
    &\qquad \qquad \qquad + \|P'(\rho_{q+1})\|_{C_{t,x} } \norm{ \rho_{q+1}}_{C_{t}C^{1}_x}+ \|P'(\rho_{\ell})\|_{C_{t,x} } \norm{ \rho_{\ell}}_{C_{t}C^{1}_x}  \bigg)^\ve,
\end{align*}
where $\min\{ \rho_{q+1},\rho_{\ell}\} \leq \xi\leq \max\{ \rho_{q+1},\rho_{\ell}\} $.

As in \eqref{P'-P''-bdd},
because $\rho_{q+1}$ and $\rho_\ell$ are uniformly away from zero and infinity,
and $P$, $P'$, $P''$ are continuous,
we have the uniform boundedness
\begin{align}   \label{P''-bdd}
   \|P'(\xi)\|_{C_{t,x}} +  \|P'(\rho_\ell)\|_{C_{t,x}}  +  \|P'(\rho_{q+1})\|_{C_{t,x}}
   \lesssim 1.
\end{align}

Then, using  \eqref{est-moll-rhoq},  \eqref{est-thq} and
\eqref{rhoq1-cn} we obtain
\begin{align}	\label{pressure estimate}
	\norm{R_{pre} }_{L^1_{t}C^{\ve}_x}
	\lesssim(\ell^{-11} \sigma^{-1} )^{1-\ve}(1+\lambda^{\frac{\ve}{4}} + \lambda_{q}^{\frac{\ve}{4}}  )^\ve
	\lesssim \ell^{-11}\lambda^{-13\ve},
\end{align}

\paragraph{\bf (v) Commutator error.}
It remains to treat the commutator error given by \eqref{def-rcom}.
Using \eqref{def-rcom}
we deduce
\begin{align}\label{est-rcom}
\norm{	R_{com} }_{L^1_{t}C^{\ve}_x}&\lesssim \| P(\rho_{\ell})- P_{\ell} \|_{L^1_{t}C^{\ve}_x} +\| |\nabla|^{2\a-1} \(\rho_\ell^{-1}m_\ell- (\rho_q^{-1}m_q)*_{x} \phi_{\ell} *_{t} \varphi_{\ell}\) \|_{L^1_{t}C^{\ve}_x} \notag \\
&\quad +\|\div \(\rho_\ell^{-1}m_\ell- (\rho_q^{-1}m_q)*_{x} \phi_{\ell} *_{t} \varphi_{\ell}\)\|_{L^1_{t}C^{\ve}_x} \notag\\
&\quad  + \|\rho_\ell^{-1} m_{\ell}\otimes m_{\ell}- (\rho_q^{-1}\m \otimes \m ) *_{x} \phi_{\ell} *_{t} \varphi_{\ell}\|_{L^1_{t}C^{\ve}_x} \notag\\
&:=R_{com.1}+R_{com.2}+R_{com.3}+R_{com.4}.
\end{align}
In the following we estimate the four parts on the right-hand side separately.
\medskip

\paragraph{\bf Estimate of $R_{com.1}$}
For the pressure commutator error $R_{com.1}$,
similarly to \eqref{pressure estimate},
we estimate
\begin{align*}\label{est-rcom1}
\|R_{com.1} \|_{L^1_tC^\ve_x}
& \lesssim \|P(\rho_{\ell})- P_{\ell}\|_{C_{t,x}}^{1-\ve}\|P(\rho_{\ell})- P_{\ell}\|_{C_{t}C^{1}_x}^\ve\notag\\
&\lesssim ( \| P(\rho_{\ell})- P(\rho_q)\|_{C_{t,x}}+ \|P(\rho_q)-P_{\ell}\|_{C_{t,x}})^{1-\ve} ( \|P(\rho_q)\|_{C_tC^1_{x}}+\|P_{\ell}\|_{C_tC^1_{x}})^\ve\notag\\
&\lesssim ( \| \rho_{\ell}- \rho_q \|_{C_{t,x}}+ \ell \|P(\rho_q)\|_{C^1_{t,x}})^{1-\ve}  ( \| P(\rho_{q}) \|_{C_{t,x}}+ \| P'(\rho_{q}) \|_{C_{t,x}}\| \rho_{q} \|_{C_tC^1_{x}})^\ve,
\end{align*}
which along with \eqref{rhoc1}, \eqref{est-moll-rl-rq} and \eqref{P'-P''-bdd} yields
\begin{align}
\|R_{com.1} \|_{L^1_tC^\ve_x}
&\lesssim  (\ell \|\rho_{q}\|_{C^1_{t,x}} + \ell\la^{\frac14})^{1-\ve} (1+\| \rho_{q}\|_{C_tC^1_{x}})^\ve \notag\\
&\lesssim \ell^{1-\ve} \la^{\frac{\ve}{4}}\lesssim \ell^{\frac{1}{2}}.
\end{align}

\paragraph{\bf Estimate of $R_{com.2}$}
Regarding the shear viscous commutator error $R_{com.2}$, if $\a\in(\frac12,1)$,
an application of the interpolation inequality gives
\begin{align}\label{est-rcom2}
\|R_{com.2} \|_{L^1_tC^\ve_x}
&\lesssim \|   \rho_\ell^{-1}m_\ell- (\rho_q^{-1}m_q)*_{x} \phi_{\ell} *_{t} \varphi_{\ell} \|_{C_{t,x}}^{1-\frac{2\a-1+\ve}{3}} \| \rho_\ell^{-1}m_\ell- (\rho_q^{-1}m_q)*_{x} \phi_{\ell} *_{t} \varphi_{\ell} \|_{C_{t}C^{3}_x}^{\frac{2\a-1+\ve}{3}}.
\end{align}
Note that, by \eqref{rhobd}, \eqref{mc1} and \eqref{est-moll-rl-rq},
\begin{align}\label{est-rcom2-1}
&\quad  \|   \rho_\ell^{-1}m_\ell- (\rho_q^{-1}m_q)*_{x} \phi_{\ell} *_{t} \varphi_{\ell} \|_{C_{t,x}}\notag\\
&\lesssim \|   \rho_\ell^{-1}m_\ell-  \rho_\ell^{-1}m_q \|_{C_{t,x}} +  \|   \rho_\ell^{-1}m_q-  \rho_q^{-1}m_q  \|_{C_{t,x}} +\|   \rho_q^{-1}m_q	- (\rho_q^{-1}m_q)*_{x} \phi_{\ell} *_{t} \varphi_{\ell} \|_{C_{t,x}}\notag\\
&\lesssim  \|   \rho_\ell^{-1}\|_{C_{t,x}} \| m_\ell-m_q \|_{C_{t,x}} + \|  \rho_\ell^{-1}-  \rho_q^{-1}  \|_{C_{t,x}}\|m_q  \|_{C_{t,x}} +\ell \| \rho_q^{-1}m_q\|_{C^1_{t,x}}\notag\\
&\lesssim \ell \|m_q\|_{C^1_{t,x}}
   + \|\rho_\ell\|_{C_{t,x}}\|\rho_q\|_{C_{t,x}} \|\rho_\ell - \rho_q\|_{C_{t,x}} \|m_q\|_{C_{t,x}}
   + \ell \| \rho_q^{-1} \|_{C^1_{t,x}}\|m_q\|_{C^1_{t,x}}\notag\\
&\lesssim \ell \la^4+\ell \lbb_q^{2+ \frac \ve 4} + \ell \la^{4+\frac{\ve}{4}}
\lesssim  \ell^{\frac12}.
\end{align}
Similarly, using also \eqref{est-moll-rl-rq},  \eqref{est-rhoq1-rhol}
and \eqref{rhoq1-1-bd} with $q$ replacing $q+1$ we estimate
\begin{align}\label{est-rcom2-2}
	&\quad  \|   \rho_\ell^{-1}m_\ell- (\rho_q^{-1}m_q)*_{x} \phi_{\ell} *_{t} \varphi_{\ell} \|_{C_tC^3_{x}}\notag\\
	&\lesssim  \|   \rho_\ell^{-1}\|_{C_tC^3_{x}} \| m_\ell-m_q \|_{C_tC^3_{x}} + \|  \rho_\ell^{-1}-  \rho_q^{-1}  \|_{C_tC^3_{x}}\|m_q  \|_{C_tC^3_{x}} +\ell \| \rho_q^{-1}m_q\|_{C_tC^4_{x}}+ \ell \| \rho_q^{-1}m_q\|_{C_t^1C^3_{x}}\notag\\
	&\lesssim \ell\la^{\frac{3\ve}{4}} \|m_q\|_{C^4_{t,x}} + \ell^{-27}\lambda^{-14\ve}\la^8  +
	\ell \| \rho_q^{-1} \|_{C_tC^4_{x}}\|  m_q\|_{C_tC^4_{x}}
	+ \ell \| \rho_q^{-1} \|_{C_t^1C^3_{x}}\|  m_q\|_{C_t^1C^3_{x}}\notag\\
	&\lesssim \ell \la^{10 +\frac{3\ve}{4}} + \ell^{-13} \lambda^{-14\ve} + \ell \la^{10+\ve}
	\lesssim  \ell^{\frac12}.
\end{align}
If $\a\in(0,\frac12]$, arguing in a similar manner with \eqref{est-rcom2}-\eqref{est-rcom2-2} we obtain
\begin{align}\label{est-rcom2-3}
	\|R_{com.2} \|_{L^1_tC^\ve_x}
	&\lesssim \|   \rho_\ell^{-1}m_\ell- (\rho_q^{-1}m_q)*_{x} \phi_{\ell} *_{t} \varphi_{\ell} \|_{C_{t,x}}^{1-\frac{\ve}{3}} \| \rho_\ell^{-1}m_\ell- (\rho_q^{-1}m_q)*_{x} \phi_{\ell} *_{t} \varphi_{\ell} \|_{C_{t}C^{3}_x}^{\frac{\ve}{3}}\lesssim \ell^{\frac12}.
\end{align}
Thus, we conclude from \eqref{est-rcom2}-\eqref{est-rcom2-3} that
\begin{align}\label{est-rcom2-end}
 \|R_{com.2} \|_{L^1_t C^\ve_x}\lesssim  \ell^{\frac12}.
\end{align}

\paragraph{\bf Estimate of $R_{com.3}$}
The bulk viscous commutator error $R_{com.3}$ can be estimated in a similar fashion as that of $R_{com.2}$:
\begin{align}\label{est-rcom3}
\|R_{com.3} \|_{L^1_tC^\ve_x}
 &\lesssim \| \rho_\ell^{-1}m_\ell- (\rho_q^{-1}m_q)*_{x} \phi_{\ell} *_{t} \varphi_{\ell} \|_{C_{t,x}}^{1-\frac{1+\ve}{3}} \| \rho_\ell^{-1}m_\ell- (\rho_q^{-1}m_q)*_{x} \phi_{\ell} *_{t} \varphi_{\ell} \|_{C_{t}C^{3}_x}^{\frac{1+\ve}{3}}\notag\\
 &\lesssim \ell^{\frac12}.
\end{align}

\paragraph{\bf Estimate of $R_{com.4}$}
At last,
concerning the nonlinear commutator error $R_{com.4}$,
we see that
\begin{align}\label{est-rcom4}
\|R_{com.4} \|_{L^1_tC^\ve_x}
\lesssim& \|\rho_\ell^{-1} m_{\ell}\otimes m_{\ell}-\rho_q^{-1}\m \otimes \m \|_{L^1_{t}C^{\ve}_x} \notag \\
   & +\|\rho_q^{-1}\m \otimes \m - (\rho_q^{-1}\m \otimes \m ) *_{x} \phi_{\ell} *_{t} \varphi_{\ell}\|_{L^1_{t}C^{\ve}_x}  \notag \\
   =:&J_1 + J_2.
\end{align}
For the first term on the right-hand side above, we have
\begin{align}\label{est-rcom4-1}
  J_1 &\lesssim \|\rho_\ell^{-1}  m_{\ell}\otimes( m_{\ell} -\m) \|_{L^1_{t}C^{\ve}_x}+ \|\rho_\ell^{-1}  (m_{\ell}-\m) \otimes m_{q} \|_{L^1_{t}C^{\ve}_x} + \|(\rho_\ell^{-1}-\rho_q^{-1} ) \m\otimes\m \|_{L^1_{t}C^{\ve}_x} \notag\\
&\lesssim \| \rho_\ell^{-1}\|_{C_{t}C^{\ve}_x} \|m_{\ell}-\m \|_{L^1_{t}C^{\ve}_x} (\|m_\ell \|_{C_{t}C^{\ve}_x}+\|\m \|_{C_{t}C^{\ve}_x}  )+  \|\rho_\ell^{-1}-\rho_q^{-1}\|_{C_{t}C^{\ve}_x}\|\m\|_{L^2_{t}C^{\ve}_x}^2.
\end{align}	
Note that, by  interpolation,
\eqref{est-moll-rl-rq}, \eqref{est-moll-ml-mq} and \eqref{narho-1-lbbq},
\begin{align}
 \|m_{\ell}-\m \|_{L^1_{t}C^{\ve}_x}
   \lesssim& \|m_{\ell}-\m \|_{C_{t,x}}^{1-\ve}\|m_{\ell}-\m \|_{C^1_{t,x}}^\ve   \notag \\
   \lesssim& (\ell \lbb_q^4)^{1-\ve}(\ell \lbb_q^6)^{\ve} \lesssim \ell \la^{4+2\ve}, \label{ml-mq-ve}
\end{align}
and
\begin{align}
  \|\rho_\ell^{-1}-\rho_q^{-1}\|_{C_{t}C^{\ve}_x}
  \lesssim& \|\rho_\ell^{-1}\|_{C_{t}C^{\ve}_x}\|\rho_q^{-1}\|_{C_{t}C^{\ve}_x}\|\rho_\ell-\rho_q\|_{C_{t}C^{\ve}_x}  \notag \\
  \lesssim& (\la^{\frac{\ve^2}{4}})^2
    \|\rho_\ell - \rho_q\|_{C_{t,x}}^{1-\ve}
    \|\rho_\ell - \rho_q\|_{C_{t}C_x^1}^{\ve}
  \lesssim \ell \la^{\ve}. \label{rl-rq-ve}
\end{align}
Thus, plugging \eqref{ml-mq-ve} and \eqref{rl-rq-ve} into \eqref{est-rcom4-1}
and using \eqref{e4.16} and \eqref{est-ml-ve}
we arrive at
\begin{align}\label{est-rcom4-1-end}
  J_1 \lesssim \la^{\frac{\ve^2}{4}} \ell \la^{4+2\ve} \la^4+ \ell\la^\ve\la^8
	\lesssim \ell \la^{9} \lesssim \ell^{\frac12}.
\end{align}	
Concerning the second term on the right-hand side of \eqref{est-rcom4},
using interpolation and \eqref{rhoc1}, \eqref{mc1} we get
\begin{align}\label{est-rcom4-2}
	J_2
    &\lesssim \|\rho_q^{-1}\m \otimes \m - (\rho_q^{-1}\m \otimes \m ) *_{x} \phi_{\ell} *_{t} \varphi_{\ell}\|_{C_{t,x}}^{1-\ve}\|\rho_q^{-1}\m \otimes \m - (\rho_q^{-1}\m \otimes \m ) *_{x} \phi_{\ell} *_{t} \varphi_{\ell}\|_{C_tC^1_{x}}^\ve \notag\\
	&\lesssim (\ell\|\rho_q^{-1}\m \otimes \m\|_{C^1_{t,x}} )^{1-\ve}
     (\ell\|\rho_q^{-1}\m \otimes \m\|_{C_{t,x}^2})^\ve\notag\\
	&\lesssim (\ell\|\rho_q^{-1}\|_{C^1_{t,x}}\| \m \|_{C^1_{t,x}}^2 )^{1-\ve}
      (\ell\|\rho_q^{-1}\|_{C_{t,x}^2} \| \m \|_{C_{t,x}}^2  )^\ve\notag\\
	&\lesssim (\ell\la^{\frac{\ve}{4}} \la^8)^{1-\ve}( \ell \la^{\frac{\ve}{2}} \la^{12}  )^\ve\lesssim \ell\la^9\lesssim \ell^{\frac12}.
\end{align}
Thus, it follows from \eqref{est-rcom4}, \eqref{est-rcom4-1-end} and \eqref{est-rcom4-2} that
\begin{align}\label{est-rcom4-end}
\|R_{com.4} \|_{L^1_tC^\ve_x}\lesssim \ell^{\frac12}.
\end{align}

Therefore, combining \eqref{est-rcom1}, \eqref{est-rcom2-end}, \eqref{est-rcom3} and \eqref{est-rcom4-end}
altogether we obtain
\begin{align}\label{est-rcom-end}
\|R_{com} \|_{L^1_tC^\ve_x} \lesssim   \ell^{\frac12}.
\end{align}

Now, combining the estimates \eqref{linear estimate},
\eqref{oscillation estimate},
\eqref{corrector estimate}, \eqref{pressure estimate}, \eqref{est-rcom-end} altogether
and using the smallness condition \eqref{b-beta-ve} we conclude that
\begin{align} \label{rq1b}
	\|R_{q+1} \|_{L^1_{t}C_x}
	&\leq \| R_{lin} \|_{L^1_tC^\ve_x} +  \| R_{osc}\|_{L^1_tC^\ve_x}
	+  \|R_{cor} \|_{L^1_tC^\ve_x} +\|R_{pre} \|_{L^1_tC^\ve_x}+\|R_{com} \|_{L^1_tC^\ve_x}\nonumber  \\
	&\lesssim   \ell^{-10}\lambda^{-9\ve}+\ell^{-14} \lambda^{-12\ve}+ \ell^{-9} \lambda^{-14\ve}+ \ell^{-11}\lambda^{-13\ve}+\ell^{\frac12}\nonumber  \\
	& \lesssim \delta_{q+2}.
\end{align}
Finally, the $L^1_tC_x$-estimate \eqref{rl1} of Reynolds stress is verified
at level $q+1$.

\begin{remark}\label{rq1-incom}
For the case of the hypo-viscous INS,
the Reynolds stress has much simpler expression.
Namely, we have
	\begin{align}  \label{ru-imcom-2}
		\displaystyle\div\mathring{R}_{q+1} - \nabla P_{q+1}
		&\displaystyle = \underbrace{\partial_t (w_{q+1}^{(p)}+w_{q+1}^{(c)}) +\nu(-\Delta)^{\alpha} w_{q+1} +\div\big(u_\ell \otimes w_{q+1} + w_{q+ 1} \otimes u_\ell\big) }_{ \div\mathring{R}_{lin} +\nabla P_{lin} }   \notag\\
		&\displaystyle\quad+ \underbrace{\div (w_{q+1}^{(p)} \otimes w_{q+1}^{(p)}+  \mathring{R}_\ell)+\partial_t \wo}_{\div\mathring{R}_{osc} +\nabla P_{osc}}  \notag\\
		&\displaystyle\quad+ \underbrace{\div\Big((w_{q+1}^{(c)}+\wo)\otimes w_{q+1}+ w_{q+1}^{(p)} \otimes (w_{q+1}^{(c)}+\wo) \Big)}_{\div\mathring{R}_{cor} +\nabla P_{cor}}\notag\\
		&\quad+ \underbrace{\div \(u_{\ell}\otimes u_{\ell}-\left(\u \otimes\u\right) *_{x} \phi_{\ell} *_{t} \varphi_{\ell}\)}_{\div\mathring{R}_{com} +\nabla P_{com}},
	\end{align}
where the linear error
\begin{align}
	\mathring{R}_{lin} & := \mathcal{R}\(\partial_t (w_{q+1}^{(p)} +w_{q+1}^{(c)}  )\)
	+ \nu \mathcal{R} (-\Delta)^{\a} w_{q+1} + \mathcal{R}\P_H \div \( u_\ell \mathring{\otimes} w_{q+1} + w_{q+ 1}
	\mathring{\otimes} u_\ell\), \label{lin-2}
\end{align}
the oscillation error
\begin{align}\label{osc-2}
	\mathring{R}_{osc} :=& \sum_{k \in \Lambda } \mathcal{R} \P_H\P_{\neq 0}\left(\g^2 \P_{\neq 0}(W_{(k)}\otimes W_{(k)})\nabla (a_{(k)}^2)\right) -\sigma^{-1}\sum_{k\in \Lambda}\mathcal{R} \P_H \P_{\neq 0}\(h_{(k)}k\otimes k\p_t\nabla(a_{(k)}^{2})\),
\end{align}
the corrector error
\begin{align}
	\mathring{R}_{cor} &
	:= \mathcal{R} \P_H \div \bigg( w^{(p)}_{q+1} \mathring{\otimes} (w_{q+1}^{(c)} +\wo)
	+ (w_{q+1}^{(c)} +\wo) \mathring{\otimes} w_{q+1} \bigg). \label{cor-2}
\end{align}
and the commutator error
\begin{align}
	\mathring{R}_{com} &
	:= \mathcal{R} \P_H \div \bigg( u_{\ell}\mathring{\otimes} u_{\ell}-\left(\u \mathring{\otimes}\u\right) *_{x} \phi_{\ell} *_{t} \varphi_{\ell} \bigg). \label{com-2}
\end{align}
Thus, by virtue of Remarks~\ref{incom-3.7} and \ref{rem-3.10},
arguing as in a fashion as above
(one may let the density $\rho_q\equiv 1$ for all $q\in\mathbb{N}$
to simplify the arguments),
we can also obtain the key inductive estimates
for the Reynolds stress at level $q+1$
\begin{align} \label{rq1b-2}
\|\mathring R_{q+1} \|_{C_{t,x}^1} \lesssim \lambda_{q+1}^9,\quad	\|\mathring R_{q+1} \|_{L^1_{t}C_x} \lesssim \delta_{q+2}.
\end{align}
\end{remark}

\section{Proof of main results}  \label{Sec-Proof-Main}

We are now in stage to prove the main results,
i.e. Theorems  \ref{Thm-Nonuniq-CNS}, \ref{Thm-Nonuniq-hypoNSE},
and \ref{Thm-hypoNSE-Euler-limit} and \ref{Thm-Iterat}.
Let us start with the proof of the main iteration result in Theorem \ref{Thm-Iterat}.
\medskip

\paragraph{\bf Proof of Theorem \ref{Thm-Iterat}}
Since  the iterative estimates \eqref{rhobd}-\eqref{m-L1tLace-conv}
have been verified in the previous sections,
below we focus on the remaining inductive eatimate \eqref{suppru}
for the temporal support.

For this purpose, by the definition of $w_{q+1}$, $\thq$ and $T_q$ in \eqref{velocity perturbation}, \eqref{zq1-def} and \eqref{def-tq},
respectively,
we have
	\begin{align}  \label{sup-w}
		& \supp_t w_{q+1}  \subseteq \bigcup_{k\in \Lambda}\supp_t a_{(k)} \subseteq N_{2\ell} ([T_q,T]),
	\end{align}
and
	\begin{align}  \label{sup-z}
	& \supp_t z_{q+1}  \subseteq N_{2\ell} ([T_q,T]),
\end{align}
which, via \eqref{q+1 velocity} and \eqref{def-rqq},
yields that
\begin{align}\label{supp-mq}
& \supp_t (\nabla\rho_{q+1},m_{q+1})\subseteq \supp_t (\nabla\rho_\ell, m_{\ell}) \cup \supp_t (\nabla\thq,w_{q+1}) \subseteq N_{2\ell} ([T_q,T]).
\end{align}

Moreover, using \eqref{rucom} we derive
\begin{align}\label{supp-rq}
\supp_t \rr_{q+1} &\subseteq \bigcup\limits_{k\in \Lambda}\supp_t
    (a_{(k)},  m_\ell,   \nabla \rho_{q+1},  \nabla \rho_{\ell})\cup  N_{\ell}(\supp_t \rr_{q} ) \notag\\
&\subseteq \bigcup\limits_{k\in \Lambda}\supp_t
    (a_{(k)},  m_\ell,  \nabla \thq,  \nabla \rho_{\ell})\cup  N_{\ell}(\supp_t \rr_{q} ) \notag\\
&\subseteq N_{2\ell} ([T_q,T]).
\end{align}
Thus, we conclude from \eqref{supp-mq} and \eqref{supp-rq} that
\begin{align}\label{supp-mrr}
	& \supp_t (\nabla\rho_{q+1}, m_{q+1}, R_{q+1})\subseteq  N_{2\ell} ([T_q,T]),
\end{align}
which along with the fact that $2\ell \ll \delta_{q+2}^{1/2}$
verifies \eqref{suppru}.
Therefore, the proof of Theorem \ref{Thm-Iterat} is complete.
\hfill $\square$

\begin{remark}\label{rem-supp-incom}
Note that,
in the incompressible case,
we also have that, similarly to \eqref{supp-mrr},
\begin{align*}
& \supp_t (u_{q+1}, R_{q+1})\subseteq  N_{2\ell}([T_q,T]).
\end{align*}
\end{remark}

\medskip

\paragraph{\bf Proof of Theorem \ref{Thm-Nonuniq-hypoNSE}}
We verify the statements $(i)$-$(iv)$ in Theorem~\ref{Thm-Nonuniq-hypoNSE} separately.

$(i)$. Let $\rho_0=\wt \rho$, $m_0=\wt m$ and set
\begin{align}\label{r0u}
	& R_0 :=\mathcal{R}\(\p_t \wt m
      +\mu(-\Delta)^{\alpha} (\wt\rho^{-1}\wt m)
      -(\mu+\nu)\nabla\div (\wt \rho^{-1}\wt m)\) + \wt\rho^{-1}\wt m\otimes\wt m+ \mathcal{R}(\nabla P(\wt \rho)).
\end{align}
In view of the fact that $(\wt \rho, \wt m)$ is a smooth solution to the transport equation $\eqref{equa-NS}$
and the identity \eqref{r-iden},
we deduce that $(p_0, m_0, \rr_0)$ is a relaxed solution to \eqref{equa-nsr}.
Moreover, for $a$ sufficiently large,
\eqref{rhobd}-\eqref{rl1} are satisfied at level $q=0$.
Thus, in view of Theorem~\ref{Thm-Iterat},
there exists a sequence of relaxed solutions $\{(\rho_q, m_{q},\rr_{q})\}_{q}$
to \eqref{equa-nsr}, which satisfy \eqref{rhobd}-\eqref{suppru} for all $q\geq 0$.

Therefore, using \eqref{rho-l9-conv}
we have that for $a$ large enough,
\begin{align}  \label{rhoq1-rhoq-ve}
\sum_{q \geq 0}\norm{ \rqq - \rq }_{C_{t}C_x^1}
 \lesssim  \sum_{q \geq 0}\delta_{q+2}^{\frac{1}{2}}
 \lesssim \sum\limits_{q\geq 2} a^{-\beta b q}
 \lesssim \frac{a^{-\beta b}}{a^{\beta b}-1}
 \leq \ve_*.
\end{align}
In particular, this yields that $\{\rho_q\}_{q\geq 0}$ is a Cauchy sequence in $C_{t}C_x^1$,
and so there exists $\rho\in C_{t}C^1_x$ such that
\begin{align} \label{rhoq-rho}
   \lim_{q\rightarrow\infty} \rho_q=\rho\ \  in\  C_{t}C_x^1,
\end{align}
which along with \eqref{rhobd} yields that
\begin{align} \label{rhoq1-rho1}
   \lim_{q\rightarrow\infty} \rho_q^{-1} = \rho^{-1}\ \ in\ C_{t,x}.
\end{align}

Moreover, using interpolation, \eqref{la}, \eqref{mc1} and \eqref{m-L2t-conv}
we derive that for any $\beta'\in (0,\frac{\beta}{4+\beta})$,
\begin{align}
\sum_{q \geq 0} \norm{ m_{q+1} - m_q }_{H^{\beta'}_{t}C_x}
	\leq  & \, \sum_{q \geq 0} \norm{ m_{q+1} - m_q }_{L^2_{t}C_x}^{1- \beta'}\norm{ m_{q+1} - m_q }_{C^1_{t,x}}^{\beta'}\notag\\
	\lesssim  &\,  \sum_{q \geq 0} \delta_{q+1}^{\frac{1-\beta'}{2}}\lambda_{q+1}^{4 \beta' } \notag\\
	\lesssim &\,   \delta_{1}^{\frac{1-\beta'}{2}}\lambda_{1}^{4 \beta' } +
	\sum_{ q \geq 1}   \lambda_{q+1}^{-\beta(1 - \beta')  + 4\beta'  } <\9, \label{interpo}
\end{align}
where the last step was due to $-\beta(1 - \beta')  + 4\beta' <0$. By \eqref{m-L1tLace-conv}, we also have
\begin{align}
	\sum_{q \geq 0}\norm{ m_{q+1} - m_q }_{C_tH^{-1}_x} < \9.  \label{result-lw-1}
\end{align}
Thus, there exists $m\in H^{\beta'}_{t}C_x\cap C_tH^{-1}_x$ such that
\begin{align}  \label{mq-m}
\lim_{q\rightarrow\infty}m_q =m\ \  in\ H^{\beta'}_{t}C_x\cap C_tH^{-1}_x.
\end{align}
Hence, it follows from \eqref{rhoq1-rho1} and \eqref{mq-m} that
\begin{align}
   \lim_{q\rightarrow\infty} \rho_q^{-1} m_q &= \rho^{-1} m\ \  in\ L^2_tC_x,\\
  \lim_{q\rightarrow\infty}  \rho_q^{-1} m_q \otimes m_q
  &=  \rho^{-1} m \otimes m\ \ in\ L^1_t C_x,
\end{align}
and
\begin{align}
m(t)\rightharpoonup m_0\quad\text{weakly in}\ H^{-1}(\T^d),
\end{align}
for some $m_0\in H^{-1}(\T^d)$. By the mean-valued theorem, \eqref{P'-P''-bdd} and \eqref{rhoq1-rho1},
\begin{align}
   \lim_{q\rightarrow\infty} P(\rho_q) = P(\rho)\ \ in\ C_{t,x}.
\end{align}
Thus, taking into account \eqref{rl1}
\begin{align}
  \lim_{q \to \infty} R_{q} = 0\ \  in\ L^1_{t}C_x,
\end{align}
we thus conclude that $(\rho,m)$ is a weak solution to \eqref{equa-NS}.

$(ii)$. Regarding the regularity statement $(ii)$,
in view of \eqref{rhoq-rho} and \eqref{mq-m},
we only need to show that
\begin{align} \label{m-LpCs}
    m\in L^p_tC^s_x.
\end{align}
To this end, using \eqref{m-L1tLace-conv}
and estimating as in \eqref{result-lw} we derive
\begin{align}
	\sum_{q \geq 0}\norm{ m_{q+1} - m_q }_{L^p_tC^{s}_{x}} < \9,  \label{result-lw}
\end{align}
which yields that $\{m_q\}_{q\geq 0}$
is also a Cauchy sequence in $L^p_tC^{s}_{x}$.
Hence, by the uniqueness of weak limits,
we obtain \eqref{m-LpCs} and thus prove the regularity statement $(ii)$.

$(iii)$. Concerning the mass preservation in $(iii)$,
by \eqref{rhopre},
\begin{align*}
& \int_{\mathbb{T}^d}\rho_{q}(t,x)\d x=\int_{\mathbb{T}^d}\rho_{0}(t,x)\d x= \int_{\mathbb{T}^d} \wt \rho(t,x)\d x,
\end{align*}
for all $q\in\mathbb{N}$ and $t\in [0,T]$.
Then, using \eqref{rhoq-rho} to pass to the limit $q\to \infty$ we get
\begin{align*}
& \int_{\mathbb{T}^d}\rho(t,x)\d x= \int_{\mathbb{T}^d} \wt \rho(t,x)\d x,\ \ \forall\ t\in [0,T],
\end{align*}
which verifies the mass perservation statement $(iii)$.

$(iv).$ Regarding the small deviations statement $(iv)$.
The small deviation between $\rho$ and $\wt \rho$ in $C_tC^1_x$
is implied by \eqref{rhoq1-rhoq-ve}.
Moreover, using \eqref{m-L1tLace-conv}
and estimating as in \eqref{result-lw} we also get
\begin{align}\label{e6.5}
	&\norm{m - \tilde{m} }_{L^1_tC_x}+\norm{ m - \tilde{m} }_{L^p_tC^s_x}+\norm{ m - \tilde{m} }_{C_tH^{-1}_x}\notag\\
	\leq &\,\sum_{q \geq 0}(\norm{m_{q+1} - m_q }_{L^1_tC_x}+ \norm{ m_{q+1} - m_q }_{L^p_tC^s_x}+\norm{ m_{q+1} - m_q }_{C_tH^{-1}_x})\notag \\
	\leq &\,2\sum_{q\geq 0} \delta_{q+1}^\frac 12  \leq \ve_*
\end{align}
for $a$ sufficiently large,
thereby proving the statement $(iv)$.

$(v).$ At last, for the temporal supports, by \eqref{rhoq-rho} and \eqref{mq-m}, taking into account
$\sum_{q\geq 0}\delta_{q+2}^{1/2}\leq \ve_*$
for $a$ large enough, we obtain
\begin{align}\label{supp-con}
	&\supp_t (\nabla\rho, m) \subseteq N_{\sum_{q\geq 0}\delta_{q+2}^{1/2}}([T_*,T])\subseteq N_{\ve_*}( [\wt T,T]).
\end{align}

Therefore, the proof of Theorem~\ref{Thm-Nonuniq-hypoNSE} is complete.
\hfill $\square$
\medskip

\paragraph{\bf Proof of Theorem~\ref{Thm-Nonuniq-CNS}}
Without loss of generality, we may choose $\wt \rho=1$
and any non-trival divergence-free
smooth functions $\wt m$ such that
\begin{align*}
	\supp_t \wt m \subseteq [\frac{T}{4}, \frac{3T}{4}].
\end{align*}
For every $j\geq 1$, let
\begin{align} \label{wtmj-wtm-v}
 \wt m_j:=\frac{j}{c_0}\wt m,
\end{align}
where $j \in \mathbb{N}_+$, $c_0:= \|\wt m\|_{L^1(0,T;C_x)} (>0)$.
Thus, $(\wt \rho, \wt m_j)$, $j\geq 1$, are smooth solutions to the transport equation \eqref{equa-trans}.

Let $0<\ve_*<\min\{\frac 18,\frac{T}{8}\}$.
Then, for every $j\geq 1$,
Theorem~\ref{Thm-Nonuniq-hypoNSE} gives
a weak solution $(\rho_j,m_j)$ to \eqref{equa-mns} such that
\begin{align}\label{pf-cor1.3-m}
   \|m_j-\wt m_j\|_{L^1(0,T;C_x)}\leq \ve_*.
\end{align}
and
\begin{align}
	& \int_{\mathbb{T}^d}\rho(t,x)\d x= \int_{\mathbb{T}^d} \wt\rho(t,x)\d x,\ \text{for}\ \forall t\in[0,T],\label{mass-pre}\\
	&\supp_t(\nabla \rho_j, m_j)
    \subseteq [\frac{T}{8}, T].\label{tem-supp}
\end{align}

We note that, by \eqref{tem-supp},
$\rho_j$ is independent of the spatial variable
for $t\in[0,\frac T 8]$.
Then, taking into account \eqref{mass-pre} we derive
\begin{align}
	\rho_j(0)=\aint_{\mathbb{T}^d}\rho_j(0,x)\d x=\aint_{\mathbb{T}^d}\wt \rho(0,x)\d x=1.
\end{align}
Therefore, $(\rho_j,m_j)$, $j\geq 1$, are weak solutions to \eqref{equa-NS}
with the same initial datum $(\rho_j(0), m_j(0))=(1,0)$.

Now we use \eqref{wtmj-wtm-v} and \eqref{pf-cor1.3-m}
to compare the solutions $m_j$ and $m_k$, $j\not = k$,
and derive that
\begin{align*}
	\|m_j-m_{k}\|_{L^1(0,T;C_x)}&\geq\|\wt m_j -\wt m_k\|_{L^1(0,T;C_x)}- \|m_j -\wt m_j\|_{L^1(0,T;C_x)}-\|m_k-\wt m_k\|_{L^2(0,T;C_x)} \\
	&\geq\frac{|j-k|}{c_0}\| \wt m\|_{L^1(0,T;C_x)}- 2\ve_* \\
	& = |j-k|-2\ve_*>\frac12,
\end{align*}
which yields that $m_j\not\equiv m_k$ on $[0,T]$ for every $j \not = k$,
thereby finishing the proof of Theorem~\ref{Thm-Nonuniq-CNS}.
\hfill $\square$

\begin{remark} \label{Rem-proof-INS}
For the hypo-viscous INS \eqref{equa-NSE-Incomp},
due to Remarks~\ref{rem-3.10}, \ref{rq1-incom} and \ref{rem-supp-incom},
Corollary \ref{Cor-Nonuniq-NSE} can be proved in a similar manner
as that of Theorem \ref{Thm-Nonuniq-CNS}.
\end{remark}

\paragraph{\bf Proof of Theorem~\ref{Thm-hypoNSE-Euler-limit}}
Let $\left\{\phi_{\varepsilon}\right\}_{\varepsilon>0}$
and $\left\{\varphi_{\varepsilon}\right\}_{\varepsilon>0}$
be two families of standard compactly supported Friedrichs mollifiers on $\T^{d}$ and $[-T,T]$, respectively.
For every $n\geq 1$, set
\begin{align} \label{un-u-Bn-B}
\rho_{n} :=\left(\rho *_{x} \phi_{\lambda_{n}^{-1}}\right) *_{t} \varphi_{\lambda_{n}^{-1}},\quad	m_{n} :=\left(m *_{x} \phi_{\lambda_{n}^{-1}}\right) *_{t} \varphi_{\lambda_{n}^{-1}},\quad P_{n} :=\left(P(\rho) *_{x} \phi_{\lambda_{n}^{-1}}\right) *_{t} \varphi_{\lambda_{n}^{-1}},
\end{align}
restricted to $[0,T]$,
where $\lbb_n := a^{b^n}$ and $ \delta_{n}:= \lbb_n^{-2\beta}$.

Here, we choose $a\in \mathbb{N}$ sufficiently large and $\beta>0,\ b\in 2\mathbb{N}$ such that
\begin{align} \label{b-beta-ve-n}
	b>\frac{1000}{\varepsilon}, \ \
	0<b^2\beta<\min\left\{\frac{1}{100}, \frac{\wt \beta}{4} \right \}
\end{align}
with $\varepsilon\in \mathbb{Q}_+$ sufficiently small such that
\begin{equation}\label{e3.1-n}
	\varepsilon\leq\frac{1}{20}\min\left\{1-\alpha,\a, \frac{2\a}{p}-\a-s, \frac{\wt \beta}{2(1+\wt \beta)}, 1-\frac{\wt \beta}{2}\right\}\quad \text{and}\quad b\ve\in\mathbb{N}.
\end{equation}

Since $(\rho,m)$ is a weak solution to the compressible Euler equations \eqref{equa-Euler},
we infer that $(\rho_n, m_n)$ satisfies
\begin{equation}\label{equa-moll-euler}
	\left\{\aligned
	& \p_t \rho_n +\div m_n = 0,\\
	&\p_t m_n+ \mu_{n} (-\Delta)^{\a}(\rho_n^{-1}m_n)
     - (\mu_n+\nu_n) \nabla \div (\rho_n^{-1}m_n) +\div(\rho_n^{-1}m_n\otimes m_n) + \nabla P(\rho_{n})=\div R_n.
	\endaligned
	\right.
\end{equation}
where $\mu_{n}:=\lambda_n^{-2}\mu$, $\nu_n= \lambda_n^{-2}\nu$,
and the Reynolds stress
	\begin{align}\label{def-rn}
	R_{n} 	&:=\mathcal{R}\nabla( P(\rho_{n})-P_n ) +\mathcal{R}\lambda_n^{-2}\mu(-\Delta)^{\a} (\rho_n^{-1}m_n) -\mathcal{R}\lambda_n^{- 2}(\mu+\nu) \nabla \div (\rho_n^{-1}m_n)\notag\\
	&\quad  + \rho_n^{-1} m_{n}\otimes m_{n}- (\rho^{-1}m \otimes m ) *_{x} \phi_{\lambda_{n}^{-1}}*_{t} \varphi_{\lambda_{n}^{-1}}.
\end{align}

We claim that for $a$ sufficiently large,
universal for $n$,
$(\rho_n,m_n,R_n)$ satisfy the iterative estimates \eqref{rhobd}-\eqref{rl1} at level $q=n+1$.

To this end, let us first consider the most delicate
$L^1_tC_x$-decay estimate \eqref{rl1} of $R_n$,
for $\ve$ small enough,
\begin{align}\label{est-rn-l1l9}
 \|R_n\|_{L^1_tC_x}
&\lesssim \| P(\rho_{n})-P_n\|_{L^1_{t}C^{\ve}_x} +\| \lambda_n^{-2 }|\nabla|^{2\a-1}(\rho_n^{-1}m_n) \|_{L^1_{t}C^{\ve}_x}  +\|\lambda_n^{-2}\div (\rho_n^{-1}m_n)\|_{L^1_{t}C^{\ve}_x} \notag\\
&\quad  + \|\rho_n^{-1} m_{n}\otimes m_n- (\rho ^{-1}m \otimes m) *_{x} \phi_{\lambda_{n}^{-1}}*_{t} \varphi_{\lambda_{n}^{-1}}\|_{L^1_{t}C^{\ve}_x} \notag \\
    &=: K_1 + K_2 + K_3 + K_4.
\end{align}
Using the standard mollification estimates we deduce that,
\begin{align} \label{rhon-bdd}
	& 0<c_1 \leq \rho_n\leq c_2.
\end{align}
Estimating as in the proof of \eqref{est-rcom1}, by \eqref{rhon-bdd} we deduce
\begin{align}\label{est-rn1}
	K_1
    & \lesssim \|P(\rho_{n})-P_n\|_{C_{t,x}}^{1-\ve}\|P(\rho_{n})-P_n\|_{C_{t}C^{1}_x}^\ve\notag\\
	&\lesssim ( \| P(\rho_{n}) - P(\rho)\|_{C_{t,x}}+ \|P(\rho)-P_n\|_{C_{t,x}})^{1-\ve} ( \| P(\rho_{n}) \|_{C_tC^1_{x}}+\|P_n\|_{C_tC^1_{x}})^\ve\notag\\
	&\lesssim ( \| \rho_{n}- \rho \|_{C_{t,x}}+\lambda_{n}^{-\wt \beta} \|P(\rho)\|_{C^{\wt \beta}_{t,x}})^{1-\ve}  \lambda_{n}^\ve \notag\\
	&\lesssim ( \lambda_{n}^{-\wt \beta} \|\rho \|_{C^{\wt \beta}_{t,x}}+\lambda_{n}^{-\wt \beta} \|\rho \|_{C^{\wt \beta}_{t,x}})^{1-\ve}  \lambda_{n}^\ve  \notag\\
	&\lesssim  \lambda_{n}^{-\wt \beta(1-\ve)+\ve} .
\end{align}
For the $L^1_{t}C^{\ve}_x$-estimate of $K_2$, if $\a\in (0,\frac12]$,
\begin{align}\label{est-rn2.1}
	\| \lambda_n^{-2 }|\nabla|^{2\a-1}(\rho_n^{-1}m_n) \|_{L^1_{t}C^{\ve}_x}
	&\lesssim \lambda_n^{-2 }\|\rho_n^{-1}m_n \|_{L^1_{t}C^{\ve}_x}\notag\\
	&\lesssim \lambda_n^{-2 } \| \rho_n^{-1}m_n \|_{L^1_{t}C_x}^{1-\frac{\ve}{3}}
	\|\rho_n^{-1}m_n \|_{L^1_{t}C^{3}_x}^{\frac{\ve}{3}}\notag\\
	&\lesssim \lambda_n^{-2 }( \sum_{N_1+N_2 = 3}\|\rho_n^{-1} \|_{C_{t}C^{N_1}_x} \| m_n \|_{C_{t}C^{N_2}_x} )^{\frac{\ve}{3}}\notag\\
	&\lesssim \lambda_n^{-2 +\ve},
\end{align}
and if $\a\in (\frac12,1)$, an application of interpolation yields
\begin{align}\label{est-rn2}
 \| \lambda_n^{-2 }|\nabla|^{2\a-1}(\rho_n^{-1}m_n) \|_{L^1_{t}C^{\ve}_x}
 &\lesssim \lambda_n^{-2 } \| \rho_n^{-1}m_n \|_{L^1_{t}C_x}^{1-\frac{2\a-1+\ve}{3}}
\|\rho_n^{-1}m_n \|_{L^1_{t}C^{3}_x}^{\frac{2\a-1+\ve}{3}}\notag\\
&\lesssim \lambda_n^{-2 }( \sum_{N_1+N_2 = 3}\|\rho_n^{-1} \|_{C_{t}C^{N_1}_x} \| m_n \|_{C_{t}C^{N_2}_x} )^{\frac{2\a-1+\ve}{3}}\notag\\
&\lesssim \lambda_n^{-2 }\lambda_n^{2\a-1+\ve} \lesssim \lambda_n^{2\a-3+\ve}.
\end{align}
Thus, we conclude from \eqref{est-rn2.1} and \eqref{est-rn2} that
\begin{align}\label{est-rn2-con}
	K_2 \lesssim \lambda_n^{-2 +\ve}+\lambda_n^{2\a-3+\ve}\lesssim \lambda_n^{-1 +\ve},
\end{align}
and, similarly,
\begin{align}\label{est-rn3}
	K_3 \lesssim \lambda_n^{-1+\ve}.
\end{align}
Moreover, we also have
\begin{align}\label{est-rn4}
	K_4
    \lesssim& 	\|\rho_n^{-1} m_{n}\otimes m_n- (\rho ^{-1}m \otimes m) *_{x} \phi_{\lambda_{n}^{-1}}*_{t} \varphi_{\lambda_{n}^{-1}}\|_{C_{t,x}}^{1-\ve}	\notag \\
    &  \times \|\rho_n^{-1} m_{n}\otimes m_n- (\rho ^{-1}m \otimes m) *_{x} \phi_{\lambda_{n}^{-1}}*_{t} \varphi_{\lambda_{n}^{-1}}\|_{C_{t}C_x^1}^{\ve} \notag \\
    =:& K_{41}^{1-\ve} \times  K^\ve_{42}.
\end{align}
Using the standard mollification estimate we get
\begin{align}\label{est-rn4-1}
 K_{41}
  \lesssim& \|\rho_n^{-1}  m_{n}\otimes( m_{n} -m) \|_{C_{t,x}}+ \|\rho_n^{-1}  (m_{n}-m) \otimes m  \|_{C_{t,x}} + \|(\rho_n^{-1}-\rho^{-1} ) m\otimes m \|_{C_{t,x}} \notag\\
	&\quad+ \| \rho^{-1}m\otimes m -(\rho ^{-1}m \otimes m) *_{x} \phi_{\lambda_{n}^{-1}}*_{t} \varphi_{\lambda_{n}^{-1}} \|_{C_{t,x}} \notag \\
	\lesssim& \| \rho_n^{-1}\|_{C_{t,x}} \|m_{n}-m \|_{C_{t,x}} (\|m_n\|_{C_{t,x}}+\|m \|_{C_{t,x}}  )+  \|\rho_n^{-1}-\rho^{-1}\|_{C_{t,x}}\|m\|_{C_{t,x}}^2
	     + \lambda_{n}^{-\wt \beta} \|  \rho^{-1}m\otimes m  \|_{C^{\wt \beta}_{t,x}}  \notag\\
	\lesssim& \lambda_{n}^{-\wt \beta} \| \rho_n^{-1}\|_{C_{t,x}} \| m \|_{C^{\wt \beta}_{t,x}} \|m \|_{C_{t,x}}
      + \lambda_{n}^{-\wt \beta} \| \rho\|_{C^{\wt \beta}_{t,x}}\|m\|_{C_{t,x}}^2
         + \lambda_{n}^{-\wt \beta}\|  \rho^{-1} \|_{C^{\wt \beta}_{t,x}}\|  m \|_{C^{\wt \beta}_{t,x}}^2 \notag \\
    \lesssim & \lambda_{n}^{-\wt \beta},
\end{align}	
where we also used
\begin{align*}
   \|\rho_n^{-1} - \rho^{-1}\|_{C_{t,x}}
   \lesssim \|\rho_n^{-1}\|_{C_{t,x}}  \|\rho^{-1}\|_{C_{t,x}}  \|\rho_n - \rho\|_{C_{t,x}}
   \lesssim \lbb_n^{-\wt \beta} \|\rho\|_{C_{t,x}^{\wt \beta}}.
\end{align*}
We also have
\begin{align}\label{est-rn4-2}
   K_{42}
  & \lesssim \sum_{N_1+N_2+N_3 =1} \| \rho_n^{-1}\|_{C_{t}C_x^{N_1}} \| m_n \|_{C_{t}C_x^{N_2}} \|m_n \|_{C_{t}C_x^{N_3}} + \lambda_{n} \| \rho ^{-1}m \otimes m\|_{C_{t,x} }\notag\\
&\lesssim \lambda_{n}.
\end{align}
Thus, plugging \eqref{est-rn4-1} and \eqref{est-rn4-2} into \eqref{est-rn4}
we lead to
\begin{align}\label{est-rn4-end}
  K_4 \lesssim  \lambda_{n}^{-\wt \beta(1-\ve)+\ve}.
\end{align}

Therefore, plugging \eqref{est-rn1}, \eqref{est-rn2}, \eqref{est-rn3}
and \eqref{est-rn4-end} into \eqref{est-rn-l1l9}
we arrive at
\begin{align} \label{Ru-wtM-L1}
	\|R_n\|_{L^{1}_{t}C_x}
	 \lesssim \lambda_n^{-1+\ve}+\lambda_{n}^{-\wt \beta(1-\ve)+\ve}\leq \delta_{n+2},
\end{align}
where we also use the fact that $-\wt \beta(1-\ve)+\ve<-\frac{\wt \beta}{2}<-2\beta b^2$ and $-1+\ve<-2\beta b^2 $ in the last step. Thus, the $L^1_tC_x$ inductive estimate \eqref{rl1}
can be verified at level $n+1$.

Now, let us turn to the verification of \eqref{rhobd}-\eqref{rc1}, by \eqref{rhon-bdd} we have
\begin{align}\label{rhon-bdd2}
	\frac{c_1}{2} -\lambda_{n+1}^{-\beta}\leq \rho_n\leq  \frac{c_2}{2}+\lambda_{n+1}^{-\beta}.
\end{align}
Moreover,
we also have that for $1\leq N\leq 4$ and $M=0,1$,
\begin{align}
   \|\partial_t^M \rho_n\|_{C_ 	tC_x^N}
   \lesssim \lbb_n^{M+N} \|\rho\|_{C_{t,x}}
   \lesssim \lbb_n^{1+N}
   \leq \lbb_{n+1}^{\frac {\ve}{4}}.
\end{align}
Hence, \eqref{rhobd} and \eqref{rhoc1} are verified at level $n+1$.

Moreover, we also see that
\begin{align} \label{un-Bn-C1}
	\left\|m_{n}\right\|_{C_{t,x}^N} \lesssim \lambda_{n}^N \left\|m \right\|_{C_{t,x}} \lesssim\lambda_{n}^N
	\ll \lbb_{n+1}^{2N+2},
\end{align}
which yields  \eqref{mc1} at level $n+1$.

Finally, in view of the Sobolev embedding $W^{1,d+2}_{t,x}\hookrightarrow L^\9_{t,x}$, we obtain
\begin{align}\label{rn-c1}
	\|{R}_n\|_{C_{t,x}^1} & \lesssim \| {R}_n\|_{W^{2,d+2}_{t,x}} \notag\\
	& \lesssim  \|P(\rho_{n})\|_{W^{2,d+2}_{t,x}}+\| P(\rho)*_{x} \phi_{\lambda_{n}^{-1}}*_{t} \varphi_{\lambda_{n}^{-1}}\|_{W^{2,d+2}_{t,x}} +\| \lambda_n^{-2\a}|\nabla|^{2\a-1}(\rho_n^{-1}m_n) \|_{W^{2,d+2}_{t,x}}  \notag\\
	&\quad   +\|\lambda_n^{-2}\div (\rho_n^{-1}m_n)\|_{W^{2,d+2}_{t,x}}+ \|\rho_n^{-1} m_{n}\otimes m_n\|_{W^{2,d+2}_{t,x}}+\| (\rho ^{-1}m \otimes m) *_{x} \phi_{\lambda_{n}^{-1}}*_{t} \varphi_{\lambda_{n}^{-1}}\|_{W^{2,d+2}_{t,x}}\notag\\
	& \lesssim  \|P(\rho_{n})\|_{C^{2 }_{t,x}}+\lambda_{n}^{2}\| P(\rho) \|_{C_{t,x}} +\lambda_n^{-2\a}\| |\nabla|^{2\a-1}(\rho_n^{-1}m_n) \|_{W^{2,d+2}_{t,x}} +\lambda_n^{- 2}\| \rho_n^{-1}m_n\|_{C^{3}_{t,x}}  \notag\\
	&\quad  + \|\rho_n^{-1} m_{n}\otimes m_n\|_{C^{2 }_{t,x}}+ \lambda_{n}^{2}\|\rho ^{-1}m \otimes m \|_{C_{t,x}}\notag\\
	& \lesssim  \|P(\rho_{n})\|_{C^{2 }_{t,x}}+\lambda_{n}^{2}\| P(\rho) \|_{C_{t,x}} +\lambda_n^{-2\a}\| |\nabla|^{2\a-1}(\rho_n^{-1}m_n) \|_{W^{2,d+2}_{t,x}} +\lambda_n^{- 2}\sum_{N_1+N_2\leq 3}\| \rho_n^{-1}\|_{C^{N_1}_{t,x}}\|m_n\|_{C^{N_2}_{t,x}}  \notag\\
	&\quad  + \sum_{N_1+N_2+N_3\leq 2}\|\rho_n^{-1}\|_{C^{N_1}_{t,x}}\| m_{n}\|_{C^{N_1}_{t,x}}\|m_n\|_{C^{N_3}_{t,x}}+ \lambda_{n}^{2}\|\rho ^{-1}m \otimes m \|_{C_{t,x}}.
\end{align}
By \eqref{p-bound} and standard mollification estimates, we have
\begin{align}\label{pn-c2}
	\|P(\rho_{n})\|_{C^{2 }_{t,x}} &\lesssim \|P(\rho_{n})\|_{C_{t,x}}+ \|P'(\rho_{n})\|_{C_{t,x}}( \| \rho_{n}\|_{C_tC^2_x}+  \| \rho_{n}\|_{C^2_tC_x}+\| \rho_{n}\|_{C^1_tC^1_x} )\notag\\
	&\quad+ \|P''(\rho_{n})\|_{C_{t,x}}( \| \rho_{n}\|_{C_tC^1_x}^2+  \| \rho_{n}\|_{C^1_tC_x}^2)\notag\\
	&\lesssim \lambda_{n}^2.
\end{align}
If $\a\in (0,\frac12]$, then we have
\begin{align}\label{a12.1}
\| |\nabla|^{2\a-1}(\rho_n^{-1}m_n) \|_{W^{2,d+2}_{t,x}}\lesssim \| \rho_n^{-1}m_n \|_{C^{2}_{t,x}}\lesssim \sum_{N_1+N_2 \leq 2}\|\rho_n^{-1}\|_{C^{N_1}_{t,x}}\| m_{n}\|_{C^{N_2}_{t,x}} \lesssim \sum_{N_1+N_2 \leq 2}\lambda_{n}^{N_1+N_2}\lesssim \lambda_{n}^{2},
\end{align}
and if $\a\in (\frac12,1)$, an application of the interpolation inequality gives
\begin{align}\label{e5.31}
\| |\nabla|^{2\a-1}(\rho_n^{-1}m_n) \|_{W^{2,d+2}_{t,x}} &\lesssim \| \rho_n^{-1}m_n \|_{C^{2}_{t,x}}^{2-2\a}\| \rho_n^{-1}m_n \|_{C^{3}_{t,x}}^{2\a-1}\notag\\
	&\lesssim
	( \sum_{N_1+N_2 \leq 2}\|\rho_n^{-1}\|_{C^{N_1}_{t,x}}\| m_{n}\|_{C^{N_2}_{t,x}})^{2-2\a}
     ( \sum_{N_1+N_2 \leq 3}\|\rho_n^{-1}\|_{C^{N_1}_{t,x}}\| m_{n}\|_{C^{N_2}_{t,x}})^{2\a-1}\notag\\
	&\lesssim ( \sum_{N_1+N_2 \leq 2}\lambda_{n}^{N_1+N_2} )^{2-2\a}( \sum_{N_1+N_2 \leq 3}\lambda_{n}^{N_1+N_2} )^{2\a-1}\lesssim \lambda_{n}^{2\a+1}.
\end{align}
Plugging \eqref{pn-c2}-\eqref{e5.31} into \eqref{rn-c1}, we deduce that
\begin{align}
	\|{R}_n\|_{C_{t,x}^1}
	& \lesssim \lambda_{n}^2+\lambda_{n}^2 +\lambda_n^{2-2\alpha}+\lambda_n+\lambda_n^2+\lambda_n^2
      \lesssim \lambda_n^2.
\end{align}
Hence, taking $a$ sufficiently large, we verify the inductive estimate \eqref{rc1}  at level $n+1$.

Thus, we can apply Theorem~\ref{Thm-Iterat} to obtain a sequence of relaxed solutions
$\{(\rho_{n,q}, m_{n,q})\}_{q\geq 0}$ to \eqref{equa-moll-euler},
and then letting $q\rightarrow\infty$
we obtain a weak solution $(\rho^{( {n})},m^{( {n})} )\in C_{t,x}\times H^{\beta'}_tC_x$
to \eqref{equa-nsr}
for some $0<\beta'<\min\{\wt \beta, \beta/(8+\beta)\}$,
where $\beta$ is as in the proof of Theorem \ref{Thm-Iterat}.

Then, taking $\beta'$ sufficiently small such that
$0<\beta'<\min\{\wt \beta, \beta/(8+\beta)\}$
we deduce from \eqref{mc1} and \eqref{m-L2t-conv} that
\begin{align*}
	\|m^{( {n})}-m\|_{H^{\beta'}_tC_x}
	&\leq\|m^{( {n})}-m_{n}\|_{H^{\beta'}_tC_x}
	+\|m-m_{n}\|_{H^{\beta'}_tC_x} \notag\\
	&\lesssim \sum_{q=n+1}^{\infty} \| m_{n,q+1}-m_{n,q} \|_{L^2_tC_x}^{1-\beta'}
    \| m_{n,q+1}-m_{n,q} \|_{C^1_{t,x}}^{\beta'} + \|m-m_{n}\|_{C^{\beta'}_tC_x}  \notag\\
	&\lesssim \sum_{q=n+1}^{\infty}  \lambda_{q+1}^{-\beta(1-\beta^{\prime})} \lambda_{q+1}^{8 \beta^{\prime}}
          + \lbb_n^{-(\wt \beta- \beta')}\|m \|_{C^{\wt \beta}_{t,x}} \notag\\
	&\lesssim \frac{a^{nb(-\beta(1-\beta') + 8\beta)}}{a^{nb(\beta(1-\beta') - 8\beta)} -1}
              + a^{-(\wt \beta - \beta') b n}
	\leq \frac{1}{n},
\end{align*}
and, via \eqref{rho-l9-conv},
\begin{align*}
	\|\rho^{( {n})}-\rho\|_{C_{t,x}}
	&\leq\|\rho^{( {n})}-\rho_{n}\|_{C_{t,x}}
	+\|\rho-\rho_{n}\|_{C_{t,x}} \notag\\
	&\lesssim \sum_{q=n+1}^{\infty} \|\rho_{n,q+1}-\rho_{n,q}\|_{C_{t,x}}+  \|\rho-\rho_{n}\|_{ C_{t,x}}   \notag\\
	&\lesssim \sum_{q=n+1}^{\infty} \delta_{q+2}^{\frac12}+ \lbb_n^{-\wt \beta} \|\rho\|_{C^{\wt \beta}_{t,x}}   \notag \\
    &\lesssim \sum_{q=n+3}^\infty a^{-\beta b q} + \lbb_n^{-\wt \beta} \lesssim \frac{a^{-\beta b(n+2)}}{a^{\beta b}-1}  + \lbb_n^{-\wt \beta}
	\leq \frac{1}{n},
\end{align*}
where the last step is valid for $a$ sufficiently large.

Therefore, letting $n\to \infty$
we obtain the strong convergence \eqref{convergence}
and finish the proof of Theorem \ref{Thm-hypoNSE-Euler-limit}.
\hfill $\square$
\medskip

\noindent{\bf Acknowledgment.}
Y. Li thanks the support by NSFC (No. 11831011, 12161141004).
P. Qu thanks the supports by  NSFC (No. 12122104, 11831011)
and Shanghai Science and Technology Programs 21ZR1406000, 21JC1400600, 19JC1420101.
Z. Zeng and D. Zhang thank the supports by NSFC (No. 12271352).
D. Zhang also thanks supports by NSFC (No. 12161141004)
and Shanghai Rising-Star Program 21QA1404500.
Y. Li and D. Zhang are also grateful for the supports by
Institute of Modern Analysis--A Shanghai Frontier Research Center.\\

\noindent{\bf Conflict of interest.} The authors declare that they have no conflict of interest.\ \\

\noindent{\bf Data availability statement.} Data sharing is not applicable to this article as no datasets were generated or analysed during the current study.

\end{document}